\documentclass[11pt]{preprint}
\usepackage[full]{textcomp}
\usepackage[osf]{newtxtext} 
\usepackage[cal=boondoxo]{mathalfa}
\usepackage{colortbl}
\usepackage{comment}

\usepackage{amssymb}
\usepackage{mathtools}
\newcommand{\brsq}[1]{\left[#1\right]}
\usepackage{hyperref}
\usepackage{mhenvs}
\usepackage{mhequ} 
\usepackage{mhsymb}
\usepackage{booktabs}
\usepackage{tikz,tikz-cd}
\usepackage{tcolorbox}
\usepackage{mathrsfs}
\usepackage[utf8]{inputenc}
\usepackage{longtable}
\usepackage{wrapfig}
\usepackage{microtype}
\usepackage{comment}
\usepackage{wasysym}
\usepackage{centernot}
\usepackage{enumitem}
\usepackage{bm}
\usepackage{stackrel}
\usepackage{graphicx}

\makeatletter
\newcommand{\globalcolor}[1]{%
  \color{#1}\global\let\default@color\current@color
}
\makeatother

\usetikzlibrary{calc}
\usetikzlibrary{decorations}
\usetikzlibrary{positioning}
\usetikzlibrary{shapes}
\usetikzlibrary{external}

\usepackage[backend=bibtex,
    sorting=anyt,
    isbn=false,
    url=false,
    doi=false,
    maxbibnames=100
    ]{biblatex}
\newif\ifdark
\darkfalse

\ifdark

\definecolor{darkred}{rgb}{0.9,0.2,0.2}
\definecolor{darkblue}{rgb}{0.7,0.3,1}
\definecolor{darkgreen}{rgb}{0.1,0.9,0.1}
\definecolor{franck}{rgb}{0,0.8,1}
\definecolor{pagebackground}{rgb}{.15,.21,.18}
\definecolor{pageforeground}{rgb}{.84,.84,.85}
\pagecolor{pagebackground}
\AtBeginDocument{\globalcolor{pageforeground}}
\definecolor{symbols}{rgb}{0,0.7,1}
\colorlet{connection}{red!80!black}
\colorlet{boxcolor}{blue!50}

\else

\definecolor{darkred}{rgb}{0.7,0.1,0.1}
\definecolor{darkblue}{rgb}{0.4,0.1,0.8}
\definecolor{darkgreen}{rgb}{0.1,0.7,0.1}
\definecolor{franck}{rgb}{0,0,1}
\definecolor{pagebackground}{rgb}{1,1,1}
\definecolor{pageforeground}{rgb}{0,0,0}
\colorlet{symbols}{blue!90!black}
\colorlet{connection}{red!30!black}
\colorlet{boxcolor}{blue!50!black}

\fi

\def\slash{\leavevmode\unskip\kern0.18em/\penalty\exhyphenpenalty\kern0.18em}
\def\dash{\leavevmode\unskip\kern0.18em--\penalty\exhyphenpenalty\kern0.18em}

\DeclareMathAlphabet{\mathbbm}{U}{bbm}{m}{n}

\DeclareFontFamily{U}{BOONDOX-calo}{\skewchar\font=45 }
\DeclareFontShape{U}{BOONDOX-calo}{m}{n}{
  <-> s*[1.05] BOONDOX-r-calo}{}
\DeclareFontShape{U}{BOONDOX-calo}{b}{n}{
  <-> s*[1.05] BOONDOX-b-calo}{}
\DeclareMathAlphabet{\mcb}{U}{BOONDOX-calo}{m}{n}
\SetMathAlphabet{\mcb}{bold}{U}{BOONDOX-calo}{b}{n}

\setlist{noitemsep,topsep=4pt,leftmargin=1.5em}

\DeclareMathAlphabet{\mathbbm}{U}{bbm}{m}{n}

\DeclareMathAlphabet{\mcb}{U}{BOONDOX-calo}{m}{n}
\SetMathAlphabet{\mcb}{bold}{U}{BOONDOX-calo}{b}{n}
\DeclareFontFamily{U}{mathx}{\hyphenchar\font45}
\DeclareFontShape{U}{mathx}{m}{n}{
      <5> <6> <7> <8> <9> <10>
      <10.95> <12> <14.4> <17.28> <20.74> <24.88>
      mathx10
      }{}
\DeclareSymbolFont{mathx}{U}{mathx}{m}{n}
\DeclareMathSymbol{\bigtimes}{1}{mathx}{"91}

\providecommand{\figures}{false}
{ \ifthenelse{\equal{\figures}{false}} {#1}{\[ {\rm Figure \ missing !} \]} }{}

\let\graftI\curvearrowright
\def\graftID{\mathrel{\textcolor{connection}{\boldsymbol{\curvearrowright}}}}

\def\CP{\mathcal{P}}

\def\CC{\mathcal{C}}

\tikzstyle{tinydots}=[dash pattern=on \pgflinewidth off \pgflinewidth]
\tikzstyle{superdense}=[dash pattern=on 4pt off 1pt]

\def\eps{\varepsilon}

\DeclareMathAlphabet{\pazocal}{OMS}{zplm}{m}{n}

\def\calA{\pazocal{A}}

\def\calF{\pazocal{F}}

\def\calH{\pazocal{H}}

\def\calO{\pazocal{O}}
\def\calP{\pazocal{P}}

\def\calS{\pazocal{S}}

\def\${|\!|\!|}

\def\SS{\mathfrak{S}}

\def\VV{\mathfrak{V}}

\newenvironment{DIFnomarkup}{}{} 

\theorembodyfont{\rmfamily}

\newfont{\indic}{bbmss12}

\def\Nabla_#1{\nabla_{\!#1}}

\makeatletter
\pgfdeclareshape{crosscircle}
{
  \inheritsavedanchors[from=circle] 
  \inheritanchorborder[from=circle]
  \inheritanchor[from=circle]{north}
  \inheritanchor[from=circle]{north west}
  \inheritanchor[from=circle]{north east}
  \inheritanchor[from=circle]{center}
  \inheritanchor[from=circle]{west}
  \inheritanchor[from=circle]{east}
  \inheritanchor[from=circle]{mid}
  \inheritanchor[from=circle]{mid west}
  \inheritanchor[from=circle]{mid east}
  \inheritanchor[from=circle]{base}
  \inheritanchor[from=circle]{base west}
  \inheritanchor[from=circle]{base east}
  \inheritanchor[from=circle]{south}
  \inheritanchor[from=circle]{south west}
  \inheritanchor[from=circle]{south east}
  \inheritbackgroundpath[from=circle]
  \foregroundpath{
    \centerpoint%
    \pgf@xc=\pgf@x%
    \pgf@yc=\pgf@y%
    \pgfutil@tempdima=\radius%
    \pgfmathsetlength{\pgf@xb}{\pgfkeysvalueof{/pgf/outer xsep}}%
    \pgfmathsetlength{\pgf@yb}{\pgfkeysvalueof{/pgf/outer ysep}}%
    \ifdim\pgf@xb<\pgf@yb%
      \advance\pgfutil@tempdima by-\pgf@yb%
    \else%
      \advance\pgfutil@tempdima by-\pgf@xb%
    \fi%
    \pgfpathmoveto{\pgfpointadd{\pgfqpoint{\pgf@xc}{\pgf@yc}}{\pgfqpoint{-0.707107\pgfutil@tempdima}{-0.707107\pgfutil@tempdima}}}
    \pgfpathlineto{\pgfpointadd{\pgfqpoint{\pgf@xc}{\pgf@yc}}{\pgfqpoint{0.707107\pgfutil@tempdima}{0.707107\pgfutil@tempdima}}}
    \pgfpathmoveto{\pgfpointadd{\pgfqpoint{\pgf@xc}{\pgf@yc}}{\pgfqpoint{-0.707107\pgfutil@tempdima}{0.707107\pgfutil@tempdima}}}
    \pgfpathlineto{\pgfpointadd{\pgfqpoint{\pgf@xc}{\pgf@yc}}{\pgfqpoint{0.707107\pgfutil@tempdima}{-0.707107\pgfutil@tempdima}}}
  }
}
\makeatother

\def\symbol#1{\textcolor{symbols}{#1}}

\def\decorate#1#2{
        \ifnum#2>0
    		\foreach \count in {1,...,#2}{
	       	let
				\p1 = (sourcenode.center),
                \p2 = (sourcenode.east),
				\n1 = {\x2-\x1},
				\n2 = {1mm},
				\n3 = {(1.3+0.6*(\count-1))*\n1},
				\n4 = {0.7*\n1}
			in 
        		node[rectangle,fill=symbols,rotate=30,inner sep=0pt,minimum width=0.2*\n2,minimum height=\n2] at ($(sourcenode.center) + (\n3,\n4)$) {}
				}
		\fi
        \ifnum#1>0
    		\foreach \count in {1,...,#1}{
	       	let
				\p1 = (sourcenode.center),
                \p2 = (sourcenode.east),
				\n1 = {\x2-\x1},
				\n2 = {1mm},
				\n3 = {(1.3+0.6*(\count-1))*\n1},
				\n4 = {0.7*\n1}
			in 
        		node[rectangle,fill=symbols,rotate=-30,inner sep=0pt,minimum width=0.2*\n2,minimum height=\n2] at ($(sourcenode.center) + (-\n3,\n4)$) {}
				}
		\fi
}

\tikzset{
    dectriangle/.style 2 args={
        triangle,
        alias=sourcenode,
        append after command={\decorate{#1}{#2}}
    },
    dectriangle/.default={0}{0},
}

\tikzset{
	cross/.style={path picture={ 
  		\draw[symbols]
			(path picture bounding box.south east) -- (path picture bounding box.north west) (path picture bounding box.south west) -- (path picture bounding box.north east);
		}},
root/.style={circle,fill=green!50!black,inner sep=0pt, minimum size=1.2mm},
        dot/.style={circle,fill=pageforeground,inner sep=0pt, minimum size=1mm},
        blank/.style={circle,fill=white,inner sep=0pt, minimum size=1mm},
        dotred/.style={circle,fill=pageforeground!50!pagebackground,inner sep=0pt, minimum size=2mm},
        var/.style={circle,fill=pageforeground!10!pagebackground,draw=pageforeground,inner sep=0pt, minimum size=3mm},
        var1/.style={circle,fill=pagebackground,draw=pageforeground,inner sep=0pt, minimum size=5mm},
         var2/.style={circle,fill=pagebackground,draw=pageforeground,inner sep=0pt, minimum size=2.5mm},
        sqvar/.style={rectangle,fill=pageforeground!10!pagebackground,draw=pageforeground,inner sep=0pt, minimum size=3mm},
        kernel/.style={semithick,shorten >=2pt,shorten <=2pt},
        kernels/.style={snake=zigzag,shorten >=2pt,shorten <=2pt,segment amplitude=1pt,segment length=4pt,line before snake=2pt,line after snake=5pt,},
        rho/.style={densely dashed,semithick,shorten >=2pt,shorten <=2pt},
           testfcn/.style={dotted,semithick,shorten >=2pt,shorten <=2pt},
        renorm/.style={shape=circle,fill=pagebackground,inner sep=1pt},
        labl/.style={shape=rectangle,fill=pagebackground,inner sep=1pt},
        xic/.style={very thin,circle,draw=symbols,fill=symbols,inner sep=0pt,minimum size=1.2mm},
        g/.style={very thin,rectangle,draw=symbols,fill=symbols!10!pagebackground,inner sep=0pt,minimum width=2.5mm,minimum height=1.2mm},
        xi/.style={very thin,circle,draw=symbols,fill=symbols!10!pagebackground,inner sep=0pt,minimum size=1.2mm},
	xies/.style={very thin,rectangle,fill=green!50!black!25,draw=symbols,inner sep=0pt,minimum size=1.1mm},
	xiesf/.style={very thin,rectangle,fill=green!50!black,draw=symbols,inner sep=0pt,minimum size=1.1mm},
        xix/.style={very thin,crosscircle,fill=symbols!10!pagebackground,draw=symbols,inner sep=0pt,minimum size=1.2mm},
        X/.style={very thin,cross,rectangle,fill=pagebackground,draw=symbols,inner sep=0pt,minimum size=1.2mm},
	xib/.style={thin,circle,fill=symbols!10!pagebackground,draw=symbols,inner sep=0pt,minimum size=1.6mm},
	xie/.style={thin,circle,fill=green!50!black,draw=symbols,inner sep=0pt,minimum size=1.6mm},
	xid/.style={thin,circle,fill=symbols,draw=symbols,inner sep=0pt,minimum size=1.6mm},
	xibx/.style={thin,crosscircle,fill=symbols!10!pagebackground,draw=symbols,inner sep=0pt,minimum size=1.6mm},
	kernels2/.style={very thick,draw=connection,segment length=12pt},
	keps/.style={thin,draw=symbols,->},
	kepspr/.style={thick,draw=connection,->},
	krho/.style={thin,draw=symbols,superdense,->},
	krhopr/.style={thick,draw=connection,superdense,->},
	triangle/.style = { regular polygon, regular polygon sides=3},
	not/.style={thin,circle,draw=connection,fill=connection,inner sep=0pt,minimum size=0.5mm},
	diff/.style = {very thin,draw=symbols,triangle,fill=red!50!black,inner sep=0pt,minimum size=1.6mm},
	diff1/.style = {very thin,dectriangle={1}{0},fill=red!50!black,draw=symbols,inner sep=0pt,minimum size=1.6mm},
	diff2/.style = {very thin,dectriangle={1}{1},fill=red!50!black,draw=symbols,inner sep=0pt,minimum size=1.6mm},
		diffmini/.style = {very thin,rectangle,fill=black,draw=black,inner sep=0pt,minimum size=0.75mm},
	 kernelsmod/.style={very thick,draw=connection,segment length=12pt},
	 rec/.style = {very thin,rectangle,fill=black,draw=black,inner sep=0pt,minimum size=2mm},
	cerc/.style={very thin,circle,draw=black,fill=symbols,inner sep=0pt,minimum size=2mm},
	stars/.style={very thin,star,star points=6,star point ratio=0.5, draw=black,fill=red,inner sep=0pt,minimum size=0.7mm},
	>=stealth,
        }

\makeatletter
\def\DeclareSymbol#1#2#3{%
	\expandafter\gdef\csname MH@symb@#1\endcsname{\tikzsetnextfilename{symbol#1}%
	\tikz[baseline=#2,scale=0.15,draw=symbols,line join=round]{#3}}%
	\expandafter\gdef\csname MH@symb@#1s\endcsname{\scalebox{0.75}{\tikzsetnextfilename{symbol#1}%
	\tikz[baseline=#2,scale=0.15,draw=symbols,line join=round]{#3}}}%
	\expandafter\gdef\csname MH@symb@#1ss\endcsname{\scalebox{0.65}{\tikzsetnextfilename{symbol#1}%
	\tikz[baseline=#2,scale=0.15,draw=symbols,line join=round]{#3}}}%
	}
\def\<#1>{\ifthenelse{\boolean{mmode}}{\mathchoice{\csname MH@symb@#1\endcsname}{\csname MH@symb@#1\endcsname}{\csname MH@symb@#1s\endcsname}{\csname MH@symb@#1ss\endcsname}}{\csname MH@symb@#1\endcsname}}
\makeatother

\DeclareSymbol{Xi22}{0.5}{\draw (0,0) node[xi] {} -- (-1,1) node[not] {} -- (0,2) node[xi] {};} 
\DeclareSymbol{Xi2}{-2}{\draw (-1,-0.25) node[xi] {} -- (0,1) node[xi] {};} 
\DeclareSymbol{Xi2C}{-2}{\draw (-1,-0.25) node[xi] {} -- (0,1) node[xi] {}; \node at (-1,1) (a) {\scriptsize 1};}
\DeclareSymbol{Xi2alpha}{-2}{\draw (-1,-0.25) node[xi] {} -- (0,1) node[xi] {}; \node at (-1.2,1) (a) {\scriptsize $\alpha$};}
\DeclareSymbol{I(Xi)}{0}{\node[xi] at (0,1) (a) {}; \draw (a) -- (0,0);}
\DeclareSymbol{pXi2}{-2}{\draw (-1,-0.25) node[xi] {} -- (0,1) node[xi] {}; \node at (-0.95,0.95) {\tiny $1$};}
\DeclareSymbol{puXi2}{-2}{\draw (-1,-0.25) node[xi] {} -- (0,1) node[xi] {}; \node at (-1.25,0.95) {\tiny $m$};}
\DeclareSymbol{PlantedNoise}{0}{\draw[black] (0,1.5) node[xi] {} -- (0,-0.25) {};}
\DeclareSymbol{PlantedNoiseC}{0}{\draw[black] (0,1.5) node[xi] {} -- (0,-0.25) {}; \node at (-0.75,0.5) (a) {\scriptsize 1};}
\DeclareSymbol{PlantedNoiseCC}{0}{\draw[black] (0,1.5) node[xi] {} -- (0,-0.25) {}; \node at (-0.75,0.5) (a) {\scriptsize 2};}
\DeclareSymbol{PlantedNoiseCX}{0}{\draw[kernels2] (0,1.5) node[xi] {} -- (0,-0.25) {}; \node at (-0.75,0.5) (a) {\scriptsize 1};}
\DeclareSymbol{PlantedNoiseX}{0}{\draw[kernels2] (0,1.5) node[xi] {} -- (0,-0.25) {};}
\DeclareSymbol{PlantedNoisea}{0}{\draw (0,1.5) node[xi] {} -- (0,-0.25) {}; \node at (-0.75,0.5) (a) {\scriptsize $\alpha$};}
\DeclareSymbol{XiIt}{2}{\node[xi] at (-1.5,-0.1) (a) {};
\node[var] at (0,2.25) (b) {\tiny{$\tau$ }};
\node[blank] at (-1.5,1.25) {\tiny{$k$}};
\draw[symbols]  (a) -- (b);}
\DeclareSymbol{Xi1It}{2}{\node[xi] at (-1.5,-0.1) (a) {};
\node[var] at (0,2.25) (b) {\tiny{$\tau$ }};
\node[blank] at (-1.5,1.25) {\tiny{$1$}};
\draw[symbols]  (a) -- (b);}
\DeclareSymbol{Xi0It}{2}{\node[xi] at (-1.5,-0.1) (a) {};
\node[var] at (0,2.25) (b) {\tiny{$\tau$ }};
\node[blank] at (-1.5,1.25) {};
\draw[symbols]  (a) -- (b);}
\DeclareSymbol{It}{2}{\node at (0,-1) (a) {};
\node[var] at (0,2.25) (b) {\tiny{$\tau$ }};
\node[blank] at (-0.75,0.5) {\tiny{$k$}};
\draw[symbols]  (a) -- (b);}
\DeclareSymbol{It0}{2}{\node at (0,-1) (a) {};
\node[var] at (0,2.25) (b) {\tiny{$\tau$ }};
\draw[symbols]  (a) -- (b);}
\DeclareSymbol{1It}{2}{\node at (0,-1) (a) {};
\node[var] at (0,2.25) (b) {\tiny{$\tau$ }};
\node[blank] at (-0.75,0.5) {\tiny{$1$}};
\draw[symbols]  (a) -- (b);}
\DeclareSymbol{1I1t1}{2}{\node at (0,-1) (a) {};
\node[var] at (0,2.25) (b) {\tiny{$\tau_1$ }};
\node[blank] at (-0.75,0.5) {\tiny{$1$}};
\draw[kernels2]  (a) -- (b);}
\DeclareSymbol{1I1t2}{2}{\node at (0,-1) (a) {};
\node[var] at (0,2.25) (b) {\tiny{$\tau_2$ }};
\node[blank] at (-0.75,0.5) {\tiny{$1$}};
\draw[kernels2]  (a) -- (b);}
\DeclareSymbol{1I1t}{2}{\node at (0,-1) (a) {};
\node[var] at (0,2.25) (b) {\tiny{$\tau$ }};
\node[blank] at (-0.75,0.5) {\tiny{$1$}};
\draw[kernels2]  (a) -- (b);}
\DeclareSymbol{I1t}{2}{\node at (0,-1) (a) {};
\node[var] at (0,2.25) (b) {\tiny{$\tau$ }};
\node[blank] at (-0.75,0.5) {\tiny{$1$}};
\draw[kernels]  (a) -- (b);}
\DeclareSymbol{I1t1}{2}{\node at (0,-1) (a) {};
\node[var] at (0,2.25) (b) {\tiny{$\tau_1$}};
\node[blank] at (-0.75,0.5) {};
\draw[kernels2]  (a) -- (b);}
\DeclareSymbol{Ita}{2}{\node at (0,-1) (a) {};
\node[var] at (0,2.25) (b) {\tiny{$\tau$ }};
\node[blank] at (-0.75,0.5) {\tiny{$\alpha$}};
\draw[symbols]  (a) -- (b);}
\DeclareSymbol{I1t1I1t2}{0}{\node[var] at (-2,2) (a) {\tiny{$\tau_1$}}; 
\node[var] at (2,2) (b) {\tiny{$\tau_2$}}; 
\draw[kernels2] (0,0) -- (a); 
\draw[kernels2] (0,0) -- (b); 
\node at (-1.5,-0.05) {{\tiny $\ell$}}; 
\node at (1.5,-0.25) {{\tiny $m$}};}
\DeclareSymbol{0I1t10I1t2}{0}{\node[var] at (-2,2) (a) {\tiny{$\tau_1$}}; 
\node[var] at (2,2) (b) {\tiny{$\tau_2$}}; 
\draw[kernels2] (0,0) -- (a); 
\draw[kernels2] (0,0) -- (b);}
\DeclareSymbol{0I1t11I1t2}{0}{\node[var] at (-2,2) (a) {\tiny{$\tau_1$}}; 
\node[var] at (2,2) (b) {\tiny{$\tau_2$}}; 
\draw[kernels2] (0,0) -- (a); 
\draw[kernels2] (0,0) -- (b); 
\node at (-1.5,-0.05) {{\tiny $0$}}; 
\node at (1.5,-0.05) {{\tiny $1$}};}
\DeclareSymbol{I1XiItI1Xi}{0}{\node[xi] at (-2,2) (a) {}; 
\node[xi] at (2,2) (b) {}; 
\node[var] at (0,3) (c) {\tiny{$\tau$}};
\draw[kernels2] (0,0) -- (a); 
\draw[kernels2] (0,0) -- (b);
\draw[] (0,0) -- (c);
\node at (-1.5,-0.05) {{\tiny $\ell$}}; 
\node at (1.5,-0.25) {{\tiny $m$}};}
\DeclareSymbol{I1t1I1t2It3}{0}{\node[var] at (-3,2) (a) {\tiny{$\tau_1$}}; 
\node[var] at (3,2) (b) {\tiny{$\tau_3$}}; 
\node[var] at (0,4) (c) {\tiny{$\tau_2$}};
\draw[kernels2] (0,0) -- (a); 
\draw (0,0) -- (b);
\draw[kernels2] (0,0) -- (c);
\node at (-1.5,-0.05) {{\tiny $\ell$}};
\node at (-1,2) {\tiny{$m$}}; 
\node at (1.5,-0.25) {{\tiny $n$}};}
\DeclareSymbol{XiIt1It2It3}{0}{\node[var] at (-3,2) (a) {\tiny{$\tau_1$}}; 
\node[var] at (3,2) (b) {\tiny{$\tau_3$}}; 
\node[var] at (0,4) (c) {\tiny{$\tau_2$}};
\node[xi] at (0,0) (d) {};
\draw (d) -- (a); 
\draw (d) -- (b);
\draw (d) -- (c);
\node at (-1.5,-0.05) {{\tiny $\ell$}};
\node at (-1,2) {\tiny{$m$}}; 
\node at (1.5,-0.25) {{\tiny $n$}};}
\DeclareSymbol{It1It2It3}{0}{\node[var] at (-3,2) (a) {\tiny{$\tau_1$}}; 
\node[var] at (3,2) (b) {\tiny{$\tau_3$}}; 
\node[var] at (0,4) (c) {\tiny{$\tau_2$}};
\draw (0,0) -- (a); 
\draw (0,0) -- (b);
\draw (0,0) -- (c);
\node at (-1.5,-0.05) {{\tiny $\ell$}};
\node at (-1,2) {\tiny{$m$}}; 
\node at (1.5,-0.25) {{\tiny $n$}};}
\DeclareSymbol{It1It2I1t3}{0}{\node[var] at (-3,2) (a) {\tiny{$\tau_1$}}; 
\node[var] at (3,2) (b) {\tiny{$\tau_3$}}; 
\node[var] at (0,4) (c) {\tiny{$\tau_2$}};
\draw (0,0) -- (a); 
\draw[kernels2] (0,0) -- (b);
\draw (0,0) -- (c);
\node at (-1.5,-0.05) {{\tiny $\ell$}};
\node at (-1,2) {\tiny{$m$}}; 
\node at (1.5,-0.25) {{\tiny $n$}};}
\DeclareSymbol{It1I1t2It3}{0}{\node[var] at (-3,2) (a) {\tiny{$\tau_1$}}; 
\node[var] at (3,2) (b) {\tiny{$\tau_3$}}; 
\node[var] at (0,4) (c) {\tiny{$\tau_2$}};
\draw (0,0) -- (a); 
\draw (0,0) -- (b);
\draw[kernels2] (0,0) -- (c);
\node at (-1.5,-0.05) {{\tiny $\ell$}};
\node at (-1,2) {\tiny{$m$}}; 
\node at (1.5,-0.25) {{\tiny $n$}};}
\DeclareSymbol{It1It2It3It4}{0}{\node[var] at (-4,2) (a) {\tiny{$\tau_1$}}; 
\node[var] at (4,2) (b) {\tiny{$\tau_3$}}; 
\node[var] at (-1.5,4) (c) {\tiny{$\tau_2$}};
\node[var] at (1.5,4) (d) {\tiny{$\tau_2$}};
\draw (0,0) -- (a); 
\draw (0,0) -- (b);
\draw (0,0) -- (c);
\draw (0,0) -- (d);
\node at (-1.5,-0.05) {{\tiny $k$}};
\node at (-1.5,1.75) {\tiny{$\ell$}}; 
\node at (1.5,-0.25) {{\tiny $m$}};
\node at (1.5, 1.75) {{\tiny $n$}};}
\DeclareSymbol{It1I1t2It3It4}{0}{\node[var] at (-4,2) (a) {\tiny{$\tau_1$}}; 
\node[var] at (4,2) (b) {\tiny{$\tau_3$}}; 
\node[var] at (-1.5,4) (c) {\tiny{$\tau_2$}};
\node[var] at (1.5,4) (d) {\tiny{$\tau_2$}};
\draw (0,0) -- (a); 
\draw (0,0) -- (b);
\draw[kernels2] (0,0) -- (c);
\draw (0,0) -- (d);
\node at (-1.5,-0.05) {{\tiny $k$}};
\node at (-1.5,1.75) {\tiny{$\ell$}}; 
\node at (1.5,-0.25) {{\tiny $m$}};
\node at (1.5, 1.75) {{\tiny $n$}};}
\DeclareSymbol{It1I1t2I1t3}{0}{\node[var] at (-3,2) (a) {\tiny{$\tau_1$}}; 
\node[var] at (3,2) (b) {\tiny{$\tau_3$}}; 
\node[var] at (0,4) (c) {\tiny{$\tau_2$}};
\draw (0,0) -- (a); 
\draw[kernels2] (0,0) -- (b);
\draw[kernels2] (0,0) -- (c);
\node at (-1.5,-0.05) {{\tiny $\ell$}};
\node at (-1,2) {\tiny{$m$}}; 
\node at (1.5,-0.25) {{\tiny $n$}};}
\DeclareSymbol{It1It2I1t3I1t4}{0}{\node[var] at (-4,2) (a) {\tiny{$\tau_1$}}; 
\node[var] at (4,2) (b) {\tiny{$\tau_4$}}; 
\node[var] at (2,4) (c) {\tiny{$\tau_3$}};
\node[var] at (-2,4) (d) {\tiny{$\tau_2$}};
\draw (0,0) -- (a); 
\draw[kernels2] (0,0) -- (b);
\draw[kernels2] (0,0) -- (c);
\draw (0,0) -- (d);
\node at (-1.75,-0.1) {{\tiny $k$}};
\node at (-2.15,2) {\tiny{$\ell$}};
\node at (2.15,2) {\tiny{$m$}}; 
\node at (1.75,-0.3) {{\tiny $n$}};}
\DeclareSymbol{XiItIXi}{0}{\node[var] at (-2,2) (a) {\tiny{$\tau$}}; 
\node[xi] at (2,2) (b) {}; 
\draw (0,0) -- (a); 
\draw (0,0) -- (b); 
\node at (-1.5,-0.05) {{\tiny $\ell$}}; 
\node at (1.5,-0.25) {{\tiny $m$}};
\node[xi] at (0,0) (c) {};}
\DeclareSymbol{It1It2}{0}{\node[var] at (-2,2) (a) {\tiny{$\tau_1$}}; 
\node[var] at (2,2) (b) {\tiny{$\tau_2$}}; 
\draw (0,0) -- (a); 
\draw (0,0) -- (b); 
\node at (-1.5,-0.05) {{\tiny $\ell$}}; 
\node at (1.5,-0.25) {{\tiny $m$}};}
\DeclareSymbol{It1I1t2}{0}{\node[var] at (-2,2) (a) {\tiny{$\tau_1$}}; 
\node[var] at (2,2) (b) {\tiny{$\tau_2$}}; 
\draw (0,0) -- (a); 
\draw[kernels2] (0,0) -- (b); 
\node at (-1.5,-0.05) {{\tiny $\ell$}}; 
\node at (1.5,-0.25) {{\tiny $m$}};}
\DeclareSymbol{Xi2b}{-2}{\draw (-1,-0.25) node[xic] {} -- (0,1) node[xic] {};} 
\DeclareSymbol{Xi2g}{-2}{\draw (-1,-0.25) node[xies] {} -- (0,1) node[xi] {};} 
\DeclareSymbol{Xi2g2}{-2}{\draw (-1,-0.25) node[xi] {} -- (0,1) node[xies] {};} 
\DeclareSymbol{cXi2}{-2}{\draw (0,-0.25) node[xi] {} -- (-1,1) node[xic] {};}
\DeclareSymbol{Xi3}{0}{\draw (0,0) node[xi] {} -- (-1,1) node[xi] {} -- (0,2) node[xi] {};}
\DeclareSymbol{XiIIXi}{0}{\draw (0,0) node[xi] {} -- (-1,1); \draw[kernels2] (-1,1) node[not] {} -- (0,2) node[xi] {};}

\DeclareSymbol{Xi4}{2}{\draw (0,0) node[xi] {} -- (-1,1) node[xi] {} -- (0,2) node[xi] {} -- (-1,3) node[xi] {};}
\DeclareSymbol{Xi4_1}{2}{\draw (0,0) node[xic] {} -- (-1,1) node[xic] {} -- (0,2) node[xi] {} -- (-1,3) node[xi] {};}
\DeclareSymbol{Xi4_2}{2}{\draw (0,0) node[xic] {} -- (-1,1) node[xi] {} -- (0,2) node[xi] {} -- (-1,3) node[xic] {};}
\DeclareSymbol{Xi2X}{-2}{\draw (0,-0.25) node[xi] {} -- (-1,1) node[xix] {};}
\DeclareSymbol{XXi2}{-2}{\draw (0,-0.25) node[xix] {} -- (-1,1) node[xi] {};}
\DeclareSymbol{IIXi}{0}{\draw (0,-0.25) node[not] {} -- (-1,1) node[xi] {} -- (0,2) node[xi] {};}
\DeclareSymbol{IXi^2}{-1}{\draw (-1,1) node[xi] {} -- (0,0) node[not] {} -- (1,1) node[xi] {};}
\DeclareSymbol{IIXi^2}{-4}{\draw (0,-1.5) node[not] {} -- (0,0);
\draw[kernels2] (-1,1) node[xi] {} -- (0,0) node[not] {} -- (1,1) node[xi] {};}
\DeclareSymbol{XiX}{-2.8}{\node[xibx] {};}
\DeclareSymbol{tauX}{-2.8}{ \node[X] {};}
\DeclareSymbol{Xi}{-2.8}{\node[xib] {};}

\DeclareSymbol{IXiX}{-1}{\draw (0,-0.25) node[not] {} -- (-1,1) node[xix] {};}
\DeclareSymbol{IXi3}{2}{\draw (0,-0.25) node[not] {} -- (-1,1) node[xi] {} -- (0,2) node[xi] {} -- (-1,3) node[xi] {};}
\DeclareSymbol{IXi}{-2}{\draw (0,-0.25) node[not] {} -- (-1,1) node[xi] {};}
\DeclareSymbol{XiI}{-2}{\draw (0,-0.25) node[xi] {} -- (-1,1) node[not] {};}

\DeclareSymbol{Xi4b}{0}{\draw(0,1.5) node[xi] {} -- (0,0); \draw (-1,1) node[xi] {} -- (0,0) node[xi] {} -- (1,1) node[xi] {};}
\DeclareSymbol{Xi4b'}{0}{\draw(0,1.5) node[xi] {} -- (0,-0.2); \draw (-1,1) node[xi] {} -- (0,-0.2) node[not] {} -- (1,1) node[xi] {};}
\DeclareSymbol{Xi4c}{0}{\draw (0,1) -- (0.8,2.2) node[xi] {};\draw (0,-0.25) node[xi] {} -- (0,1) node[xi] {} -- (-0.8,2.2) node[xi] {};}
\DeclareSymbol{Xi4d}{-4.5}{\draw (0,-1.5) node[not] {} -- (0,0); \draw (-1,1) node[xi] {} -- (0,0) node[xi] {} -- (1,1) node[xi] {};}
\DeclareSymbol{Xi4e}{0}{\draw (0,2) node[xi] {} -- (-1,1) node[xi] {} -- (0,0) node[xi] {} -- (1,1) node[xi] {};}
\DeclareSymbol{Xi4e'}{0}{\draw (0,2) node[xi] {} -- (-1,1) node[xi] {} -- (0,-0.2) node[not] {} -- (1,1) node[xi] {};}

\DeclareSymbol{Xitwo}
{0}{\draw[kernels2] (0,0) node[not] {} -- (-1,1) node[not] {}
-- (-2,2) node[not]{} -- (-3,3) node[xi]  {};
\draw[kernels2] (0,0) -- (1,1) node[xi] {};
\draw[kernels2] (-1,1) -- (0,2) node[xi] {};
\draw[kernels2] (-2,2) -- (-1,3) node[xi] {};}

\DeclareSymbol{IXitwo}
{0}{\draw (-.7,1.2) node[xi] {} -- (0,-0.2) -- (.7,1.2) node[xi] {};}
\DeclareSymbol{I1Xitwo}
{0}{\draw[kernels2] (0,0) node[not] {} -- (-1,1) node[xi] {};
\draw[kernels2] (0,0) -- (1,1) node[xi] {};}

\DeclareSymbol{I1Xitwob}
{0}{\draw[kernels2] (0,0) node[not] {} -- (-1,1) node[xic] {};
	\draw[kernels2] (0,0) -- (1,1) node[xic] {};}

\DeclareSymbol{I1Xitwou}
{0}{\draw[kernels2] (0,0) node[not] {} -- (-1,1) node[xi] {};
\draw[kernels2] (0,0) -- (1,1) node[xi] {}; \node at (-0.85,-0.2) {{\tiny $1$}}; \node at (0.9,-0.2) {{\tiny $0$}};}

\DeclareSymbol{I1Xitwoub}
{0}{\draw[kernels2] (0,0) node[not] {} -- (-1,1) node[xi] {};
	\draw[kernels2] (0,0) -- (1,1) node[xi] {}; \node at (-0.85,-0.2) {{\tiny $0$}}; \node at (0.9,-0.2) {{\tiny $1$}};}

\DeclareSymbol{I1Xitwoab}
{0}{\draw[kernels2] (0,0) node[not] {} -- (-1,1) node[xi] {};
\draw[kernels2] (0,0) -- (1,1) node[xi] {}; \node at (-0.85,-0.1) {{\tiny $\alpha$}}; \node at (0.9,-0.2) {{\tiny $\beta$}};}
\DeclareSymbol{I1Xitwoup}
{0}{\draw[kernels2] (0,0) node[not] {} -- (-1,1) node[xi] {};
\draw[kernels2] (0,0) -- (1,1) node[xi] {}; \node at (-0.85,0) {{\tiny $k$}}; \node at (0.9,-0.1) {{\tiny $l$}};}
\DeclareSymbol{I1Xitwobis}
{0}{\draw[kernels2] (0,0) node[not] {} -- (-1,1) node[xies] {};
\draw[kernels2] (0,0) -- (1,1) node[xies] {};}

\DeclareSymbol{I1Xitwog}
{0}{\draw[kernels2] (0,0) node[not] {} -- (-1,1) node[xies] {};
\draw[kernels2] (0,0) -- (1,1) node[xi] {};}

\DeclareSymbol{cI1Xitwo}
{0}{\draw[kernels2] (0,0) node[not] {} -- (-1,1) node[xic] {};
\draw[kernels2] (0,0) -- (1,1) node[xi] {};}

\DeclareSymbol{I1IXi3}{0}{\draw (0,0) node[xi] {} -- (-1,1) ; 
\draw[kernels2] (-1,1) node[not] {} -- (0,2) node[xi] {};
\draw[kernels2] (-1,1) node[not] {} -- (-2,2) node[xi] {};}

\DeclareSymbol{I1Xi3c}{-1}{\draw[kernels2](0,1.5) node[xi] {} -- (0,0) node[not] {}; \draw (-1,1) node[xi] {} -- (0,0) ; \draw[kernels2] (0,0) -- (1,1) node[xi] {};}

\DeclareSymbol{I1Xi3cbis}{-1}{\draw[kernels2](0,1.5) node[xies] {} -- (0,0) node[not] {}; \draw (-1,1) node[xies] {} -- (0,0) ; \draw[kernels2] (0,0) -- (1,1) node[xies] {};}

\DeclareSymbol{I1IXi3b}{0}{\draw[kernels2] (0,0) node[not] {} -- (-1,1) ; \draw[kernels2] (0,0)   -- (1,1) node[xi] {} ;
\draw (-1,1) node[xi] {} -- (0,2) node[xi] {};
}

\DeclareSymbol{I1IXi3c}{0}{\draw[kernels2] (0,0) node[not] {} -- (-1,1) ; \draw[kernels2] (0,0)   -- (1,1) node[xi] {} ;
\draw[kernels2] (-1,1) node[not] {} -- (0,2) node[xi] {};
\draw[kernels2] (-1,1) node[not] {} -- (-2,2) node[xi] {};}

\DeclareSymbol{I1IXi3cbis}{0}{\draw[kernels2] (0,0) node[not] {} -- (-1,1) ; \draw[kernels2] (0,0)   -- (1,1) node[xies] {} ;
\draw[kernels2] (-1,1) node[not] {} -- (0,2) node[xies] {};
\draw[kernels2] (-1,1) node[not] {} -- (-2,2) node[xies] {};}

\DeclareSymbol{I1Xi}{0}{\draw[kernels2] (0,0) node[not] {} -- (-1,1)  node[xi] {} ;}

\DeclareSymbol{I1Xi4a}{2}{\draw[kernels2] (0,0) node[not] {} -- (-1,1) ; \draw[kernels2] (0,0) node[not] {} -- (1,1) node[xi] {} ;
\draw (-1,1) node[xi] {} -- (0,2) node[xi] {} -- (-1,3) node[xi] {};}

\DeclareSymbol{cI1Xi4a}{2}{\draw[kernels2] (0,0) node[not] {} -- (-1,1) ; \draw[kernels2] (0,0) node[not] {} -- (1,1) node[xic] {} ;
\draw (-1,1) node[xic] {} -- (0,2) node[xi] {} -- (-1,3) node[xi] {};}

\DeclareSymbol{I1Xi4b}{2}{\draw (0,0) node[xi] {} -- (-1,1) node[xi] {} -- (0,2) ; \draw[kernels2] (0,2) node[not] {} -- (-1,3) node[xi] {};\draw[kernels2] (0,2)  -- (1,3) node[xi] {};
}

\DeclareSymbol{cI1Xi4b}{2}{\draw (0,0) node[xic] {} -- (-1,1) node[xic] {} -- (0,2) ; \draw[kernels2] (0,2) node[not] {} -- (-1,3) node[xi] {};\draw[kernels2] (0,2)  -- (1,3) node[xi] {};
}

\DeclareSymbol{I1Xi4c}{2}{\draw (0,0) node[xi] {} -- (-1,1) node[not] {}; \draw[kernels2] (-1,1) -- (0,2) ; 
\draw[kernels2] (-1,1) -- (-2,2) node[xi] {} ;
\draw (0,2) node[xi] {} -- (-1,3) node[xi] {};}

\DeclareSymbol{cI1Xi4c}{2}{\draw (0,0) node[xic] {} -- (-1,1) node[not] {}; \draw[kernels2] (-1,1) -- (0,2) ; 
\draw[kernels2] (-1,1) -- (-2,2) node[xic] {} ;
\draw (0,2) node[xi] {} -- (-1,3) node[xi] {};}

\DeclareSymbol{I1Xi4ab}{2}{\draw[kernels2] (0,0) node[not] {} -- (-1,1) ; \draw[kernels2] (0,0) node[not] {} -- (1,1) node[xi] {};\draw (-1,1) node[xi] {} -- (0,2) ; \draw[kernels2] (0,2) node[not] {} -- (-1,3) node[xi] {};\draw[kernels2] (0,2)  -- (1,3) node[xi] {}; }

\DeclareSymbol{cI1Xi4ab}{2}{\draw[kernels2] (0,0) node[not] {} -- (-1,1) ; \draw[kernels2] (0,0) node[not] {} -- (1,1) node[xic] {};\draw (-1,1) node[xic] {} -- (0,2) ; \draw[kernels2] (0,2) node[not] {} -- (-1,3) node[xi] {};\draw[kernels2] (0,2)  -- (1,3) node[xi] {}; }

\DeclareSymbol{I1Xi4bc}{2}{\draw (0,0) node[xi] {} -- (-1,1) node[not] {}; \draw[kernels2] (-1,1) -- (0,2) ; 
\draw[kernels2] (-1,1) -- (-2,2) node[xi] {} ; \draw[kernels2] (0,2) node[not] {} -- (-1,3) node[xi] {};\draw[kernels2] (0,2)  -- (1,3) node[xi] {};
}

\DeclareSymbol{cI1Xi4bc}{2}{\draw (0,0) node[xic] {} -- (-1,1) node[not] {}; \draw[kernels2] (-1,1) -- (0,2) ; 
\draw[kernels2] (-1,1) -- (-2,2) node[xic] {} ; \draw[kernels2] (0,2) node[not] {} -- (-1,3) node[xi] {};\draw[kernels2] (0,2)  -- (1,3) node[xi] {};
}

\DeclareSymbol{I1Xi4abcc1}{2}{\draw[kernels2] (0,0) node[not] {} -- (-1,1) node[not] {}
-- (-2,2) node[not]{} -- (-3,3) node[xic]  {};
\draw[kernels2] (0,0) -- (1,1) node[xic] {};
\draw[kernels2] (-1,1) -- (0,2) node[xi] {};
\draw[kernels2] (-2,2) -- (-1,3) node[xi] {};
}

\DeclareSymbol{I1Xi4abcc1b}{2}{\draw[kernels2] (0,0) node[not] {} -- (-1,1) node[not] {}
-- (-2,2) node[not]{} -- (-3,3) node[xi]  {};
\draw[kernels2] (0,0) -- (1,1) node[xic] {};
\draw[kernels2] (-1,1) -- (0,2) node[xic] {};
\draw[kernels2] (-2,2) -- (-1,3) node[xi] {};
}

\DeclareSymbol{I1Xi4abcc2}{2}{\draw[kernels2] (0,0) node[not] {} -- (-1,1) node[not] {}
-- (-2,2) node[not]{} -- (-3,3) node[xic]  {};
\draw[kernels2] (0,0) -- (1,1) node[xi] {};
\draw[kernels2] (-1,1) -- (0,2) node[xi] {};
\draw[kernels2] (-2,2) -- (-1,3) node[xic] {};
}

\DeclareSymbol{I1Xi4ac}{2}{\draw[kernels2] (0,0) node[not] {} -- (-1,1) ; \draw[kernels2] (0,0) node[not] {} -- (1,1) node[xi] {}; 
\draw[kernels2] (-1,1) node[not] {} -- (0,2) ; 
\draw[kernels2] (-1,1) -- (-2,2) node[xi] {} ;
\draw (0,2) node[xi] {} -- (-1,3) node[xi] {};}

\DeclareSymbol{cI1Xi4ac}{2}{\draw[kernels2] (0,0) node[not] {} -- (-1,1) ; \draw[kernels2] (0,0) node[not] {} -- (1,1) node[xic] {}; 
\draw[kernels2] (-1,1) node[not] {} -- (0,2) ; 
\draw[kernels2] (-1,1) -- (-2,2) node[xic] {} ;
\draw (0,2) node[xi] {} -- (-1,3) node[xi] {};}

\DeclareSymbol{I1Xi4acc1}{2}{\draw[kernels2] (0,0) node[not] {} -- (-1,1) ; \draw[kernels2] (0,0) node[not] {} -- (1,1) node[xic] {}; 
\draw[kernels2] (-1,1) node[not] {} -- (0,2) ; 
\draw[kernels2] (-1,1) -- (-2,2) node[xi] {} ;
\draw (0,2) node[xic] {} -- (-1,3) node[xi] {};}

\DeclareSymbol{I1Xi4acc2}{2}{\draw[kernels2] (0,0) node[not] {} -- (-1,1) ; \draw[kernels2] (0,0) node[not] {} -- (1,1) node[xic] {}; 
\draw[kernels2] (-1,1) node[not] {} -- (0,2) ; 
\draw[kernels2] (-1,1) -- (-2,2) node[xi] {} ;
\draw (0,2) node[xi] {} -- (-1,3) node[xic] {};}

\DeclareSymbol{2I1Xi4}{2}{\draw[kernels2] (0,0) node[not] {} -- (-1,1) node[not] {};
\draw[kernels2] (0,0) -- (1,1) node[not] {};
\draw[kernels2] (-1,1) -- (-1.5,2.5) node[xi] {};
\draw[kernels2] (-1,1) -- (-0.5,2.5) node[xi] {};
\draw[kernels2] (1,1) -- (0.5,2.5) node[xi] {};
\draw[kernels2] (1,1) -- (1.5,2.5) node[xi] {};
}

\DeclareSymbol{2I1Xi4dis}{2}{\draw[kernels2] (0,0) node[not] {} -- (-1,1) node[not] {};
\draw[kernels2] (0,0) -- (1,1) node[not] {};
\draw[kernels2] (-1,1) -- (-1.5,2.5) node[xies] {};
\draw[kernels2] (-1,1) -- (-0.5,2.5) node[xies] {};
\draw[kernels2] (1,1) -- (0.5,2.5) node[xies] {};
\draw[kernels2] (1,1) -- (1.5,2.5) node[xies] {};
}

\DeclareSymbol{2I1Xi4c1}{2}{\draw[kernels2] (0,0) node[not] {} -- (-1,1) node[not] {};
\draw[kernels2] (0,0) -- (1,1) node[not] {};
\draw[kernels2] (-1,1) -- (-1.5,2.5) node[xic] {};
\draw[kernels2] (-1,1) -- (-0.5,2.5) node[xi] {};
\draw[kernels2] (1,1) -- (0.5,2.5) node[xic] {};
\draw[kernels2] (1,1) -- (1.5,2.5) node[xi] {};
}

\DeclareSymbol{2I1Xi4c2}{2}{\draw[kernels2] (0,0) node[not] {} -- (-1,1) node[not] {};
\draw[kernels2] (0,0) -- (1,1) node[not] {};
\draw[kernels2] (-1,1) -- (-1.5,2.5) node[xic] {};
\draw[kernels2] (-1,1) -- (-0.5,2.5) node[xic] {};
\draw[kernels2] (1,1) -- (0.5,2.5) node[xi] {};
\draw[kernels2] (1,1) -- (1.5,2.5) node[xi] {};
}

\DeclareSymbol{2I1Xi4b}{2}{\draw[kernels2] (0,0) node[not] {} -- (-1,1) ;
\draw[kernels2] (0,0) -- (1,1);
\draw (-1,1) node[xi] {} -- (-1,2.5) node[xi] {};
\draw (1,1)  node[xi] {} -- (1,2.5) node[xi] {};
}

\DeclareSymbol{2I1Xi4bb}{2}{\draw[kernels2] (0,0) node[not] {} -- (-1,1) ;
\draw[kernels2] (0,0) -- (1,1);
\draw (-1,1) node[xi] {} -- (-1,2.5) node[xiesf] {};
\draw (1,1)  node[xi] {} -- (1,2.5) node[xic] {};
}

\DeclareSymbol{2I1Xi4c}{2}{\draw[kernels2] (0,0) node[not] {} -- (-1,1);
\draw[kernels2] (0,0) -- (1,1) node[not] {};
\draw (-1,1)  node[xi] {} -- (-1,2.5) node[xi] {};
\draw[kernels2] (1,1) -- (0.4,2.5) node[xi] {};
\draw[kernels2] (1,1) -- (1.6,2.5) node[xi] {};
}

\DeclareSymbol{2I1Xi4cc1}{2}{\draw[kernels2] (0,0) node[not] {} -- (-1,1);
\draw[kernels2] (0,0) -- (1,1) node[not] {};
\draw (-1,1)  node[xic] {} -- (-1,2.5) node[xi] {};
\draw[kernels2] (1,1) -- (0.4,2.5) node[xic] {};
\draw[kernels2] (1,1) -- (1.6,2.5) node[xi] {};
}

\DeclareSymbol{2I1Xi4cc2}{2}{\draw[kernels2] (0,0) node[not] {} -- (-1,1);
\draw[kernels2] (0,0) -- (1,1) node[not] {};
\draw (-1,1)  node[xic] {} -- (-1,2.5) node[xic] {};
\draw[kernels2] (1,1) -- (0.4,2.5) node[xi] {};
\draw[kernels2] (1,1) -- (1.6,2.5) node[xi] {};
}

\DeclareSymbol{Xi4ba}{0}{\draw(-0.5,1.5) node[xi] {} -- (0,0); \draw (-1.5,1) node[xi] {} -- (0,0) node[not] {}; \draw[kernels2] (0,0) -- (1.5,1) node[xi] {};
\draw[kernels2] (0,0) -- (0.5,1.5) node[xi] {} ;}

\DeclareSymbol{Xi4badis}{0}{\draw(-0.5,1.5) node[xies] {} -- (0,0); \draw (-1.5,1) node[xies] {} -- (0,0) node[not] {}; \draw[kernels2] (0,0) -- (1.5,1) node[xies] {};
\draw[kernels2] (0,0) -- (0.5,1.5) node[xies] {} ;}

\DeclareSymbol{Xi4ba1}{0}{\draw(-0.5,1.5) node[xi] {} -- (0,0); \draw (-1.5,1) node[xi] {} -- (0,0) node[not] {}; \draw[kernels2] (0,0) -- (1.5,1) node[xic] {};
\draw[kernels2] (0,0) -- (0.5,1.5) node[xic] {} ;}

\DeclareSymbol{Xi4ba1b}{0}{\draw(-0.5,1.5) node[xic] {} -- (0,0); \draw (-1.5,1) node[xic] {} -- (0,0) node[not] {}; \draw[kernels2] (0,0) -- (1.5,1) node[xi] {};
\draw[kernels2] (0,0) -- (0.5,1.5) node[xi] {} ;}

\DeclareSymbol{Xi4ba1bdiff}{0}{\draw(-0.5,1.5) node[xic] {} -- (0,0); \draw (-1.5,1) node[xic] {} -- (0,0) node[not] {}; \draw (0,0) -- (1.5,1) node[xi] {};
\draw (0,0) -- (0.5,1.5) node[xi] {};
\draw(0,0) node[diff] {};}

\DeclareSymbol{Xi4ba1bb}{0}{\draw(-0.5,1.5) node[xic] {} -- (0,0); \draw (-1.5,1) node[xiesf] {} -- (0,0) node[not] {}; \draw[kernels2] (0,0) -- (1.5,1) node[xi] {};
\draw[kernels2] (0,0) -- (0.5,1.5) node[xi] {} ;}

\DeclareSymbol{Xi4ba2}{0}{\draw(-0.5,1.5) node[xi] {} -- (0,0); \draw (-1.5,1) node[xic] {} -- (0,0) node[not] {}; \draw[kernels2] (0,0) -- (1.5,1) node[xi] {};
\draw[kernels2] (0,0) -- (0.5,1.5) node[xic] {} ;}

\DeclareSymbol{Xi4ba2b}{0}{\draw(-0.5,1.5) node[xi] {} -- (0,0); \draw (-1.5,1) node[xic] {} -- (0,0) node[not] {}; \draw[kernels2] (0,0) -- (1.5,1) node[xi] {};
\draw[kernels2] (0,0) -- (0.5,1.5) node[xiesf] {} ;}

\DeclareSymbol{Xi4ca}{0}{\draw (0,1) -- (-1,2.2) node[xi] {};\draw (0,-0.25) node[xi] {} -- (0,1) ; \draw[kernels2] (0,1) node[not] {} -- (1,2.2) node[xi] {};
\draw[kernels2] (0,1) {} -- (0,2.7) node[xi] {};
}

\DeclareSymbol{Xi4cb}{0}{\draw (-1,1) -- (-2,2) node[xi] {};\draw[kernels2] (0,0)  -- (-1,1) node[xi] {} ; \draw[kernels2] (0,0) node[not] {} -- (1,1) node[xi] {} ; 
\draw (-1,1) node[xi] {} -- (0,2) node[xi] {};}

\DeclareSymbol{Xi4cbb}{0}{\draw (-1,1) -- (-2,2) node[xiesf] {};\draw[kernels2] (0,0)  -- (-1,1) node[xi] {} ; \draw[kernels2] (0,0) node[not] {} -- (1,1) node[xi] {} ; 
\draw (-1,1) node[xi] {} -- (0,2) node[xic] {};}

\DeclareSymbol{Xi4cbc1}{0}{\draw (-1,1) -- (-2,2) node[xic] {};\draw[kernels2] (0,0)  -- (-1,1) node[xic] {} ; \draw[kernels2] (0,0) node[not] {} -- (1,1) node[xi] {} ; 
\draw (-1,1) node[xic] {} -- (0,2) node[xi] {};}

\DeclareSymbol{Xi4cbc2}{0}{\draw (-1,1) -- (-2,2) node[xi] {};\draw[kernels2] (0,0)  -- (-1,1) node[xi] {} ; \draw[kernels2] (0,0) node[not] {} -- (1,1) node[xic] {} ; 
\draw (-1,1) node[xic] {} -- (0,2) node[xi] {};}

\DeclareSymbol{Xi4cab}{0}{\draw (-1,1) -- (-2,2) node[xi] {};\draw[kernels2] (0,0)  -- (-1,1); \draw[kernels2] (0,0) node[not] {} -- (1,1) node[xi] {} ; 
\draw[kernels2] (-1,1)  {} -- (0,2) node[xi] {};
\draw[kernels2] (-1,1) node[not] {} -- (-1,2.5) node[xi] {};
}

\DeclareSymbol{Xi4cabdis}{0}{\draw (-1,1) -- (-2,2) node[xies] {};\draw[kernels2] (0,0)  -- (-1,1); \draw[kernels2] (0,0) node[not] {} -- (1,1) node[xies] {} ; 
\draw[kernels2] (-1,1)  {} -- (0,2) node[xies] {};
\draw[kernels2] (-1,1) node[not] {} -- (-1,2.5) node[xies] {};
}

\DeclareSymbol{Xi4cabc1}{0}{\draw (-1,1) -- (-2,2) node[xi] {};\draw[kernels2] (0,0)  -- (-1,1); \draw[kernels2] (0,0) node[not] {} -- (1,1) node[xic] {} ; 
\draw[kernels2] (-1,1)  {} -- (0,2) node[xic] {};
\draw[kernels2] (-1,1) node[not] {} -- (-1,2.5) node[xi] {};
}

\DeclareSymbol{Xi4cabc2}{0}{\draw (-1,1) -- (-2,2) node[xic] {};\draw[kernels2] (0,0)  -- (-1,1); \draw[kernels2] (0,0) node[not] {} -- (1,1) node[xic] {} ; 
\draw[kernels2] (-1,1)  {} -- (0,2) node[xi] {};
\draw[kernels2] (-1,1) node[not] {} -- (-1,2.5) node[xi] {};
}

\DeclareSymbol{Xi4ea}{1.5}{\draw (-1,2.5) node[xi] {} -- (-1,1) node[xi] {} -- (0,0); 
 \draw[kernels2] (0,0)  -- (1,1) node[xi] {};
\draw[kernels2] (0,0) node[not] {} -- (0,1.5) node[xi] {}; }

\DeclareSymbol{Xi4eac1}{1.5}{\draw (-1,2.5) node[xic] {} -- (-1,1) node[xi] {} -- (0,0); 
 \draw[kernels2] (0,0)  -- (1,1) node[xic] {};
\draw[kernels2] (0,0) node[not] {} -- (0,1.5) node[xi] {}; }

\DeclareSymbol{Xi4eac1b}{1.5}{\draw (-1,2.5) node[xic] {} -- (-1,1) node[xi] {} -- (0,0); 
 \draw[kernels2] (0,0)  -- (1,1) node[xiesf] {};
\draw[kernels2] (0,0) node[not] {} -- (0,1.5) node[xi] {}; }

\DeclareSymbol{Xi4eac2}{1.5}{\draw (-1,2.5) node[xic] {} -- (-1,1) node[xic] {} -- (0,0); 
 \draw[kernels2] (0,0)  -- (1,1) node[xi] {};
\draw[kernels2] (0,0) node[not] {} -- (0,1.5) node[xi] {}; }

\DeclareSymbol{Xi4eact1}{1.5}{\draw (-1,2.5) node[xic] {} -- (-1,1) node[xi] {} -- (0,0); 
 \draw (0,0)  -- (1,1) node[xic] {};
\draw[rho] (0,0) node[not] {} -- (0,1.5) node[xi] {}; }

\DeclareSymbol{Xi4eact2}{1.5}{\draw[rho] (-1,2.5) node[xic] {} -- (-1,1) node[xi] {} -- (0,0); 
 \draw (0,0)  -- (1,1) node[xic] {};
\draw (0,0) node[not] {} -- (0,1.5) node[xi] {}; }

\DeclareSymbol{Xi4eabis}{1.5}{\draw (-1,2.5) node[xi] {} -- (-1,1) ; \draw[kernels2] (-1,1) node[xi] {} -- (0,0); 
 \draw (0,0)  -- (1,1) node[xi] {};
\draw[kernels2] (0,0) node[not] {} -- (0,1.5) node[xi] {}; }

\DeclareSymbol{Xi4eabisc1}{1.5}{\draw (-1,2.5) node[xic] {} -- (-1,1) ; \draw[kernels2] (-1,1) node[xi] {} -- (0,0); 
 \draw (0,0)  -- (1,1) node[xi] {};
\draw[kernels2] (0,0) node[not] {} -- (0,1.5) node[xic] {}; }

\DeclareSymbol{Xi4eabisc1b}{1.5}{\draw (-1,2.5) node[xic] {} -- (-1,1) ; \draw[kernels2] (-1,1) node[xi] {} -- (0,0); 
 \draw (0,0)  -- (1,1) node[xi] {};
\draw[kernels2] (0,0) node[not] {} -- (0,1.5) node[xiesf] {}; }

\DeclareSymbol{Xi4eabisc1bis}{1.5}{\draw (-1,2.5) node[xi] {} -- (-1,1) ; \draw[kernels2] (-1,1) node[xi] {} -- (0,0); 
 \draw (0,0)  -- (1,1) node[xi] {};
\draw[kernels2] (0,0) node[not] {} -- (0,1.5) node[xi] {};
\draw (-2,1) node[] {\tiny{$i$}};
\draw (-2,2.5) node[] {\tiny{$\ell$}};
\draw (2,1) node[] {\tiny{$k$}};
\draw (0,2.5) node[] {\tiny{$j$}};
 }

\DeclareSymbol{Xi4eabisc1tris}{1.5}{\draw (-1,2.5) node[xi] {} -- (-1,1) ; \draw[kernels2] (-1,1) node[xi] {} -- (0,0); 
 \draw (0,0)  -- (1,1) node[xi] {};
\draw[kernels2] (0,0) node[not] {} -- (0,1.5) node[xi] {};
\draw (-2,1) node[] {\tiny{i}};
\draw (-2,2.5) node[] {\tiny{j}};
\draw (2,1) node[] {\tiny{j}};
\draw (0,2.5) node[] {\tiny{i}};
 }

\DeclareSymbol{Xi4eabisc1quater}{1.5}{\draw (-1,2.5) node[xic] {} -- (-1,1) ; \draw[kernels2] (-1,1) node[xi] {} -- (0,0); 
 \draw (0,0)  -- (1,1) node[xic] {};
\draw[kernels2] (0,0) node[not] {} -- (0,1.5) node[xi] {};
 }

\DeclareSymbol{Xi4eabisc2}{1.5}{\draw (-1,2.5) node[xic] {} -- (-1,1) ; \draw[kernels2] (-1,1) node[xi] {} -- (0,0); 
 \draw (0,0)  -- (1,1) node[xic] {};
\draw[kernels2] (0,0) node[not] {} -- (0,1.5) node[xi] {}; }

\DeclareSymbol{Xi4eabisc2l}{1.5}{\draw (-1,2.5) node[xiesf] {} -- (-1,1) ; \draw[kernels2] (-1,1) node[xi] {} -- (0,0); 
 \draw (0,0)  -- (1,1) node[xic] {};
\draw[kernels2] (0,0) node[not] {} -- (0,1.5) node[xi] {}; }

\DeclareSymbol{Xi4eabisc2r}{1.5}{\draw (-1,2.5) node[xic] {} -- (-1,1) ; \draw[kernels2] (-1,1) node[xi] {} -- (0,0); 
 \draw (0,0)  -- (1,1) node[xiesf] {};
\draw[kernels2] (0,0) node[not] {} -- (0,1.5) node[xi] {}; }

\DeclareSymbol{Xi4eabisc3}{1.5}{\draw (-1,2.5) node[xic] {} -- (-1,1) ; \draw[kernels2] (-1,1) node[xic] {} -- (0,0); 
 \draw (0,0)  -- (1,1) node[xi] {};
\draw[kernels2] (0,0) node[not] {} -- (0,1.5) node[xi] {}; }

\DeclareSymbol{Xi4eb}{0}{
\draw[kernels2] (0,2) node[xi] {} -- (-1,1) ; \draw[kernels2] (-2,2)  node[xi] {} -- (-1,1) ; \draw (-1,1)  node[not] {} -- (0,0); 
 \draw (0,0) node[xi] {}  -- (1,1) node[xi] {};
}

\DeclareSymbol{Xi4eab}{1.5}{\draw[kernels2] (-1,2.5) node[xi] {} -- (-1,1) ; \draw[kernels2] (-2,2)  node[xi] {} -- (-1,1) ; \draw (-1,1)  node[not] {} -- (0,0); 
 \draw[kernels2] (0,0)  -- (1,1) node[xi] {};
\draw[kernels2] (0,0) node[not] {} -- (0,1.5) node[xi] {}; 
}

\DeclareSymbol{Xi4eabdis}{1.5}{\draw[kernels2] (-1,2.5) node[xies] {} -- (-1,1) ; \draw[kernels2] (-2,2)  node[xies] {} -- (-1,1) ; \draw (-1,1)  node[not] {} -- (0,0); 
 \draw[kernels2] (0,0)  -- (1,1) node[xies] {};
\draw[kernels2] (0,0) node[not] {} -- (0,1.5) node[xies] {}; 
}

\DeclareSymbol{Xi4eabc1}{1.5}{\draw[kernels2] (-1,2.5) node[xic] {} -- (-1,1) ; \draw[kernels2] (-2,2)  node[xi] {} -- (-1,1) ; \draw (-1,1)  node[not] {} -- (0,0); 
 \draw[kernels2] (0,0)  -- (1,1) node[xic] {};
\draw[kernels2] (0,0) node[not] {} -- (0,1.5) node[xi] {}; 
}

\DeclareSymbol{Xi4eabc2}{1.5}{\draw[kernels2] (-1,2.5) node[xi] {} -- (-1,1) ; \draw[kernels2] (-2,2)  node[xi] {} -- (-1,1) ; \draw (-1,1)  node[not] {} -- (0,0); 
 \draw[kernels2] (0,0)  -- (1,1) node[xic] {};
\draw[kernels2] (0,0) node[not] {} -- (0,1.5) node[xic] {}; 
}

\DeclareSymbol{Xi4eabbis}{1.5}{\draw[kernels2] (-1,2.5) node[xi] {} -- (-1,1) ; \draw[kernels2] (-2,2)  node[xi] {} -- (-1,1) ; \draw[kernels2] (-1,1)  node[not] {} -- (0,0); 
 \draw (0,0)  -- (1,1) node[xi] {};
\draw[kernels2] (0,0) node[not] {} -- (0,1.5) node[xi] {}; 
}

\DeclareSymbol{Xi4eabbisc1}{1.5}{\draw[kernels2] (-1,2.5) node[xic] {} -- (-1,1) ; \draw[kernels2] (-2,2)  node[xi] {} -- (-1,1) ; \draw[kernels2] (-1,1)  node[not] {} -- (0,0); 
 \draw (0,0)  -- (1,1) node[xic] {};
\draw[kernels2] (0,0) node[not] {} -- (0,1.5) node[xi] {}; 
}

\DeclareSymbol{Xi4eabbisc1perm}{1.5}{\draw[kernels2] (-1,2.5) node[xi] {} -- (-1,1) ; \draw[kernels2] (-2,2)  node[xic] {} -- (-1,1) ; \draw[kernels2] (-1,1)  node[not] {} -- (0,0); 
 \draw (0,0)  -- (1,1) node[xic] {};
\draw[kernels2] (0,0) node[not] {} -- (0,1.5) node[xi] {}; 
}

\DeclareSymbol{Xi4eabbisc2}{1.5}{\draw[kernels2] (-1,2.5) node[xi] {} -- (-1,1) ; \draw[kernels2] (-2,2)  node[xi] {} -- (-1,1) ; \draw[kernels2] (-1,1)  node[not] {} -- (0,0); 
 \draw (0,0)  -- (1,1) node[xic] {};
\draw[kernels2] (0,0) node[not] {} -- (0,1.5) node[xic] {}; 
}

\DeclareSymbol{Xi2cbis}{0}{\draw[kernels2] (0,1) -- (0.8,2.2) node[xi] {};\draw[kernels2] (0,-0.25) node[not] {} -- (0,1); \draw[kernels2] (0,1) node[not] {} -- (-0.8,2.2) node[xi] {};}

\DeclareSymbol{Xi2cbis1}{0}{\draw (0,1) -- (-0.8,2.2) node[xi] {};\draw[kernels2] (0,-0.25) node[not] {} -- (0,1) node[xi] {}; }

\DeclareSymbol{Xi2Xbis}{-2}{\draw[kernels2] (0,-0.25)  -- (-1,1) ; \draw (-1,1) node[xix] {};
\draw[kernels2] (0,-0.25) node[not] {} -- (1,1) node[xi] {};}

\DeclareSymbol{XXi2bis}{-2}{\draw[kernels2] (0,-0.25) -- (-1,1) node[xi] {};
\draw[kernels2] (0,-0.25) node[X] {} -- (1,1) node[xi] {};}

\DeclareSymbol{I1XiIXi}{0}{\draw[kernels2] (0,-0.25) -- (1,1) node[xi] {};
\draw (0,-0.25) node[not] {} -- (-1,1) node[xi] {};}

\DeclareSymbol{I1XiIXib}{0}{\draw  (0,-0.25) node[xi] {} -- (0,1) node[not] {};
\draw[kernels2] (0,1) -- (0,2.25) ; \draw (0,2.25) node[xi]{}; }

\DeclareSymbol{I1XiIXic}{0}{
\draw[kernels2] (0,0) -- (1,1) node[xi] {} ; 
\draw[kernels2] (0,0) node[not] {}  -- (-1,1) node[not] {} -- (0,2) node[xi] {};
}

\DeclareSymbol{thin}{1.4}{\draw[pagebackground] (-0.3,0) -- (0.3,0); \draw  (0,0) -- (0,2);}
\DeclareSymbol{thin2}{1.4}{\draw[pagebackground] (-0.3,0) -- (0.3,0); \draw[tinydots]  (0,0) -- (0,2);}

\DeclareSymbol{thick}{1.4}{\draw[pagebackground] (-0.3,0) -- (0.3,0); \draw[kernels2]  (0,0) -- (0,2);}
\DeclareSymbol{thick2}{1.4}{\draw[pagebackground] (-0.3,0) -- (0.3,0); \draw[kernels2,tinydots]  (0,0) -- (0,2);}

\DeclareSymbol{Xi4ind}{2}{\draw (0,0) node[xi,label={[label distance=-0.2em]right: \scriptsize  $ i $}]  { } -- (-1,1) node[xi,label={[label distance=-0.2em]left: \scriptsize  $ j $}] {} -- (0,2) node[xi,label={[label distance=-0.2em]right: \scriptsize  $ k $}] {} -- (-1,3) node[xi,label={[label distance=-0.2em]left: \scriptsize  $ \ell $}] {};}

\DeclareSymbol{Xi4c1}{2}{\draw (0,0) node[xic] {} -- (-1,1) node[xi] {} -- (0,2) node[xic] {} -- (-1,3) node[xi] {};} 
\DeclareSymbol{IXi2ex}{0}{\draw (0,-0.25) node[xie] {} -- (-1,1) node[xi] {} ; \draw (0,-0.25)-- (1,1) node[xi] {};}
\DeclareSymbol{IXi2ex1}{0}{\draw (0,-0.25) node[xie] {} -- (-1,1) node[xi] {} -- (0,2) node[xi] {};}

\DeclareSymbol{Xi4b1}{0}{\draw(0,1.5) node[xic] {} -- (0,0); \draw (-1,1) node[xic] {} -- (0,0) node[xi] {} -- (1,1) node[xi] {};}

\DeclareSymbol{Xi4ec1}{0}{\draw (0,2) node[xi] {} -- (-1,1) node[xic] {} -- (0,0) node[xic] {} -- (1,1) node[xi] {};}
\DeclareSymbol{Xi4ec2}{0}{\draw (0,2) node[xic] {} -- (-1,1) node[xi] {} -- (0,0) node[xic] {} -- (1,1) node[xi] {};}
\DeclareSymbol{Xi4ec3}{0}{\draw (0,2) node[xic] {} -- (-1,1) node[xic] {} -- (0,0) node[xi] {} -- (1,1) node[xi] {};}

\DeclareSymbol{I1Xi4ac1}{2}{\draw[kernels2] (0,0) node[not] {} -- (-1,1) ; \draw[kernels2] (0,0) node[not] {} -- (1,1) node[xic] {} ;
\draw (-1,1) node[xi] {} -- (0,2) node[xic] {} -- (-1,3) node[xi] {};}

\DeclareSymbol{I1Xi4ac2}{2}{\draw[kernels2] (0,0) node[not] {} -- (-1,1) ; \draw[kernels2] (0,0) node[not] {} -- (1,1) node[xic] {} ;
\draw (-1,1) node[xi] {} -- (0,2) node[xi] {} -- (-1,3) node[xic] {};}

\DeclareSymbol{I1Xi4bp}{2}{\draw (0,0) node[not] {} -- (-1,1) node[xi] {} -- (0,2) ; \draw[kernels2] (0,2) node[not] {} -- (-1,3) node[xi] {};\draw[kernels2] (0,2)  -- (1,3) node[xi] {};
}

\DeclareSymbol{I1Xi4bc1}{2}{\draw (0,0) node[xic] {} -- (-1,1) node[xi] {} -- (0,2) ; \draw[kernels2] (0,2) node[not] {} -- (-1,3) node[xi] {};\draw[kernels2] (0,2)  -- (1,3) node[xic] {};
}

\DeclareSymbol{I1Xi4bc2}{2}{\draw (0,0) node[xic] {} -- (-1,1) node[xi] {} -- (0,2) ; \draw[kernels2] (0,2) node[not] {} -- (-1,3) node[xic] {};\draw[kernels2] (0,2)  -- (1,3) node[xi] {};
}

\DeclareSymbol{I1Xi4cp}{2}{\draw (0,0) node[not] {} -- (-1,1) node[not] {}; \draw[kernels2] (-1,1) -- (0,2) ; 
\draw[kernels2] (-1,1) -- (-2,2) node[xi] {} ;
\draw (0,2) node[xi] {} -- (-1,3) node[xi] {};}

\DeclareSymbol{I1Xi4cc1}{2}{\draw (0,0) node[xic] {} -- (-1,1) node[not] {}; \draw[kernels2] (-1,1) -- (0,2) ; 
\draw[kernels2] (-1,1) -- (-2,2) node[xi] {} ;
\draw (0,2) node[xic] {} -- (-1,3) node[xi] {};}

\DeclareSymbol{I1Xi4cc2}{2}{\draw (0,0) node[xic] {} -- (-1,1) node[not] {}; \draw[kernels2] (-1,1) -- (0,2) ; 
\draw[kernels2] (-1,1) -- (-2,2) node[xi] {} ;
\draw (0,2) node[xi] {} -- (-1,3) node[xic] {};}

\DeclareSymbol{I1Xi4abc1}{2}{\draw[kernels2] (0,0) node[not] {} -- (-1,1) ; \draw[kernels2] (0,0) node[not] {} -- (1,1) node[xic] {};\draw (-1,1) node[xi] {} -- (0,2) ; \draw[kernels2] (0,2) node[not] {} -- (-1,3) node[xic] {};\draw[kernels2] (0,2)  -- (1,3) node[xi] {}; }

\DeclareSymbol{I1Xi4abc2}{2}{\draw[kernels2] (0,0) node[not] {} -- (-1,1) ; \draw[kernels2] (0,0) node[not] {} -- (1,1) node[xic] {};\draw (-1,1) node[xi] {} -- (0,2) ; \draw[kernels2] (0,2) node[not] {} -- (-1,3) node[xi] {};\draw[kernels2] (0,2)  -- (1,3) node[xic] {}; }

\DeclareSymbol{R1}{0}{\draw (-1,1) node[xi] {} -- (0,0) node[not] {};
\draw[kernels2] (0,1.5) node[xic] {} -- (0,0) -- (1,1) node[xic] {};}
\DeclareSymbol{R2}{0}{\draw (-1,1) node[xic] {} -- (0,0) node[not] {};
\draw[kernels2] (0,1.5)  {} -- (0,0) -- (1,1)  {};
\draw (0,1.5) node[xi] {};
\draw (1,1) node[xic] {};
}
\DeclareSymbol{R3}{1}{\draw[kernels2] (-1,1.5)  {} -- (0,0) node[not] {} -- (1,1.5);
\draw (-1,1.5) node[xi] {};
\draw[kernels2] (0,3) {} -- (1,1.5) -- (2,3)  {};
\draw  (0,3) node[xic] {} ;
\draw (2,3) node[xic] {};}
\DeclareSymbol{R4}{1}{\draw[kernels2] (-1,1.5) node[xic] {} -- (0,0) node[not] {} -- (1,1.5);
\draw[kernels2] (0,3) {} -- (1,1.5) -- (2,3) node[xic] {};
\draw (0,3) node[xi] {};}

\DeclareSymbol{I1Xi4bcp}{2}{\draw (0,0) node[not] {} -- (-1,1) node[not] {}; \draw[kernels2] (-1,1) -- (0,2) ; 
\draw[kernels2] (-1,1) -- (-2,2) node[xi] {} ; \draw[kernels2] (0,2) node[not] {} -- (-1,3) node[xi] {};\draw[kernels2] (0,2)  -- (1,3) node[xi] {};
}

\DeclareSymbol{I1Xi4bcc1}{2}{\draw (0,0) node[xic] {} -- (-1,1) node[not] {}; \draw[kernels2] (-1,1) -- (0,2) ; 
\draw[kernels2] (-1,1) -- (-2,2) node[xi] {} ; \draw[kernels2] (0,2) node[not] {} -- (-1,3) node[xi] {};\draw[kernels2] (0,2)  -- (1,3) node[xic] {};
}

\DeclareSymbol{I1Xi4bcc2}{2}{\draw (0,0) node[xic] {} -- (-1,1) node[not] {}; \draw[kernels2] (-1,1) -- (0,2) ; 
\draw[kernels2] (-1,1) -- (-2,2) node[xi] {} ; \draw[kernels2] (0,2) node[not] {} -- (-1,3) node[xic] {};\draw[kernels2] (0,2)  -- (1,3) node[xi] {};
} 

\DeclareSymbol{2I1Xi4bc1}{2}{\draw[kernels2] (0,0) node[not] {} -- (-1,1) ;
\draw[kernels2] (0,0) -- (1,1);
\draw (-1,1) node[xic] {} -- (-1,2.5) node[xi] {};
\draw (1,1)  node[xic] {} -- (1,2.5) node[xi] {};
}

\DeclareSymbol{2I1Xi4bc2}{2}{\draw[kernels2] (0,0) node[not] {} -- (-1,1) ;
\draw[kernels2] (0,0) -- (1,1);
\draw (-1,1) node[xi] {} -- (-1,2.5) node[xic] {};
\draw (1,1)  node[xic] {} -- (1,2.5) node[xi] {};
}

\DeclareSymbol{diff2I1Xi4bc2}{2}{\draw (0,0) node[diff] {} -- (-1,1) ;
\draw (0,0) -- (1,1);
\draw (-1,1) node[xi] {} -- (-1,2.5) node[xic] {};
\draw (1,1)  node[xic] {} -- (1,2.5) node[xi] {};
}

\DeclareSymbol{2I1Xi4bc3}{2}{\draw[kernels2] (0,0) node[not] {} -- (-1,1) ;
\draw[kernels2] (0,0) -- (1,1);
\draw (-1,1) node[xic] {} -- (-1,2.5) node[xic] {};
\draw (1,1)  node[xi] {} -- (1,2.5) node[xi] {};
}

\DeclareSymbol{Xi41}{0}{\draw (0,1) -- (0.8,2.2) node[xic] {};\draw (0,-0.25) node[xi] {} -- (0,1) node[xi] {} -- (-0.8,2.2) node[xic] {};} 

\DeclareSymbol{Xi42}{0}{\draw (0,1) -- (0.8,2.2) node[xi] {};\draw (0,-0.25) node[xic] {} -- (0,1) node[xi] {} -- (-0.8,2.2) node[xic] {};}

\DeclareSymbol{Xi4ca1b}{0}{\draw (0,1) -- (-1,2.2) node[xi] {};\draw (0,-0.25) node[xic] {} -- (0,1) ; \draw[kernels2] (0,1) node[not] {} -- (1,2.2) node[xi] {};
	\draw[kernels2] (0,1) {} -- (0,2.7) node[xic] {};
}

\DeclareSymbol{Xi4ca1}{0}{\draw (0,1) -- (-1,2.2) node[xic] {};\draw (0,-0.25) node[xi] {} -- (0,1) ; \draw[kernels2] (0,1) node[not] {} -- (1,2.2) node[xic] {};
\draw[kernels2] (0,1) {} -- (0,2.7) node[xi] {};
}

\DeclareSymbol{Xi4ca2}{0}{\draw (0,1) -- (-1,2.2) node[xi] {};\draw (0,-0.25) node[xi] {} -- (0,1) ; \draw[kernels2] (0,1) node[not] {} -- (1,2.2) node[xic] {};
\draw[kernels2] (0,1) {} -- (0,2.7) node[xic] {};
}

\DeclareSymbol{Xi4cap}{0}{\draw (0,1) -- (-1,2.2) node[xi] {};\draw (0,-0.25) node[not] {} -- (0,1) ; \draw[kernels2] (0,1) node[not] {} -- (1,2.2) node[xi] {};
\draw[kernels2] (0,1) {} -- (0,2.7) node[xi] {};
}

\DeclareSymbol{Xi3a}{0}{
 \draw (-1,1)  node[xi] {} -- (0,0); 
 \draw (0,0) node[xi] {}  -- (1,1) node[xi] {};
}

\DeclareSymbol{Xi4ebc1}{0}{
\draw[kernels2] (0,2) node[xi] {} -- (-1,1) ; \draw[kernels2] (-2,2)  node[xic] {} -- (-1,1) ; \draw (-1,1)  node[not] {} -- (0,0); 
 \draw (0,0) node[xic] {}  -- (1,1) node[xi] {};
}

\DeclareSymbol{Xi4ebc2}{0}{
\draw[kernels2] (0,2) node[xi] {} -- (-1,1) ; \draw[kernels2] (-2,2)  node[xi] {} -- (-1,1) ; \draw (-1,1)  node[not] {} -- (0,0); 
 \draw (0,0) node[xic] {}  -- (1,1) node[xic] {};
}

\DeclareSymbol{Xi2cbispex}{0}{\draw[kernels2] (0,1) -- (0.8,2.2) node[xi] {};\draw (0,-0.25) node[xie] {} -- (0,1); \draw[kernels2] (0,1) node[not] {} -- (-0.8,2.2) node[xi] {};}

\DeclareSymbol{Xi2cbis1p}{0}{\draw (0,1) -- (-0.8,2.2) node[xi] {};\draw (0,-0.25) node[not] {} -- (0,1) node[xi] {}; }

\DeclareSymbol{Xi2Xp}{-2}{\draw (0,-0.25) node[not] {} -- (-1,1) node[xix] {};} 

\DeclareSymbol{I1XiIXib}{0}{\draw  (0,-0.25) node[xi] {} -- (0,1) node[not] {};
\draw[kernels2] (0,1) -- (0,2.25) ; \draw (0,2.25) node[xi]{}; }

\DeclareSymbol{IXi2b}{0}{\draw  (0,-0.25) node[xi] {} -- (0,1) node[not] {};
\draw (0,1) -- (0,2.25) ; \draw (0,2.25) node[xi]{}; }

\DeclareSymbol{IXi2bex}{0}{\draw  (0,-0.25) node[xi] {} -- (0,1) node[xie] {};
\draw (0,1) -- (0,2.25) ; \draw (0,2.25) node[xi]{}; }

 \def\1{\mathbf{\symbol{1}}}

\def\eps{\varepsilon}

\DeclareSymbol{diff}{0}{
\draw (0,0.5) node[diff] {};
}

\DeclareSymbol{diff1}{0}{
\draw (0,0.5) node[diff1] {};
}

\DeclareSymbol{diff2}{0}{
\draw (0,0.5) node[diff2] {};
}

\DeclareSymbol{geo}{0}{
\draw (0,0) node[diff] {};
\draw (0.3,0) node[diff] {};
}

\DeclareSymbol{generic}{0}{
\draw (0,0.6) node[xi] {};
}

\DeclareSymbol{g}{0}{
\draw (0,0.6) node[g] {};
}

\DeclareSymbol{Ito}{0}{
\draw (0,0.6) node[xies] {};
}

\DeclareSymbol{Itob}{0}{
\draw (0,0.6) node[xiesf] {};
}

\DeclareSymbol{greycirc}{0}{
\draw (0,0.3) node[xi] {};
}

\DeclareSymbol{not}{0}{
\draw (0,0.6) node[not] {};
\draw[tinydots] (0,0.6) circle (0.8);
}

\DeclareSymbol{genericb}{0}{
\draw (0,0.6) node[xic] {};
}

\DeclareSymbol{bluecirc}{0}{
\draw (0,0.3) node[xic] {};
}

\DeclareSymbol{genericxix}{0}{
\draw (0,0.6) node[xix] {};
}

\DeclareSymbol{genericX}{0}{
\draw (0,0.6) node[X] {};
}

\DeclareSymbol{diffIto}{1}{
\draw  (0,2.5) -- (0,0) ;
\draw (0,-0.1) node[diff] {};
\draw (0,2.5) node[xies] {};
}
\DeclareSymbol{Itodiff}{2}{
\draw(0,2.9) -- (0,-0.2);
\draw (0,2.9) node[diff] {};
\draw (0,-0.1) node[xies] {};
}

\DeclareSymbol{diffgeneric}{1}{
\draw  (0,2.5) -- (0,0) ;
\draw (0,-0.1) node[diff] {};
\draw (0,2.5) node[xi] {};
}

\DeclareSymbol{genericdiff}{2}{
\draw(0,2.9) -- (0,-0.2);
\draw (0,2.9) node[diff] {};
\draw (0,-0.1) node[xi] {};
}

\DeclareSymbol{diffdot}{2}{
\draw  (0,3) -- (0,-0.1) ;
\draw (0,3) node[not] {};
\draw (0,-0.1) node[diff] {};
}

\DeclareSymbol{diffdotmini}{0}{
\draw  (0,0) -- (0,1.2) ;
\draw (0,1.2) node[not] {};
\draw (0,0) node[diffmini] {};
}

\DeclareSymbol{dotdiff}{2}{
\draw[kernelsmod]  (0,3) -- (0,-0.1) ;
\draw (0,3) node[diff] {};
\draw (0,-0.1) node[not] {};
}

\DeclareSymbol{3}{-2}{\draw (0,-0.25) node[xi] {} -- (-1,1) node[xi] {};}
\DeclareSymbol{AAA}{-0.5}{\draw (0,0) node[xi] {} -- (-1,1) node[xi] {} -- (0,2) node[xi] {};
	\draw (0,0) node[xi] {} -- (-1,-1) node[xi] {};}
\DeclareSymbol{AAM}{0.5}{\draw (-1,1)  -- (0,2) node[xi] {};
	\draw (-1,1) node[xi] {} -- (0,0) node[xi] {} -- (1,1) node[xi] {};}
\DeclareSymbol{AMA}{-4}{\draw (-1,1) node[xi] {} -- (0,0) node[xi] {} -- (1,1) node[xi] {};
	\draw (0,0) node[xi] {} -- (0,-1.4) node[xi] {};}
\DeclareSymbol{AMM}{0.5}{\draw (0,0)  -- (0,1.4) node[xi] {};
	\draw (-1,1) node[xi] {} -- (0,0) node[xi] {} -- (1,1) node[xi] {};}

\DeclareSymbol{dotdiff1}{2}{
\draw[kernelsmod]  (0,3) -- (0,-0.1) ;
\draw (0,3) node[diff1] {};
\draw (0,-0.1) node[not] {};
}

\DeclareSymbol{dotdiff1mini}{0}{
\draw[kernelsmod]  (0,1.2) -- (0,0) ;
\draw (0,1.2) node[diffmini] {};
\draw (0,0) node[not] {};
}

\DeclareSymbol{dotdiff2}{2}{
\draw (0,3) -- (0,-0.1) ;
\draw (0,3) node[diff] {};
\draw (0,-0.1) node[not] {};
}

\DeclareSymbol{dotdiff2mini}{0}{
\draw (0,1.2) -- (0,0) ;
\draw (0,1.2) node[diffmini] {};
\draw (0,0) node[not] {};
}

\DeclareSymbol{dotdiffstraight}{0}{
\draw  (0,3) -- (0,-0.1) ;
\draw (0,3) node[diff] {};
\draw (0,-0.1) node[not] {};
}

\DeclareSymbol{arbre1}{0}{
\draw  (0,0) -- (1.5,1.5) ;
\draw (1.5,1.5) node[not] {};
\draw (0,0) node[not] {};
}

\DeclareSymbol{arbre2}{0}{
\draw  (0,0) -- (1.5,1.5) ;
\draw[kernelsmod] (0,0) -- (-1.5,1.5);
\draw (1.5,1.5) node[not] {};
\draw (0,0) node[not] {};
\draw (-1.5,1.5) node[xi] {};
}

\DeclareSymbol{arbre3}{0}{
\draw  (0,0) -- (1.5,1.5) ;
\draw[kernelsmod] (1.5,1.5) -- (0,3);
\draw (0,0) node[not] {};
\draw (1.5,1.5) node[not] {};
\draw (0,3) node[xi] {};
}

\DeclareSymbol{treeeval}{0}{
\draw (0,0) -- (1,1);
\draw (0,0) node[xi] {};
\draw (1.25,1.25) node[xi] {};
\draw (-0.6,0.6) node[]{\tiny{$i$}};
\draw (0.65,1.85) node[]{\tiny{$j$}};
}

\DeclareSymbol{testeval}{0}{
\draw (0,0) -- (1,1);
\draw (0,0) -- (-1,1);
\draw (0,0) node[xi] {};
\draw (1.25,1.25) node[xi] {};
\draw (-1.25,1.25) node[xi] {};
\draw (-0.6,-0.6) node[]{\tiny{$i$}};
\draw (0.65,1.85) node[]{\tiny{$j$}};
\draw (-1.95,1.85) node[]{\tiny{$k$}};
}

\DeclareSymbol{treeeval2}{0}{
\draw[kernelsmod] (-0.25,-1) -- (1,0.5) ;
\draw[kernelsmod] (1,0.5) -- (-0.25,2);
\draw (1,0.5) node[diff2] {};
\draw (-0.25,-1) node[not] {};
\draw (-0.25,2) node[xi] {};
\draw (-0.6,1.2) node[]{\tiny{1}};
}

\DeclareSymbol{arbreact}{1}{
\draw (0,0) node[not] {};
\draw[kernelsmod] (0,0) -- (1,1);
\draw[kernelsmod] (0,0) -- (-1,1);
\draw (-1,1) node[xic] {};
\draw  (0,2) -- (1,1) ;
\draw (0,2) node[xic] {};
\draw (1,1) node[xi] {};
}

\DeclareSymbol{arbreact1}{0}{
\draw (0,-1.5) -- (0,0);
\draw[kernelsmod] (0,0) -- (1,1);
\draw[kernelsmod] (0,0) -- (-1,1);
\draw  (0,2) -- (1,1) ;
\draw (0,-1.5) node[diff] {};
\draw (0,0) node[not] {};
\draw (-1,1) node[xic] {};
\draw (0,2) node[xic] {};
\draw (1,1) node[xi] {};
}

\DeclareSymbol{arbreact2}{0}{
\draw (0,-0.75) -- (-1,0.5); 
\draw (0,-0.75) -- (1,0.5);
\draw (0,1.5) -- (1,0.5);
\draw (0,1.5) node[xic] {};
\draw (1,0.5) node[xi] {};
\draw (-1,0.5) node[xic] {};
\draw (0,-0.75) node[diff] {};
}

\DeclareSymbol{arbreact3}{0}{
\draw[kernelsmod] (0,-0.75) -- (-1,0.5); 
\draw[kernelsmod] (0,-0.75) -- (1,0.5);
\draw (0,1.75) -- (1,0.5);
\draw (2,1.75) -- (1,0.5);
\draw (0,1.75) node[xic] {};
\draw (1,0.5) node[diff] {};
\draw (-1,0.5) node[xic] {};
\draw (2,1.75) node[xi] {};
\draw (0,-0.75) node[not] {};
}

\DeclareSymbol{pre_im_1}{0}{
\draw[kernels2] (0,-0.5) node[not] {} -- (-0.6,0.5) ;
\draw[kernels2] (0,-0.5) -- (0.6,0.5);
\draw (0,1.1)  -- (-0.55,2);
\draw (0,1.1)  -- (0.55,2);
\draw (0,0.7) node[g] {};
\draw (0,2.2) node[g] {};
}

\DeclareSymbol{disconnect}{0}{
\draw[kernels2] (0,-0.5) node[not] {} -- (-0.6,0.5) ;
\draw[kernels2] (0,-0.5) -- (0.6,0.5);
\draw (-0.55,1.1)  -- (-0.55,2.3);
\draw (0.55,2.3) -- (0.55,1.5) -- (1.2,1.5) -- (1.2,3.5) -- (0.55,3.5) -- (0.55,2.7);
\draw (0,0.7) node[g] {};
\draw (0,2.5) node[g] {};
}

\DeclareSymbol{pre_im_2}{2}{\draw[kernels2] (0,0) node[not] {} -- (-1,1) node[not] {};
\draw[kernels2] (0,0) -- (1,1) node[not] {};
\draw[kernels2] (-1,1) -- (-1.5,2.5);
\draw[kernels2] (-1,1) -- (-0.5,2.5);
\draw[kernels2] (1,1) -- (0.5,2.5);
\draw[kernels2] (1,1) -- (1.5,2.5);
\draw (-1,2.7) node[g] {};
\draw (1,2.7) node[g] {};
}

\DeclareSymbol{CX_rec}{0}{
\draw [black] (-0.3,1) to (-0.3,-0.3);
\draw [black] (0.3,1) to (0.3,-0.3);
\draw [black] (-0.3,1) to (-0.3,2.3);
\draw [black] (0.3,1) to (0.3,2.3);
\draw (0,1) node[rec] {};
}

\DeclareSymbol{CX_cerc}{0}{
\draw [black] (0,1) to (0,-0.3);
\draw (0,1) node[cerc] {};
}

\DeclareSymbol{proof0}{0}{
\draw (0,-3) node[] {\tiny$\tau_0$};
\draw (0,-2.3) -- (0,0.5);`
\draw (0,0.5) -- (-1.5,2.5);
\draw (0,0.5) -- (1.5,2.5);
\draw (-1,-2) node[] {\tiny$v$};
\draw (-1.5,3.1) node[] {\tiny$\tau_1$};
\draw (1.5,3.1) node[] {\tiny$\tau_2$};
\draw (0,0.5) node[diff] {};
}

\DeclareSymbol{proof0b}{0}{
\draw (0,0.5) -- (-1.5,2.5);
\draw (0,0.5) -- (1.5,2.5);
\draw (-1.5,3.1) node[] {\tiny$\tau_1$};
\draw (1.5,3.1) node[] {\tiny$\tau_2$};
\draw (0,0.5) node[diff] {};
}

\DeclareSymbol{proof}{0}{
\draw[kernelsmod] (-2,3) -- (0,0);
\draw[kernelsmod] (2,3) -- (0,0);
\draw (0,0) node[not] {};
}

\DeclareSymbol{prooftri}{0}{
\draw[kernelsmod] (-2,3) -- (0,0);
\draw[kernelsmod] (0,4) -- (0,0);
\draw (2,2.7)--(0,0);
\draw (0,0) node[not] {};
}

\DeclareSymbol{proofqua}{0}{
\draw[kernelsmod] (-3,3) -- (0,0);
\draw[kernelsmod] (-1,4) -- (0,0);
\draw (1,3.6)--(0,0);
\draw (3,2.7)--(0,0);
\draw (0,0) node[not] {};
}

\DeclareSymbol{proof1_1}{0}{
\draw (0,-2.7) node[] {\tiny$\tau_0$};
\draw (0,-2) -- (0,0.5);`
\draw[kernelsmod] (0,0.5) -- (-1.5,2.5); 
\draw[kernelsmod] (0,0.5) -- (1.5,2.5);
\draw (-1,-1.7) node[] {\tiny$v$};
\draw (-1.5,3.1) node[] {\tiny$\tau_1$};
\draw (1.5,3.1) node[] {\tiny$\tau_2$};
\draw (0,0.5) node[not] {};
}

\DeclareSymbol{proof1b_1}{0}{
\draw[kernelsmod] (0,0.5) -- (-1.5,2.5); 
\draw[kernelsmod] (0,0.5) -- (1.5,2.5);
\draw (-1.5,3.1) node[] {\tiny$\tau_1$};
\draw (1.5,3.1) node[] {\tiny$\tau_2$};
\draw (0,0.5) node[not] {};
}

\DeclareSymbol{proof1_2}{0}{
\draw (0,-2.7) node[] {\tiny$\tau_0$};
\draw[kernelsmod] (0,-2) -- (0,0.5);`
\draw[kernelsmod] (0,0.5) -- (-1.5,2.5); 
\draw[kernelsmod] (0,0.5) -- (1.5,2.5);
\draw (-1,-1.7) node[] {\tiny$v$};
\draw (-1.5,3.1) node[] {\tiny$\tau_1$};
\draw (1.5,3.1) node[] {\tiny$\tau_2$};
\draw (0,0.5) node[not] {};
}

\DeclareSymbol{proof2_1}{0}{
\draw (0,-2.7) node[] {\tiny$\tau_0$};
\draw (0,-2) -- (-1.8,-0.3);
\draw (-1.8,0.7) -- (0,2.5); 
\draw (1,-1.7) node[] {\tiny$v$};
\draw (-2.5,1.2) node[] {\tiny$r_1$};
\draw (0,3.1) node[] {\tiny$\tau_2$};
\draw (-2.5,0) node[] {\tiny$\tau_1$};
}

\DeclareSymbol{proof2b_1}{0}{
\draw (-1.8,0.7) -- (0,2.5); 
\draw (-2.5,1.2) node[] {\tiny$r_1$};
\draw (0,3.1) node[] {\tiny$\tau_2$};
\draw (-2.5,0) node[] {\tiny$\tau_1$};
}

\DeclareSymbol{proof2b_2}{0}{
\draw (-1.8,0.7) -- (0,2.5); 
\draw (-2.5,1.2) node[] {\tiny$r_2$};
\draw (0,3.1) node[] {\tiny$\tau_1$};
\draw (-2.5,0) node[] {\tiny$\tau_2$};
}

\DeclareSymbol{proof2_2}{0}{
\draw (0,-2.7) node[] {\tiny$\tau_0$};
\draw[kernelsmod] (0,-2) -- (-1.8,-0.3);
\draw (-1.8,0.7) -- (0,2.5); 
\draw (1,-1.7) node[] {\tiny$v$};
\draw (-2.5,1.2) node[] {\tiny$r_1$};
\draw (0,3.1) node[] {\tiny$\tau_2$};
\draw (-2.5,0) node[] {\tiny$\tau_1$};
}

\DeclareSymbol{proof3_1}{0}{
\draw (0,-2.7) node[] {\tiny$\tau_0$};
\draw (0,-2) -- (1.8,-0.3);
\draw (1.8,0.7) -- (0,2.5); 
\draw (-1,-1.7) node[] {\tiny$v$};
\draw (2.5,1.2) node[] {\tiny$r_2$};
\draw (0,3.1) node[] {\tiny$\tau_1$};
\draw (2.5,0) node[] {\tiny$\tau_2$};
}

\DeclareSymbol{proof3_2}{0}{
\draw (0,-2.7) node[] {\tiny$\tau_0$};
\draw[kernelsmod] (0,-2) -- (1.8,-0.3);
\draw (1.8,0.7) -- (0,2.5); 
\draw (-1,-1.7) node[] {\tiny$v$};
\draw (2.5,1.2) node[] {\tiny$r_2$};
\draw (0,3.1) node[] {\tiny$\tau_1$};
\draw (2.5,0) node[] {\tiny$\tau_2$};
}

\DeclareSymbol{proof4_1}{0}{
\draw (0,-2) node[] {\tiny$\tau_0$};
\draw (0,-1.3) -- (-1.5,1.8); 
\draw  (0,-1.3) -- (1.5,1.8);
\draw (-1,-1) node[] {\tiny$v$};
\draw (-1.5,2.4) node[] {\tiny$\tau_1$};
\draw (1.5,2.4) node[] {\tiny$\tau_2$};
}

\DeclareSymbol{proof4_2}{0}{
\draw (0,-2) node[] {\tiny$\tau_0$};
\draw[kernelsmod] (0,-1.3) -- (-1.5,1.8); 
\draw  (0,-1.3) -- (1.5,1.8);
\draw (-1,-1) node[] {\tiny$v$};
\draw (-1.5,2.4) node[] {\tiny$\tau_1$};
\draw (1.5,2.4) node[] {\tiny$\tau_2$};
}

\DeclareSymbol{proof4_3}{0}{
\draw (0,-2) node[] {\tiny$\tau_0$};
\draw (0,-1.3) -- (-1.5,1.8); 
\draw[kernelsmod]  (0,-1.3) -- (1.5,1.8);
\draw (-1,-1) node[] {\tiny$v$};
\draw (-1.5,2.4) node[] {\tiny$\tau_1$};
\draw (1.5,2.4) node[] {\tiny$\tau_2$};
}

\DeclareSymbol{prooftriple}{0}{
\draw (0,-2.7) node[] {\tiny$\tau_0$};
\draw (0,-2) -- (0,0.25);`
\draw (0,0.25) -- (-1.5,2.5); 
\draw (0,0.25) -- (1.5,2.5);
\draw (0,0.25) -- (0,2.5);
\draw (-1,-1.7) node[] {\tiny$v$};
\draw (-2,3.1) node[] {\tiny$\tau_1$};
\draw (0,3.1) node[] {\tiny$\tau_2$};
\draw (2,3.1) node[] {\tiny$\tau_3$};
\draw (0,0.25) node[diff] {};
}

\DeclareSymbol{prooftriple1}{0}{
\draw (0,-2.7) node[] {\tiny$\tau_0$};
\draw (0,-2) -- (0,0.25);
\draw[kernelsmod] (0,0.25) -- (-1.5,2.5); 
\draw (0,0.25) -- (1.5,2.5);
\draw[kernelsmod] (0,0.25) -- (0,2.5);
\draw (-1,-1.7) node[] {\tiny$v$};
\draw (-2,3.1) node[] {\tiny$\tau_1$};
\draw (0,3.1) node[] {\tiny$\tau_2$};
\draw (2,3.1) node[] {\tiny$\tau_3$};
\draw (0,0.25) node[not] {};
}

\DeclareSymbol{prooftripleperm1}{0}{
\draw (0,-2.7) node[] {\tiny$\tau_0$};
\draw (0,-2) -- (0,0.25);
\draw[kernelsmod] (0,0.25) -- (-3.5,2.5); 
\draw (0,0.25) -- (2.5,2.5);
\draw[kernelsmod] (0,0.25) -- (0,2.5);
\draw (-1,-1.7) node[] {\tiny$v$};
\draw (-3,3.1) node[] {\tiny$\tau_{\sigma_1}$};
\draw (0,3.1) node[] {\tiny$\tau_{\sigma_{\!2}}$};
\draw (3,3.1) node[] {\tiny$\tau_{\sigma_3}$};
\draw (0,0.25) node[not] {};
}

\DeclareSymbol{proofdouble}{0}{
\draw (0,0.5) -- (-1.5,2.5);
\draw (0,0.5) -- (1.5,2.5);
\draw (-1.5,3.1) node[] {\tiny$\tau_1$};
\draw (1.5,3.1) node[] {\tiny$\tau_2$};
\draw (0,0.5) node[diff] {};
}

\DeclareSymbol{proofquadruple}{0}{
\draw (0,0) -- (-2.5,2.5); 
\draw (0,0) -- (-1,2.5);
\draw (0,0) -- (1,2.5);
\draw (0,0) -- (2.5,2.5);
\draw (-3,3.1) node[] {\tiny$\tau_1$};
\draw (-1,3.1) node[] {\tiny$\tau_2$};
\draw (1,3.1) node[] {\tiny$\tau_3$};
\draw (3,3.1) node[] {\tiny$\tau_4$};
\draw (0,0) node[diff] {};
}

\DeclareSymbol{proofquadruple1}{0}{
\draw[kernelsmod] (0,0) -- (-2.5,2.5); 
\draw[kernelsmod] (0,0) -- (-1,2.5);
\draw (0,0) -- (1,2.5);
\draw (0,0) -- (2.5,2.5);
\draw (-3,3.1) node[] {\tiny$\tau_1$};
\draw (-1,3.1) node[] {\tiny$\tau_2$};
\draw (1,3.1) node[] {\tiny$\tau_3$};
\draw (3,3.1) node[] {\tiny$\tau_4$};
\draw (0,0) node[not] {};
}

\DeclareSymbol{proofquadrupleperm1}{0}{
\draw[kernelsmod] (0,-0.4) -- (-3.5,2.3); 
\draw[kernelsmod] (0,-0.4) -- (-1,2.3);
\draw (0,-0.4) -- (1,2.5);
\draw (0,-0.4) -- (3.5,2.5);
\draw (-4.5,3.1) node[] {\tiny$\tau_{\sigma_1}$};
\draw (-1.5,3.1) node[] {\tiny$\tau_{\sigma_2}$};
\draw (1.5,3.1) node[] {\tiny$\tau_{\sigma_3}$};
\draw (4.5,3.1) node[] {\tiny$\tau_{\sigma_4}$};
\draw (0,-0.4) node[not] {};
}

\DeclareMathAlphabet{\mathpzc}{OT1}{pzc}{m}{it}

\let\d\partial
\let\eps\varepsilon

\def\eqref#1{(\ref{#1})}

\def\geo{{\text{\rm\tiny geo}}}

\DeclareMathOperator{\Lie}{Lie}
\DeclareMathOperator{\LA}{LieAdm}
\DeclareMathOperator{\ComMag}{ComMag}
\DeclareMathOperator{\PL}{PreLie}

\DeclareMathOperator{\Perm}{Perm}

\DeclareMathOperator{\Tw}{Tw}
\DeclareMathOperator{\MC}{MC}

\DeclareMathOperator{\NT}{\nabla{}Trees}

\DeclareMathAlphabet{\mathbbold}{U}{bbold}{m}{n}

\def\kk{\mathbbold{k}}

\makeatletter 
\newcommand*{\bigcdot}{}
\DeclareRobustCommand*{\bigcdot}{%
  \mathbin{\mathpalette\bigcdot@{}}%
}
\newcommand*{\bigcdot@scalefactor}{.5}
\newcommand*{\bigcdot@widthfactor}{1.15}
\newcommand*{\bigcdot@}[2]{%
  \sbox0{$#1\vcenter{}$}
  \sbox2{$#1\cdot\m@th$}%
  \hbox to \bigcdot@widthfactor\wd2{%
    \hfil
    \raise\ht0\hbox{%
      \scalebox{\bigcdot@scalefactor}{%
        \lower\ht0\hbox{$#1\bullet\m@th$}%
      }%
    }%
    \hfil
  }%
}
\makeatother

\def\act{\bigcdot}

\tcbset
{colframe=boxcolor,colback=symbols!7!pagebackground,coltext=pageforeground,
fonttitle=\bfseries,nobeforeafter,center title,size=fbox,boxsep=1.5pt,
top=0mm,bottom=0mm,boxsep=0mm,tcbox raise base}

\def\two{{\<generic>\kern0.05em\<genericb>}}
\def\twoI{{\<Ito>\kern0.05em\<Itob>}}

\begin{document}

\title{Chain rule symmetry for singular SPDEs}
\author{Yvain Bruned$^1$, Vladimir Dotsenko$^2$}
\institute{ 
	IECL (UMR 7502), Universit\'e de Lorraine
	\and IRMA (UMR 7501), Universit\'e de Strasbourg \\
	Email:\ \begin{minipage}[t]{\linewidth}
		\texttt{yvain.bruned@univ-lorraine.fr},
		\\ \texttt{vdotsenko@unistra.fr}.
\end{minipage}}

\maketitle

\begin{abstract}
We characterise the chain rule symmetry for the geometric stochastic heat equations in the full subcritical regime for Gaussian and non-Gaussian noises. We show that the renormalised counter-terms that give a solution invariant under changes of coordinates are generated by iterations of covariant derivatives. 
The result was known only for space-time white noises, with a very specific proof that so far could not be extended to the general case. The key idea of the present paper is to change the perspective on several levels and to use ideas coming from operad theory and homological algebra. Concretely, we introduce the operad of Christoffel trees that captures the counter-terms of the renormalised equation; our main new insight is to describe the space of invariant terms homologically, using a suitable perturbation of the differential of the operadic twisting of that operad. As a consequence, we obtain the correct renormalisation for the quasi-linear KPZ equation in the subcritical regime completing the programme started by Hairer and Gerencser. Previously, the main algebraic tool used in the study of singular SPDEs were Hopf algebras of decorated trees; our work shows that operad theory and homological algebra add new powerful tools with immediate  applications to open problems that were out of reach by other methods.
\end{abstract}

\setcounter{tocdepth}{2}
\tableofcontents

\section{Introduction}

\subsection{Motivation, context, and methods: the SPDE side}
For stochastic differential equations (SDEs) with smooth coefficients driven by independent Brownian motions, one can get two notions of solution depending on the choice of the stochastic product. The first one called Itô solution guarantees that one gets the so called Itô Isometry for this product. The second one is the Stratonovich solution as one gets  equivariance under changes of coordinates.
One cannot get the two symmetries at the same time in this finite dimensional case as there are not enough degrees of freedom. Indeed, one starts with a one parameter family of solutions out of which only two points are relevant (one point for Itô and one point for Stratonovich) giving only one symmetry at a time. 

In the infinite dimensional case, things are somehow different. It has been shown in \cite{BGHZ} that the two symmetries can cohexist for geometric stochastic heat equations of the form
\begin{equ}[e:genClass_intro]
	\d_t u^\alpha = \d_x^2 u^\alpha + \Gamma^\alpha_{\beta\gamma}(u)\,\d_x u^\beta\d_x u^\gamma
	+ K^\alpha_\beta(u)\,\d_x u^\beta
	+h^\alpha(u) + \sigma_i^\alpha(u)\, \xi_i\;,
\end{equ}
where $ i \in \lbrace 1,...,m \rbrace $ and the functions
$\Gamma^\alpha_{\beta\gamma},\sigma_i^\alpha:\mathbb{R}^d\to\mathbb{R}$ with $\Gamma^\alpha_{\beta\gamma}=\Gamma^\alpha_{\gamma\beta}$ are smooth. 
Here $u:\mathbb{R}_+\times \mathbb{T} \mapsto\mathbb{R}^d$ and the $ \xi_i $ are independent space-time white noises. Itô Isometry and chain rule give a natural choice of solution for \eqref{e:genClass_intro} which is a stochastic partial differential equation (SPDE). Equation \eqref{e:genClass_intro} is motivated from a geometric context when one sees $\Gamma$ as the Christoffel symbols for an arbitrary connection on $\mathbb{R}^d$ and,
for each $i$, the $(\sigma_i^\alpha)_\alpha$ as the components of a vector field on~$\mathbb{R}^d$. It provides a natural stochastic
process taking values in the space of loops in a compact Riemannian manifold. Its invariant measure is expected to be the Brownian loop measure. 

This equation was first considered with coloured noise in space in \cite{F92}. Having a space-time white noise transforms this equation into a singular SPDEs with distributional products. This requires the use of recent techniques such as the theory of Regularity Structures in order to provide a notion of solution. Actually, one of the reasons for which Regularity Structures were invented was for treating this equation. 
Before, for $d=1$ and only one noise, some simple versions of the equation were considered with rough paths techniques such as Burgers type equations ($\Gamma_{\beta \gamma}^{\alpha} = 0, \sigma_i^\alpha(u) = \sigma$) in \cite{Rough11} and the KPZ equation ($\Gamma_{\beta \gamma}^{\alpha} = 1,\sigma_i^\alpha(u) = 1$)  in \cite{KPZ}. With the Regularity Structures black box developed in \cite{reg,BHZ,CH,BCCH}, one is able to solve \eqref{e:genClass_intro} and to produce a renormalised equation where the space-time white noises are replaced by a regularised version $ \xi_i^{\eps} $ that converges to $ \xi_i $ when $\eps$ is sent to zero. The renormalised equation is parametrised by $54$ renormalisation constants which gives a finite dimensional space of solutions. For a review on the theory of Regularity Structures see \cite{MR4174393,BH21}.
   
The work \cite{BGHZ} selects a natural solution out of this finite dimensional space by using the symmetries of the chain rule and the Itô Isometry. This result was partially annouced in 2016 via the proceeding \cite{hairer2016motion}. The uniqueness of the solution relies on a precise dimension counting of the vector space associated to the two symmetries and it is mostly done by hand preventing any type of generalisation without the use of a new framework. By generalisation, we mean to consider other type of noises that could be non-Gaussian described by cumulants and/or more singular than the space-time white noise but still in the subcritical regime. These new noises produce a bigger finite dimensional space parametrising the renormalised equation.
   
In the present work, we focus on the chain rule symmetry and we provide a full characterisation of this space. Informally, our main result can be described in the next theorem. We suppose that the $ \xi_i $ are  independent identically distributed  noises satisfying the assumptions of convergence given in \cite{CH}. These noises could be Gaussian or non-Gaussian described by their cumulants. We denote by $ \xi_i^{\eps} $ their regularisation (see Section~\ref{geometric SPDEs} for a precise definition of the regularisation).
\begin{theorem}
   	\label{thm:main renormalisation_intro}
   	Let $u_0^{\alpha}\in\CC^r(\mathbb{T})$ for some $r>0$. 	there exist renormalisation constants $ C_{\eps}(\tau) $ such that
   	the renormalised equation of \eqref{e:genClass} is given by:
   	\begin{equs}[eq:renorm nonlocal intro1]
   		\d_t u^\alpha_{\eps} & = \d_x^2 u^\alpha_{\eps} + \Gamma^\alpha_{\beta\gamma}(u_{\eps})\,\d_x u^\beta_{\eps}\d_x u^\gamma_{\eps}
   		+ K^\alpha_\beta(u_{\eps})\,\d_x u^\beta_{\eps}
   		+h^\alpha(u_{\eps}) + \sigma_i^\alpha(u_{\eps})\, \xi_i^{\eps}\; \\ & + \sum_{\tau \in \VV_{\xi}} C_{\eps}(\tau)  \Upsilon_{\Gamma,\sigma}[\tau](u_{\eps})\,.
   	\end{equs}
   	where $ \VV_{\xi}  $ is a combinatorial set whose dimension as a vector space can be computed.  It depends on the choice of the noises $ \xi_i $ and the   $ \Upsilon_{\Gamma,\sigma}[\tau](u_{\eps}) $ are computed using the vector fields $ \sigma_i $ and 
   	the covariant derivative $\nabla_X Y$ defined for two vector fields $X,Y$ by  
   	\begin{equ}
   		(\nabla_X Y)^\alpha (u) = X^\beta(u)\,\d_\beta Y^\alpha(u) + \Gamma^\alpha_{\beta\gamma}(u) \,X^\beta(u) \,Y^\gamma(u)\;.
   	\end{equ}
   	The solution $u_\eps$ of the random PDEs \eqref{eq:renorm nonlocal intro1} converges as $\eps\to 0$ in probability, locally in time, to a nontrivial limit $u$.  The equations \eqref{eq:renorm nonlocal intro1} transform according to the chain rule under composition with diffeomorphisms.
\end{theorem}
   
This theorem is a consequence of Theorem \ref{thm:main renormalisation_general} and Theorem \ref{main_theorme_chain_rule}. As a corollary, we are able to obtain a solution theory in the full subcritical regime for a quasi-linear version of \eqref{e:genClass_intro}, completing the programme started in \cite{GH19,G20,BGN}
   
\begin{theorem}
   	\label{thm:main renormalisation_intro_quasi}
   	Let $u_0^{\alpha}\in\CC^r(\mathbb{T})$ for some $r>0$ and $ a : \mathbb{R}^d \rightarrow \mathbb{R} $ smooth such that $a$ takes values in $[\lambda,\lambda^{-1}]$ for some $\lambda>0$.	There exist smooth functions $ c \mapsto C^{c}_{\eps}(\tau) $ such that for
   	the renormalised equation:
   	\begin{equs}[eq:renorm nonlocal intro2]
   		\d_t u^\alpha_{\eps} & = a(u_{\eps})\d_x^2 u^\alpha_{\eps} + \Gamma^\alpha_{\beta\gamma}(u_{\eps})\,\d_x u^\beta_{\eps}\d_x u^\gamma_{\eps}
   		+ K^\alpha_\beta(u_{\eps})\,\d_x u^\beta_{\eps}
   		+h^\alpha(u_{\eps}) + \sigma_i^\alpha(u_{\eps})\, \xi_i^{\eps}\; \\ & + \sum_{\tau \in \VV_{\xi}} C^{a(u_{\eps})}_{\eps}(\tau)  \Upsilon_{\Gamma,\sigma}[\tau](u_{\eps})\,
   	\end{equs}
   	the solution $u_\eps$ of the random PDEs \eqref{eq:renorm nonlocal intro2} converges as $\eps\to 0$ in probability, locally in time, to a nontrivial limit $u$.  The equations \eqref{eq:renorm nonlocal intro2} transform according to the chain rule under composition with diffeomorphisms.
\end{theorem}
   
The strategy developed in \cite{BGHZ} for getting Theorem~\ref{thm:main renormalisation_intro} for space-time white noises was to characterise the combinatorial objects that give the chain rule as the kernel of a linear map denoted by  $  \hat \phi_\geo $. These combinatorial objects are actually decorated trees. Then, the proof mostly performed by hand was divided into two steps:
  \begin{itemize}
  	\item Find independent linear relations for $ \ker \hat \phi_\geo  $.
  	\item Compute the dimension of a subspace of decorated trees constructed from combinatorial covariant derivatives.
  \end{itemize}
One was able to conclude via a carefull dimension counting. It is easy to see that such a strategy cannot be extended to more singular noises as the number of decorated trees describing the renormalisation grows very fast as the number of vertices increases. The key idea of the present paper is to change the perspective on several levels. First, we interpret the kernel of $\hat \phi_\geo$ as the degree zero homology of a huge chain complex, which then allows us to use powerful tools coming from homological algebra to show that the homology of that complex is concentrated in degree zero. (This means that elements of higher homological degree, even though useful for the strategy of the proof, do not carry any extra useful information.) Next, we note that, while the SPDE considerations mean that we should identify some of the trees that we consider with one another, that identification, amounting to taking coinvariants of some finite groups, can be done before or after computing the homology, and doing it after computing the homology means that we use one extra tool, to harness the problem, namely the theory of operads. That allows us to do just one universal homology computation, from which then we can derive the answer in every particular case, specialising from an operad to a free algebra on a certain number of generators and taking (co)invariants.

\subsection{Motivation, context, and methods: the algebraic side}

In general, the notion of naturality with respect to the chain rule symmetries corresponds to the celebrated programme of classification of invariant differential operators initiated by Veblen \cite{zbMATH02573963} almost 100 years ago; in the case of invariant differential operators acting on vector fields and connections, a version of classification was obtained in 1950s by Schouten \cite{MR0066025}. In 1970s, Kirillov indicated that classification of invariant differential operators on the affine line can be interpreted in terms of cohomology of the Lie algebra of formal vector fields, see, e.g., \cite[p.~7]{MR0611158}. Around the same time, in cohomology computations for the Lie algebra of vector fields on an affine space $V$ of sufficiently large dimension, use of trees and other graphs to describe $GL(V)$-invariants was pioneered by Gelfand and Fuchs (for instance, one finds ``the readers will find it easier to understand this formula by means of the pictures which I. M . Gelfand and I used to represent the functionals $\Psi_r$ and other similar functionals'' on \cite[p.~81]{MR0874337}). A couple of decades later, the two stories were successfully brought together by Markl \cite{MR2503978} who described a rigorous graph complex formalism allowing one to determine, in sufficiently large dimensions, all invariant differential operators of the given type. 

The algebraic story presented in this paper is unravelled as follows. Using the combinatorics of trees that appear from the black box of Regularity Structures, we define a new natural operad of decorated trees, which we call the operad of Christoffel trees, and denote by $\NT$. Then, one has from  Proposition~\ref{operad_christoffel} 
\begin{equs}
\NT \cong \PL \vee \ComMag  
\end{equs}
where $ \PL $ is the operad of pre-Lie algebras, $\ComMag$ is the operad of commutative magmatic algebras and $ \vee $ is the coproduct of operads. We then define the universal map $  \hat \Phi_\geo $ on the level of the operad of Christoffel trees. To interpret that map conceptually, we make use of the so called operadic twisting \cite{MR3348138,MR3299688,MR4621635} that defines, for an operad $\calP$ concentrated in degree zero that is equipped with a map of operads 
$f\colon \Lie\to\calP$, a differential graded operad $\Tw(\calP)$ by
\[
\Tw(\calP)=\left(\calP\vee\kk\alpha, d_{\Tw}=d_{\MC}+\mathrm{ad}_{\ell_1^{\alpha}} \right)
\]
where
\[
d_{\MC}(\alpha)=-\frac12[\alpha,\alpha], \quad \ell_1^{\alpha}(a_1)=[\alpha,a_1],  
\]
and $\mathrm{ad}_{\ell_1^{\alpha}}(\mu)=\ell_1^{\alpha}\circ_1\mu-(-1)^{|\mu|}\sum_i\mu\circ_i\ell_1^{\alpha}$. Here we denote by $[-,-]$ the image of the generator of $\Lie$ in $\calP$ under the map~$f$. Here, $ \alpha $ is called a Maurer-Cartan element. In fact, we use a slight modification of this formalism: we construct in Proposition~\ref{correspondence_ker_phi_0}, a perturbation of the differential $d_{\Tw}$ such that the operad 
	\begin{equs} \label{operad_twisting}
	(\PL\vee\ComMag\vee \kk\alpha, d_{\MC}+\mathrm{ad}_{\ell_1^{\alpha}}+d_0)
	\end{equs}
has $\ker\hat\Phi_{\geo}$ as its the degree zero homology. Furthermore, we establish in Theorem~\ref{th:twisting} that the homology of the operad \eqref{operad_twisting} is concentrated in homological degree zero and is isomorphic to the operad $\LA$ of Lie-admissible algebras. This is done by an argument involving spectral sequences \cite[Ch.~5]{MR1269324}; one may say that the spectral sequence argument makes precise sense of the statement that $d_0$ is an ``insignificant'' perturbation of $d_{\Tw}$ that does not affect the size of the homology. This gives our main algebraic result:
\begin{theorem}\label{th:LA-operadic}
We have an operad isomorphism $ \ker\hat{\Phi}_{\geo} \cong \LA$. Moreover, that isomorphism identifies $\ker\hat{\Phi}_{\geo}$ as the linear span of iterations of covariant derivatives. 
\end{theorem}
It is probably fair to say that Theorem \ref{th:LA-operadic} could have been discovered many times in the past years, but, surprisingly, does not appear in the literature in its precise form. In particular, the coproduct of operads appears in \cite[Prop.~7.4]{MR2503978} where not necessarily torsion-free connections are studied, while the case of a torsion-free connection is discussed in great detail in the sequel \cite{MR2817591} by Jany\v{s}ka and Markl, where however the exposition seems to be guided by wishing to compare the results with the abovementioned classification of Schouten, and so Lie-admissible algebras do not appear. At the same time, the claim that Lie-admissible algebras is precisely the structure one obtains on vector fields in the presence of a torsion-free non necessarily flat connection appears in various places in the literature \cite{MR2032454,munthekaas2023lie,MR4057606}, but no proof of that claim has ever been given. It is however worth mentioning that operadic twisting was recently applied in a similar context by Laubie \cite{laubie2024hypertrees} in his proof of a conjecture of the second author of the present paper identifying the operad of F-manifold algebras inside the operad of the so called ComPreLie algebras as the kernel of an appropriate derivation; moreover, our operad of Christoffel trees has a superficial similarity with the operad of Greg trees defined in previous work of Laubie \cite{laubie2023combinatorics}.

From Theorem \ref{th:LA-operadic}, one is able, via appropriate specialisations of our universal formulae, to obtain descriptions of $\ker\hat\varphi_{\geo}$ for Gaussian subcritical noises in Corollary~\ref{cor:gaussian-dim}  and for non-Gaussian subcritical noises in Corollary~\ref{cor:cumulant-dim}  that will allow us to describe the renormalised equation in Theorem~\ref{thm:main renormalisation_intro} with the chain rule symmetry. We also have an explicit way to compute the dimensions of those spaces using appropriate generating series from the ring of symmetric functions \cite{MR3443860}.

\subsection{Outline of the paper}

Let us outline the paper by summarising the content of its sections. In Section~\ref{sec::2}, we start by introducing in full details the geometric stochastic heat equations. We recall the main theorem for the renormalised equation (see Theorem~\ref{thm:main renormalisation}) that can be obtained from the black box offered by the theory of Regularity Structures \cite{reg,BHZ,CH,BCCH}. We then present the set $\SS_4$ of decorated trees in \eqref{e:SS}
that is used for parametrising the renormalisation when one looks at space-time white noises. These decorated trees have at most four noise type nodes that come in pairs. This is due to the fact that the renormalisation constants are constructed via the expectation of some stochastic iterated integrals and the Gaussianity of the space-time white noise allows us to use Wick formula in \eqref{wick_formula}. We also introduce the elementary differentials $ \Upsilon_{\Gamma,\sigma}$ in \eqref{recursive_Upsilon} which is the second crucial component for describing the renormalised equation.

We explain how the set $ \SS_4 $ evolves by looking at Gaussian noises more singular than space-time white noise but still in the subcritical regime which is a H\"older regularity in space-time greater than $-2$. We define the spaces $ \SS^g_{2n} $ as the same as for $\SS_4$ but now with at most $2n$ nodes. We also consider the case when Gaussianity is removed and replaced by cumulants expansion (see \eqref{cumulant_formula} for the computation of the expectation via cumulants). This leads to the introduction of the sets $ \SS^c_{n} $ that contain decorated trees with at most $n$ noise type nodes with these nodes being partitioned. We finish the section with Theorem~\ref{thm:main renormalisation_general} which is a version of Theorem~\ref{thm:main renormalisation} for more general noises.

In Section~\ref{appendix_A}, we start by recalling the notion of a combinatorial species. We proceed with recalling the definitions of two products of species, the Cauchy product and the composition product. The composition product is used to define (symmetric) operads. We recall the two kinds of generating series one can associate to a linear species, the exponential generating function for dimensions and the generating series for characters of symmetric groups. We also recall the explicit construction of the coproduct of augmented operads, which we use in a meaningful way in the main algebraic result of this paper. Finally, we  give a short overview of the two aspects of homotopical algebra for operads that we use, the Koszul duality theory and operadic twisting.

In Section~\ref{sec::3}, we introduce the main definition for the chain rule symmetry (see Definition~\ref{geo_terms}) wich allows us to extract subspaces of the previous sets that correspond to geometric elements. This will guarantee that the renormalised equation is invariant under changes of coordinates. 
We describe the characterisation of this space in the specific case of $ \SS_4 $ (see Theorem~\ref{thm_covariant_derivatives}) which was obtained in~\cite{BGHZ}. A basis of this vector space is given by covariant derivatives defined in \eqref{eq:covderiv} and \eqref{covariant_derivative_semi_linear}. Then, one is able to refine Theorem~\ref{thm:main renormalisation} by proposing renormalisation constants such that now the solutions are now invariant under diffeomorphisms (see Theorem~\ref{chain:rule}).
 For extending the chain rule result, one has to go through a more abstract formalism. We introduce the operad of Christoffel trees that encompasses the decorated trees used for the renormalisation. We first define in Definition~\ref{Christoffel trees} the species of Christoffel trees, which we denote by $\NT$. Then, we show in Proposition~\ref{operad_christoffel} that the linear span of these trees  can be turned into an operad with the appropiate rule for computing the operad compositions. This operad is shown to be isomorphic to the coproduct of the operads $\PL$ and $\ComMag$ denoted by $\PL \vee \ComMag$, which is crucially used in the sequel.

Our main result is Theorem~\ref{main_theorme_chain_rule} that gives a full characterisation of the chain rule symmetry in the full subcritical regime for Gaussian and non-Gaussian noises. This boils down to explicitly describing the vector spaces 
$	\CS_{\tiny{\geo},2n}^g $ and  $\CS_{\tiny{\geo},n}^c$ expressed in terms of covariant derivatives on decorated trees. In order to prove such a result one has to look at the kernel of a linear map $ \hat \phi_\geo $ given in \eqref{def_phi_geo_1} and \eqref{def_phi_geo_2}. To make a precise connection between the map $\hat \phi_\geo$ and the differential arising in the context of the operadic twisting, we introduce a map $\hat\Phi_{\geo}$, which is the universal version of $\hat \phi_\geo$ on the operad level. In Proposition~\ref{correspondence_ker_phi_0}, we establish that $ \ker \hat \Phi_\geo  $ is isomorphic to the degree zero homology of the operad $\PL\vee\ComMag\vee\kk\alpha$ equipped  with a suitable differential $ \mathrm{d}$. Here $ \alpha $ is an extra constant (arity zero element) that is a Maurer--Cartan element for the Lie algebra structure corresponding to the Lie bracket inside the operad $\PL$. In the main algebraic result of this paper in Theorem~\ref{th:twisting}, we show that the homology of the previous differential graded operad is isomorphic to the operad $\LA$ of Lie-admissible algebras.
Then, one has a full characterisation of $ \CS_{\tiny{\geo},2n}^g $  in Corollary~\ref{cor:gaussian-dim} with a way to compute its dimension in Section~\ref{dimension_counting_sec}. A similar precise characterisation is obtained for $ \CS_{\tiny{\geo},n}^c $  in Corollary~\ref{cor:cumulant-dim}. These results show Theorem~\ref{main_theorme_chain_rule}.

In Section~\ref{sec::4}, we consider a quasi-linear version of the main equation given in \eqref{e:genClass quasi}. We recall the main result of \cite{BGN}, Theorem~\ref{thm:main_quasi_KPZ_4}, that says that the chain rule symmetry allows us to get local counter-terms for the renormalised equation (see also  
 \eqref{S_4_parameter}). We are able to extend this result to the full subcritical regime in Theorem~\ref{thm:main renormalisation_intro_quasi} by using Theorem~\ref{main_theorme_chain_rule}.

\subsection{Open questions arising from our work}
\label{open_problems}
Let us conclude the introduction by outlining several natural questions raised by the results we established.

One interesting question is to incorporate a version of the Itô Isometry in the picture. In \cite{BGHZ}, the role of the Itô Isometry has been well understood for space-time white noise. By a precise dimension counting, one is able to single out one natural solution in a specific geometric context. It is natural to ask whether it is possible to get this isometry in the full subcritical regime for Gaussian noises, and to use it for some uniqueness result.  In the non-Gaussian case, it has no clear meaning.

Another very intriguing question is that of global solutions. Theorems \ref{thm:main renormalisation_intro_quasi} and \ref{thm:main renormalisation_intro} provide only local solution in time. Global solutions have been obtained in~\cite{BGN} by using the chain rule symmetry. The idea is to perform some type of Cole Hopf transform in order to move from the generalised KPZ equation to a stochastic multiplicative heat equation. Then, by checking that some renormalisation constants are orthogonal to the covariant derivative, one can transfer the long time results of this equation to the generalised KPZ equation. Such a strategy could work for the subcritical regime as soon as one has a clear understanding of the long time behaviour of the associated stochastic heat equation. But this result is very specific to dimension one ($d=1$) as it is not clear how to find the generalised Cole Hopf transform in higher dimension. 

One may also wonder if the language of decorated trees is optimal for the given low dimension. Its advantage is that is captures all dimensions at the same time in a universal way, but it has been observed that one gets an overparametrisation in small dimension by using decorated trees, see, e. g. \cite[Rem. 1.9]{BGHZ}. The chain rule symmetry is still valid (see Remark \ref{remark_over_paremetrisation}) but one may want to assure that $ \Upsilon_{\Gamma,\sigma} $ is a bijection. In high dimension, decorated trees and their pre-Lie structures are the best choice. For lower dimension, it is not clear how to fully characterise the kernel of $ \Upsilon_{\Gamma,\sigma}$ except in dimension one where multi-indices \cite{OSSW,MR4537770,LOTT,BL23} and their Novikov structure specified in \cite{BD23} could be used. In numerical analysis multi-indices B-series introduced \cite{BEFH24} are the most natural expansions in dimension one as they characterise uniquely affine equivariant methods, a specific instance  of a general result involving Aromatic series in  \cite{MR3451427}. 

Finally, let us comment on the structure of the vector space $\ker\hat\phi_{\geo}$ in a given low dimension. First of all, it is known that there some exotic operations on vector fields in given low dimensions that are invariant with respect to automorphisms, for example, for $N=d^2+2d-2$, there is a diffeomorphism-invariant ``$N$-commutator'' that computes out of $N$ vector fields on $\mathbb{R}^d$ another vector field, see \cite{MR2081725}, and there are some other examples of similar kind in concrete finite dimensions \cite{MR2383581}. Such examples will inevitably lead to elements in $\ker\hat\phi_{\geo}$ that cannot be obtained as iterations of covariant derivatives; in fact, they will be combinations of trees without thick edges. This means that the elementary differentials associated to those trees will not depend on the Christoffel symbols. However, there are no exotic operations known for $d=1$, leading us to the following

\begin{conjecture}
Consider the universal map ${\Phi}_{\geo}^{\mathrm{Novikov}}$ induced by ${\Phi}_{\geo}$ after taking the operad quotient $\PL\twoheadrightarrow\mathrm{Novikov}$ in the construction given by our approach. Then $\ker\hat{\Phi}_{\geo}^{\mathrm{Novikov}}$ is the linear span of iterations of covariant derivatives. 
\end{conjecture}

This conjecture would mean that some part of our result is, exceptionally, still valid in dimension $d=1$. However, the homological method of this paper relies on knowing the homology of the operadic twisting of the operad of pre-Lie algebras. For the operad of Novikov algebras, the homology of operadic twisting is not known, and for sure is not concentrated in degree zero, so our strategy will not work \emph{mutatis mutandis}, and it remains to be seen if a similar proof can be furnished.

\subsection*{Acknowledgements}

{\small
	Y. B. gratefully acknowledges funding support from the European Research Council (ERC) through the ERC Starting Grant Low Regularity Dynamics via Decorated Trees (LoRDeT), grant agreement No.\ 101075208. V. D. is funded by the ANR project HighAGT (ANR-20-CE40-0016), and by Institut Universitaire de France. V. D. is grateful to Paul Laubie for several useful discussions.
}

\section{Background from SPDEs}
\label{sec::2}
 
 \subsection{Geometric stochastic heat equations} 
 \label{geometric SPDEs}
 Let smooth functions
$\Gamma^\alpha_{\beta\gamma},\sigma_i^\alpha:\mathbb{R}^d\to\mathbb{R}$ with $\Gamma^\alpha_{\beta\gamma}=\Gamma^\alpha_{\gamma\beta}$. The 
Greek indices run over $\{1,\ldots,d\}$
while Roman indices run over $\{1,\ldots,m\}$, with $m$ being the number of driving noises.
We consider the equation
\begin{equ}[e:genClass]
	\d_t u^\alpha = \d_x^2 u^\alpha + \Gamma^\alpha_{\beta\gamma}(u)\,\d_x u^\beta\d_x u^\gamma
	+ K^\alpha_\beta(u)\,\d_x u^\beta
	+h^\alpha(u) + \sigma_i^\alpha(u)\, \xi_i\;,
\end{equ}
where $u:\mathbb{R}_+\times \mathbb{T} \to\mathbb{R}^d$, and the $ \xi_i$, $i=1,\ldots,m$, are independent space-time noises. We interpret $\Gamma$ as the Christoffel symbols for an arbitrary connection on $\mathbb{R}^d$ and,
for each $i$, the $(\sigma_i^\alpha)_\alpha$ as the components of a vector field on $\mathbb{R}^d$. 

We introduce  a class of mollifiers denoted $ \mathrm{Moll} $ which is the set of all compactly supported smooth functions $ \varrho : \mathbb{R}^2 \rightarrow \mathbb{R}  $ integrating to $1$, such that
$\varrho(t, -x) = \varrho(t, x)$, and such that $\varrho(t, x) = 0$ for $t \leq 0 $ (i.e. is non-anticipative). For $ \varepsilon > 0 $, we replace $ \xi_i $ by its regularisation
$ \xi^{\varepsilon}_i = \varrho_{\varepsilon} * \xi_i $, the space-time convolution of the noise $ \xi_i $ 
with $ \varrho_{\varepsilon} $ given by:
\begin{equs}
	\varrho_{\varepsilon} = \varepsilon^{-3} \varrho(\eps^{-2}t,\eps^{-1}x)
\end{equs}
where we have used the parabolic scaling $(2,1)$ for the rescaling. To describe the renormalised version of Equation \eqref{e:genClass} obtained using the general machinery of Regularity Structures \cite{reg,BHZ,CH,BCCH}, one uses integrals of the type  
\begin{equs} \label{iterated_integrals_four_indices}
I = \sum_{i,j, k, \ell =1}^m I_{ijk\ell} =  \sum_{i,j, k, \ell =1}^m 
  \xi^{\eps}_i K * ( (K * \xi^{\eps}_j) (\partial_x K * \xi^{\eps}_k) (\partial_x K * \xi^{\eps}_{\ell})).
\end{equs}
where we are summing over four different indices. Here,  the kernel $K$ is such that it appears in a fix decomposition $P = K +R$  of the heat kernel $ P $ on the real line.  The kernel $K$ is even in the spatial variable, integrates to zero, and is compactly supported
in a neighbourhood of the origin and the remainder $R$ that is globally smooth. The reason for using such terms as $I$ is coming from renormalisation as correction terms need to be added to the right hand side of \eqref{e:genClass} after regularisation of the noises $ \xi_i $ in order to make sense of the various distributional products. Indeed, one way to renormalise $I$ is to subtract its mean~$\mathbb{E}(I)$. (In practice, one needs a bit more, implementing a version of the BPHZ renormalisation algorithm.) The $ \xi^{\eps}_i $ is a centered ($\mathbb{E}(\xi^{\eps}_i) = 0 $) Gaussian noise such that the $ \xi^{\eps}_i $ are i.i.d, this implies that one has for $ i \neq j $, $ z, \bar{z} \in \mathbb{R}_+ \times \mathbb{T}$ 
\begin{equs}
  \mathbb{E}(\xi^{\eps}_i(z) \xi^{\eps}_j(\bar{z})) = \mathbb{E}(\xi^{\eps}_i(z)) \mathbb{E}( \xi^{\eps}_j(\bar{z}))  =0,  
  \end{equs}
and therefore for every $i,j$
\begin{equs}
    \mathbb{E}(\xi^{\eps}_i(z) \xi^{\eps}_j(\bar{z})) = \delta_{i,j} \mathbb{E}(\xi^{\eps}_i(z) \xi^{\eps}_i(\bar{z})) .
  \end{equs}
where $ \delta_{i,j} $ is the Kronecker delta. First, one can notice that $I$ is polynomial in the Gaussian noises $ \xi_i $ which encourages us to use Gaussian Calculus via the Wick formula given for a product of random Gaussian variables $ \prod_{i \in J} X_{i} $, where $J$ is a finite set, by
\begin{equs} \label{wick_formula}
\mathbb{E}[\prod_{i \in J} X_{i}] = \sum_{\pi \in \mathcal{P}_2(J)} \prod_{(i,j) \in \pi} \mathbb{E}[X_i X_j]
\end{equs}
where $\mathcal{P}_2(J)$ are all the possible pairings of $ J $ and a pairing $ \pi $ is a partition of disjoint pairs of $ J $.
Applying this formula, one gets
\begin{equs} \label{identity_e}
  \mathbb{E}(I_{ijk\ell}) = \delta_{i,j} \delta_{k,\ell}  \tilde{\mathbb{E}}(I_{iikk}) + \delta_{i,k} \delta_{j,\ell} \tilde{\mathbb{E}}(I_{ijij}) +  \delta_{i,\ell} \delta_{j,k}  \tilde{\mathbb{E}}(I_{ijji}).
\end{equs}
where $ \tilde{\mathbb{E}}(I_{iikk}) $ is a short hand notation for saying that we take the expectation $ \mathbb{E}(\xi^{\eps}_i(z_1) \xi^{\eps}_i(z_2)) \mathbb{E}(\xi^{\eps}_k(z_3) \xi^{\eps}_{k}(z_4)) $ inside $ I_{iikk} $ where the $ z_{n} $ are some integration variables. We shall now explain how this leads to using a certain type of decorated trees to organise the correction terms. We consider rooted trees with vertices of $m$ different colours corresponding to the regularised noises $\xi^{\varepsilon}_i$, and edges of two types corresponding to 
\begin{equs} \label{interpretation_renormalisation}
  \<thin> \equiv K * \cdot, \quad   \<thick> \equiv \partial_x K * \cdot, 
\end{equs} 
where $ * $ is the space-time convolution. The edge type $  \<thick>$ in \eqref{interpretation_renormalisation} will be denoted as thick edge in the sequel. These decorated trees are generated by a certain rule which is used to construct the Regularity Structure associated to the equation (see \cite{BHZ} for a precise exposition).
The rule  $ R $ is given by
\begin{equ}
  R(\<thin>) = \{(\<thin>^k,\<generic>_i), (\<thick>^2,\<thin>^k)\,:\, k \ge 0, \, i \in I \}\;
\end{equ}
where $ I $ is a set of generators. Moreover, our trees will correspond to sums like the one in \eqref{iterated_integrals_four_indices}, so that one has, for example,
\begin{equs}
   \<Xi2s> & \equiv \sum_{i =1}^m \xi^{\eps}_i K* \xi^{\eps}_i, \quad  \<I1Xitwos> \equiv \sum_{i=1}^m (\partial_x K * \xi^{\eps}_i)^2, \\ \<Xi4ca1s> &  \equiv \sum_{i,j =1}^m 
   \xi^{\eps}_i K * ( (K * \xi^{\eps}_j) (\partial_x K * \xi^{\eps}_i) (\partial_x K * \xi^{\eps}_j)).
\end{equs}
This convention corresponds to some identifications on the level of trees. For example, in \eqref{identity_e}, the second and the third term are the same if we permute $ \ell $ and $k$. In general, if a tree contains the same number of vertices of two different colours $c_1$ and $c_2$, we declare it to be equal to the tree obtained from it by swapping $c_1$ with $c_2$, for example, 
\begin{equs} \label{identification_trees}
  \<generic> =  \<genericb> \, , \quad  \<Xi2s> = \<Xi2bs>, \, \quad \<I1Xitwos> = \<I1Xitwobs> \, , \quad  \<Xi4ca1s> = \<Xi4ca1bs>.
\end{equs}
Note that since our noises are Gaussian, we must consider only trees with an even number of noises as expectation of an odd monomial is equal to zero. 

Among the trees that we consider, only some are meaningful for us. This is due to some power counting and symmetry considerations. The power counting comes from H\"older regularity considerations as follows. One defines on each decorated tree a degree map denoted by $\text{deg}$. If we assume that the noises are space-time white noises, then their space-time trajectories belong to $  \mathcal{C}^{\alpha}$ where $ \alpha = -\frac{3}{2} - \kappa $ for every~$ \kappa > 0 $. Therefore, one postulates that the degree of each vertex of a tree is equal to~$-\frac{3}{2} - \kappa$. Then, the Schauder estimates given by the space-time convolution with the heat kernel gives a $ + 2 $ gain in H\"older regularity and a space derivative introduces a loss~$ -1 $, so that
$\text{deg}(\<thin>) = 2$, $\text{deg}(\<thick>)= 1$. The degree of a decorated tree is the sum of the degrees of its noises nodes and edges. For example, we have
\begin{equs}
  \text{deg}(  \<Xi2s> ) & = 2 \, \text{deg}(\<generic>) + \text{deg}(\<thin>) = 2 (-\frac{3}{2}- \kappa) + 2 = -1 - 2 \kappa, \\
  \text{deg}( \<Xi4ca1s>  ) & = 2 \,  \text{deg}(\<generic>) +  2 \,  \text{deg}(\<genericb>) + 2 \, \text{deg}(\<thick>)+ 2 \, \text{deg}(\<thin>) =  - 4 \kappa. 
\end{equs}
Only the decorated trees of negative degree whose expectation does not vanish will provide a contribution to the renormalised equation of \eqref{e:genClass}.
Additionally, it turns out that the decorated trees we are interested in may have $0$ or $2$ thick edges attached to each node. Indeed, since the variable $ \partial_x u $ appears in the right hand side of \eqref{e:genClass} under the form $ \d_x u^\beta\d_x u^\gamma $, performing a perturbative expansion, one can get at most two thick edges attached to a node. It is not possible to have just one thick edge because the renormalisation constants are zero due to antisymmetry (odd number of derivatives $\partial_x$), see \cite[Lem. 2.5]{BGHZ}. For the same reason other decorated trees in \cite[Sec. 2.1]{BGHZ} are disregarded. These are the trees with extra node decorations that encode monomials of the form $ X^n $ inside an iterated integral.
 
Now, if we consider trees with at most four nodes, there are exactly $54$ trees fulfilling all the constraints we imposed, namely
\begin{equs}[e:SS]
  {} & \<Xi2s>  \,, \<I1Xitwos> \,,  \<Xi4_1s> \,, \<Xi4c1s>\,, \<Xi4_2s>\,, \<cI1Xi4as> \,,   \<I1Xi4ac1s> \,, \<I1Xi4ac2s> \,, \<cI1Xi4bs> \,, \<I1Xi4bc1s> \,, \<cI1Xi4cs> \,,  \<I1Xi4cc1s> \,, \<I1Xi4cc2s>     \,, \<cI1Xi4abs> \,, \<I1Xi4abc1s> \,,  \<cI1Xi4bcs> \,, \<I1Xi4bcc1s> \,, \<cI1Xi4acs> \,, \<I1Xi4acc1s> \,,   \<I1Xi4acc2s> \,,
  \\ & \<I1Xi4abcc2s> \,, \<I1Xi4abcc1s> \,,   \<2I1Xi4c2s> \,,    \<2I1Xi4c1s> \,,  \<2I1Xi4bc3s> \,,  \<2I1Xi4bc1s> \,,  \<2I1Xi4bc2s> \,,  \<2I1Xi4cc2s> \,,  \<2I1Xi4cc1s> \,, \<Xi4b1s> \,, \<Xi4ba1s> \,, \<Xi4ba2s> \,, \<Xi41s> \,, \<Xi42s> \,,   
  \<Xi4cbc1s>\,,  \<Xi4cbc2s>\,, \\ & \<Xi4ca1s> \,, \<Xi4ca2s> \,,    \<Xi4cabc1s>\,, \<Xi4cabc2s>\,, \<Xi4ec3s> \,,  \<Xi4ec1s> \,,  \<Xi4ec2s> \,, \<Xi4eac2s> \,,    \<Xi4eac1s> \,,  \<Xi4eabisc3s> \,, \<Xi4eabisc1s> \,, \<Xi4eabisc2s> \,, \<Xi4ebc2s> \,,  \<Xi4ebc1s> \,,\<Xi4eabc2s> \,,  \<Xi4eabc1s> \,, \<Xi4eabbisc2s> \,, \<Xi4eabbisc1s>\;.
\end{equs}
This set $\SS_4$ of decorated trees was first considered in \cite[Sec. 2.4]{BGHZ}, where they were called reduced trees. We denote by $\CS_4$ \label{CS page ref2} the vector space with $\SS_4$ as a basis. 

Let us define a map  $\Upsilon_{\Gamma,\sigma}$ that turns each decorated tree into a vector field on $\mathbb{R}^d$. Specifically,  we set \label{Evaluationmap1 page ref}
\begin{equs}\label{recursive_Upsilon}
  \begin{aligned}
    \big(\Upsilon_{\Gamma,\sigma}\tau\big)^{\alpha}(u,q) 
    & = \sum_{\gamma:{\bf N}\to\{1,\ldots,m\}} \sum_{\beta:{\bf E}\to\{1,\ldots,d\}} \prod_{v \in V_{T}}
    \\ & \left[\Big(\prod_{e \in \mathbb{E}^{+}_{\graftI }(v) } \partial_{u^{\beta_e}}\Big)\Big( \prod_{e \in \mathbb{E}^{+}_{\graftID }(v)} \partial_{q^{\beta_e}}\Big)
    \bigl(\bar \Upsilon^{\beta_{e_v},\gamma}_{\Gamma,\sigma}(v)\bigr)(u,q)\right]
  \end{aligned}
\end{equs}
where 
\begin{itemize}
  \item ${\bf E}$ is the set of edges $e\in E_T$ of type $\<thick>$ or $\<thin>$ and ${\bf N}$ is the set of noise type nodes. 
  \item For $v$ with an incoming edge of type $\<generic>$, we set 
  $\bar \Upsilon^{\beta,\gamma}_{\Gamma,\sigma}(v)(u,q) = \sigma_{\gamma(v)}^\beta(u)$, otherwise
  \begin{equ}[e:substGamma]
    \bar \Upsilon^{\beta,\gamma}_{\Gamma,\sigma}(v)(u,q) = \Gamma^\beta_{\gamma \eta}(u)\, q^\gamma q^\eta\;.
  \end{equ}
  \item $\mathbb{E}^{+}_{\graftI }(v)$ and $\mathbb{E}^{+}_{\graftID }(v)$ are the sets of edges with 
  decorations $\<thin>$ and $\<thick>$ respectively coming into $v\in V_T$.
  \item We use the convention $\beta_{e_v} =\alpha$ for $v = \rho_T$, the root of $T$.
\end{itemize}
For instance, one has
\begin{equs}
  \big( \Upsilon_{\Gamma, \sigma}(\<Xi4ca1>)\big)^{\alpha}
  & = \sum_{i,j = 1}^m 2 \partial_{\zeta} \sigma_i^\alpha(u)   \d_\eta \Gamma^\zeta_{\beta\gamma}(u)\,\sigma_j^\eta(u)\,\sigma_i^\gamma(u)\, \sigma_j^\beta(u)\;.
\end{equs}
For a decorated tree $ \tau $, we denote by $ S(\tau) $ the number of tree automorphisms of $\tau$. 
\newpage
One can observe that the decorated trees are used in two different ways:
\begin{itemize}
	\item One way is for computing the renormalised constants via its interpretation as a stochastic iterated integral and the expectation. Therefore, one can give the interpretation \eqref{interpretation_renormalisation} for the edges via space-time convolution with some kernels.
	\item The other way is for getting the expression of some elementary differentials where now the edges are interpreted as some derivatives according to the variables of the equation: $ u $ and $ \partial_x u $ (the later is denoted by $q$ in \eqref{recursive_Upsilon}).
\end{itemize}
These two interpretations allow to compute the following Butcher-type series that is needed for the renormalised equation:
\begin{equs}
	\sum_{\tau \in \SS_4} C_{\eps}(\tau) \frac{ \Upsilon_{\Gamma,\sigma}[\tau](u_{\eps})}{ S(\tau)}.
\end{equs}

Using the notation we introduced, we can finally state a result one obtains, in the case where the $ \xi_i $ are space-time white noises, by applying the general machinery of Regularity Structures \cite{reg,BHZ,CH,BCCH}.
\begin{theorem}\label{thm:main renormalisation}
	Let $u_0^{\alpha}\in\CC^r(\mathbb{T})$ for some $r>0$. For every $ \varrho \in \mathrm{Moll} $, $\eps>0$,	there exist renormalisation constants $ C_{\eps}(\tau) $ such that
	the renormalised equation of \eqref{e:genClass} is given by:
	\begin{equs}[eq:renorm nonlocal1]
		\d_t u^\alpha_{\eps} & = \d_x^2 u^\alpha_{\eps} + \Gamma^\alpha_{\beta\gamma}(u_{\eps})\,\d_x u^\beta_{\eps}\d_x u^\gamma_{\eps}
		+ K^\alpha_\beta(u_{\eps})\,\d_x u^\beta_{\eps}
		+h^\alpha(u_{\eps}) + \sigma_i^\alpha(u_{\eps})\, \xi_i^{\eps}\; \\ & + \sum_{\tau \in \SS_4} C_{\eps}(\tau) \frac{ \Upsilon_{\Gamma,\sigma}[\tau](u_{\eps})}{ S(\tau)}\,.
	\end{equs}
	That is, the solution $u_\eps$ of the random PDEs \eqref{eq:renorm nonlocal1} converges as $\eps\to 0$ in probability, locally in time, to a nontrivial limit $u$. 
\end{theorem}

\subsection{Full subcritical regime and cumulants}

The result stated in the previous section changes if we consider a noise that is different from the space-time white noise. There are two options. One is to consider Gaussian noises that are different from the space-time white noise, the H\"older regularity $ \alpha $ of the noise tells us which set we have to consider for the renormalised equation. If $ \alpha = - 2 $, then decorated trees with negative degree form an infinite set. For every $ \alpha > -2 $, one obtains a finite set, and this range is referred to as the full subcritical regime. The other option is to replace Gaussian noises by non-Gaussian noises described by their cumulants.

If the H\"older regularity of the noise changes, we consider the same trees as before, with the number of vertices ensuring that we obtain a tree of negative degree. As above, only the trees with an even number of vertices of each colour and with at most two vertices of the same colour are needed. Moreover, we make the exact same identification as in \eqref{identification_trees} where the choice of the colour does not matter. We denote the set of decorated trees generated in this way by $ \SS^g $, and its linear span by $\CS^g$. For $ n \in \mathbb{N} $, we denote by  $ \SS^g_{2n} $ the subset of $ \SS^g $ with at most $ 2n  $ noises; we denote by $\CS^g_{2n}$ its linear span. The exponent $g$ is to stress that the noises $ \xi_i^{\eps} $ are Gaussian. Note that for $n=2$, the set $ \SS^g_{2n} $ is the set $ \SS_4 $ discussed above.
 
Let us describe the set-up corresponding to the other option, that is to considering noises described by their cumulants. If one 
considers a collection of random variables   $ \lbrace X_i, \, i \in  S\rbrace $ for some finite index set $S$. 
 For the subsets $J\subseteq S$, we write $X_J=\{X_i : i\in J\}\text{ and }X^J=\prod_{i\in J}X_i$. Further, we write $\CP(J)$ for the set of all partitions of $J$.
 The cumulant $\mathbb{E}_c(X_J)$ is defined inductively over $|J|$ by $\mathbb{E}_c(X_J)=\mathbb{E}(X_i)$, if $J$ is the singleton containing $i$ and
	\begin{equs} \label{cumulant_formula}
	\mathbb{E}(X^J)=\sum_{\pi\in\CP(J)}\prod_{\bar{J}\in\pi}\mathbb{E}_c(X_{\bar{J}}),\qquad\text{if }|J|\ge 2.
	\end{equs}
Now, if we assume that the $ \xi_i $ are i.i.d, centred, described by their cumulants, one has
\begin{equs}
 \label{identity_e_cumulants}
 \begin{aligned}
	\mathbb{E}(I_{ijk\ell}) &= \delta_{i,j} \delta_{k,\ell}	\tilde{\mathbb{E}}_c(I_{iikk}) + \delta_{i,k} \delta_{j,\ell}	\tilde{\mathbb{E}}_c(I_{ijij}) +	\delta_{i,\ell} \delta_{j,k}	\tilde{\mathbb{E}}_c(I_{ijji}) 
\\ &	+ \delta_{i,j,k,\ell} 	\tilde{\mathbb{E}}_c(I_{iiii})
\end{aligned}
\end{equs}
where $\tilde{\mathbb{E}}_c$ has the same interpretation as for $ \tilde{\mathbb{E}} $ but with the expectations $ \mathbb{E} $ replaced by the cumulants $ \mathbb{E}_c $.
One has to add another decorated tree:
\begin{equs}
\<Xi4cas> \equiv \	\delta_{i,j,k,\ell} 	\tilde{\mathbb{E}}_c(I_{iiii}).
\end{equs}
whose elementary differential is given by
\begin{equs}
		\big( \Upsilon_{\Gamma, \sigma}\<Xi4ca>)\big)^{\alpha}
		 = \sum_{i= 1}^m 2 \partial_{\zeta} \sigma_i^\alpha(u)   \d_\eta \Gamma^\zeta_{\beta\gamma}(u)\,\sigma_i^\eta(u)\,\sigma_i^\gamma(u)\, \sigma_i^\beta(u)\;
\end{equs}
 which has to be compared with \eqref{e:substGamma}. 
When one considers cumulants, one has to look at partitions of the noise type nodes. As before, we make an identification where we allow swapping colours of parts of partition with the same number of elements. We denote the decorated trees generated in this way by $ \SS^c $, and its linear span by $\CS^c$. Moreover, for every $ n \in \mathbb{N} $, $ n \geq 1 $, we denote by $ \SS^c_{n} $ the subset of $ \SS^c $ with at most $ n  $ noises, and we denote its linear span by $\CS^c_n$. The exponent $  c$ is to stressed that the noises $ \xi_i^{\eps} $ are described by cumulants. 

Let us stress that Theorem~\ref{thm:main renormalisation} is still true when the regularity of noises $\xi_i$ changed satisfying the assumptions in \cite{CH}. This implies the replacement of $ \SS_4 $ by one of the sets mentioned above that we will denote by $ \SS_{\xi}$. 

\begin{theorem}\label{thm:main renormalisation_general}
	Let $u_0^{\alpha}\in\CC^r(\mathbb{T})$ for some $r>0$. For every $ \varrho \in \mathrm{Moll} $, $\eps>0$,	there exist renormalisation constants $ C_{\eps}(\tau) $ such that
	the renormalised equation of \eqref{e:genClass} is given by:
	\begin{equs}[eq:renorm nonlocal2]
		\d_t u^\alpha_{\eps} & = \d_x^2 u^\alpha_{\eps} + \Gamma^\alpha_{\beta\gamma}(u_{\eps})\,\d_x u^\beta_{\eps}\d_x u^\gamma_{\eps}
		+ K^\alpha_\beta(u_{\eps})\,\d_x u^\beta_{\eps}
		+h^\alpha(u_{\eps}) + \sigma_i^\alpha(u_{\eps})\, \xi_i^{\eps}\;, \\ & + \sum_{\tau \in \SS_{\xi}} C_{\eps}(\tau) \frac{ \Upsilon_{\Gamma,\sigma}[\tau](u_{\eps})}{ S(\tau)}\,.
	\end{equs}
	That is, the solution $u_\eps$ of the random PDEs \eqref{eq:renorm nonlocal2} converges as $\eps\to 0$ in probability, locally in time, to a nontrivial limit $u$. 
\end{theorem}

\section{Background from operad theory}\label{appendix_A}

Let us give a short summary of operadic notions and results that we use in this paper. We refer the reader to~\cite{MR1629341} for more details on species and operations on them, to~\cite{MR3443860} for information about symmetric functions, to~\cite{MR2954392} for more details on operads and Koszul duality, and to~\cite{MR4621635} for more details on operadic twisting.

\subsection{Combinatorial species}

The notion of a \emph{combinatorial species} was introduced by Joyal \cite{MR0633783} in order to produce a categorification of formal power series and various operations on them. A combinatorial species of structures (or simply a species) is a rule $\calF$ that defines, for each  finite set $S$, a finite set $\calF(S)$ called the set of $\calF$-structures on $S$. The only condition that should be satisfied is that this rule should be ``canonical'', that is, compatible with bijections between finite sets: if $\imath\colon S_1\to S_2$ is a bijection, we should have a bijection $\calF(\imath)\colon\calF(S_1)\to\calF(S_2)$, and these bijections must be coherent, that is $\calF(\imath_1\circ\imath_2)=\calF(\imath_1)\circ\calF(\imath_2)$. A category theorist would say that a species is an endofunctor of the groupoid (category where only bijections are valid morphisms) of finite sets. 

As an example, one can consider the species $\mathrm{RT}$ of rooted tree structures on $S$, for which $\mathrm{RT}(S)$ is the set of all possible rooted trees on the vertex set $S$. There is also a very trivial example of the \emph{species of sets} $E$ given by $E(S)=\{S\}$ (which basically means that there is just one structure of a set on any set), and an even more trivial example of a \emph{singleton species} $E_1$ for which $E_1(S)=\varnothing$ unless $|S|=1$, in which case $E_1(S)=\{S\}$. The species $E_k$ of $k$-element sets is defined similarly.

Each species $\calF$ has its (exponential) generating function
 \[
f_{\calF}(t)=\sum_{n\ge 0}\frac{|\calF(\{1,\ldots,n\})|}{n!}t^n.
 \]
Two most natural operations one has on power series are product and substitution. They correspond to two very natural operations on the level of species. 

The \emph{Cauchy product} of two species $\calF_1$ and $\calF_2$ is defined by the formula
 \[
(\calF_1\cdot\calF_2)(S):=\bigsqcup_{S=S_1\sqcup S_2}\calF_1(S_1)\times\calF_2(S_2).
 \]
In plain words, a $(\calF_1\cdot\calF_2)$-structure on a set $S$ is obtained by partitioning $S$ into two disjoint parts $S_1$ and $S_2$ and then defining separately an $\calF_1$-structure on $S_1$ and an $\calF_2$-structure on $S_2$.

Algebraically, substitution $f_1(f_2(t))$ of one power series into another makes sense when the series $f_2(t)$ has no constant term (for otherwise infinite sums appear). For that reason, the \emph{composition product} of two species $\calF_1$ and $\calF_2$ is defined if $\calF_2(\varnothing)=\varnothing$, and is given by the formula
 \[
(\calF_1\circ\calF_2)(S)
=\bigsqcup_{S=S_1\sqcup \cdots\sqcup S_n}\bigl(\calF_1(\{S_1,\ldots,S_n\})\times\calF_2(S_1)\times\cdots\times \calS_2(S_n)\bigr)/\Sigma_n.
 \]
In plain words, a $(\calF_1\circ\calF_2)$-structure on a set $S$ is obtained by forming a set partition of $S$ into several disjoint nonempty parts $S_1$,\ldots, $S_n$, and then defining separately a $\calF_1$-structure on the set of parts of that partition, and an $\calF_2$-structure on each of the parts. 

Let us illustrate these notions on the examples of species we mentioned: We shall prove the equality
\begin{equation}
\mathrm{RT}=E_1\cdot (E\circ\mathrm{RT}).
\end{equation}
Due to the definition of the Cauchy product, the right hand side of this equation, defines a species of structures on $S$ for which one has to partition $S$ into two parts, and impose an $E_1$-structure on the first part and an $E(\mathrm{RT})$-structure on the second part. An $E_1$-structure on the first part exists only if the first part consists of a single element. Defining an $E\circ\mathrm{RT}$-structure on the complement of that element amounts to partitioning that complement into several non-empty parts and defining a rooted tree structure on each part. It is now clear that the result $E_1\cdot (E\circ\mathrm{RT})$ is precisely the species $\mathrm{RT}$, since one can construct a rooted tree structure on a set $S$ by first choosing the root label, and then partitioning the remaining vertices into several parts, choosing rooted tree structures on those parts, and grafting those trees at the chosen root. 

\subsection{Linear species}

Examining the definitions above, one notices that one can give these definitions in more generality, saying that $\calF$ assumes values in a symmetric monoidal category that has all coproducts and coequalisers; one just needs to replace $\times$ by the monoidal structure, and the quotient by $\Sigma_n$ by the corresponding coequaliser. Two examples that will be used in this paper are the symmetric monoidal category of $\kk$-vector spaces and the symmetric monoidal category of $\kk$-chain complexes; we shall call the corresponding versions of species \emph{linear species} and \emph{differential graded species} (or \emph{dg species}). For differential graded species, the following K\"unneth formula holds:
\begin{equation}
H_\bullet(\calF_1\circ\calF_2)\cong H_\bullet(\calF_1) \circ H_\bullet(\calF_2).
\end{equation}

It is important to note that the composition product $\circ$ makes the category of species (or its generalisations) into a (very non-symmetric) monoidal category, with the monoidal unit being the singleton species $E_1$, which in this context will be denoted by $\mathbbold{1}$. A \emph{(symmetric) operad} is a monoid in that monoidal category, that is a triple $(\calO,\gamma,\eta)$, where $\gamma\colon\calO\circ\calO\to\calO$ is the product and $\eta\colon \mathbbold{1}\to\calO$ is the unit, which satisfy the usual axioms of a monoid in a monoidal category \cite{MR0354798}. This notion was introduced independently by Artamonov \cite{MR0237408} and May \cite{MR0420610}. It will be useful for us that one can alternatively give axioms of an operad in terms of so called ``partial compositions'' 
 \[
\circ_\star\colon\calO(S_1\sqcup\{\star\})\otimes\calO(S_2)\to\calO(S_1\sqcup S_2),
 \]
which correspond to substitutions according to the composition product where most of arguments used are equal to the operad unit (and $\star$ indicates the argument where a nontrivial substitution is made). In order for these operations to define an operad structure, it is necessary and sufficient that they satisfy the so called sequential and parallel axioms \cite{MR2954392}, which are usually easier to check than the full associativity of composition.

Most of operads that one studies, are studied because they ``control'' important algebras. To a type of algebras $\calA$ (such as associative algebras, Lie algebras etc.) one can associate an operad as follows. Consider the free algebra of type $\calA$ generated by a finite set $S$, and consider the vector space $\calA(S)$ spanned by elements where each generator from $S$ is used once. These subspaces form a linear species, and, moreover, they form an operad, since we can perform substitutions of elements of that kind into one another, giving our species a lot of extra structure. 

One important example of an operad that, together with its variations, plays an important role in algebraic aspects of SPDEs, is the operad $\PL$. The underlying linear species of this operad is the linearisation of the species $\mathrm{RT}$, and the operad structure is defined in a simple combinatorial way \cite{MR1827084}: if $T_1\in\mathrm{RT}(V_1)$, and $T_2\in\mathrm{RT}(V_2)$, then for $v\in V_1$, we define
 \[
T_1\circ_v T_2=\sum_{f\colon \mathbb{E}^{+}_{\graftI }(v)\to V_2} T_1\circ_v^f T_2 .
 \] 
Here $\mathbb{E}^{+}_{\graftI }(v)$ is the set of incoming edges of the vertex $v$ of $T_1$; the tree $T_1\circ_v^f T_2$ is obtained by replacing the vertex $v$ of the tree $T_1$ by the tree $T_2$, and grafting each incoming edge $e$ of $v$ at the vertex $f(e)$ of $T_2$. This operad controls  \emph{pre-Lie algebras}, that is algebras with one binary operation $\triangleleft$ satisfying the identity 
 \[
(a_1\triangleleft a_2)\triangleleft a_3-a_1\triangleleft(a_2\triangleleft a_3)=(a_1\triangleleft a_3)\triangleleft a_2-a_1\triangleleft(a_3\triangleleft a_2),
 \]
 see \cite{MR1827084} for details.

The category of all operads admits coproducts, which seems to be first spelled out in \cite{MR2034546}. The coproduct $\calO_1\vee\calO_2$ is simply the free product of monoids (one takes the monoid generated by $\calO_1$ and $\calO_2$, and imposes the relations saying that the two units are equal, and that all compositions that ``can be computed'', that is, compositions of several elements of $\calO_1$ or of several elements of $\calO_2$, are computed in the way they would be computed in $\calO_1$ or in $\calO_2$, respectively). All operads we consider are augmented, that is, equipped with a map to the trivial operad $\epsilon\colon\calO\to \mathbbold{1}$. For such operads, it is very easy to describe the underlying linear species of the coproduct $\calO_1\vee\calO_2$ in terms of the underlying species of $\calO_1$ and $\calO_2$. Specifically, we have the equalities of linear species
\begin{gather}
\calO_1\vee\calO_2=\mathbbold{1}\oplus(\calO_1\vee\calO_2)_1\oplus(\calO_1\vee\calO_2)_2,\\
(\calO_1\vee\calO_2)_1=\overline{\calO_1}\circ(\calO_1\vee\calO_2)_2,\\
(\calO_1\vee\calO_2)_2=\overline{\calO_2}\circ(\calO_1\vee\calO_2)_1,
\end{gather}
where $\overline{\calO}$ denotes the kernel of the augmentation, and $(\calO_1\vee\calO_2)_1$ and $(\calO_1\vee\calO_2)_2$ denote the parts of the coproduct obtained as operadic substitutions into elements of the augmentation ideals of $\calO_1$ and $\calO_2$ of some other elements of the coproduct. It follows immediately from this description and from the K\"unneth formula that for two augmented differential graded operads $\calO_1$ and $\calO_2$, we have an operad isomorphism
\begin{equation}
H_\bullet(\calO_1\vee\calO_2)\cong H_\bullet(\calO_1)\vee H_\bullet(\calO_2).
\end{equation}

One important example of a coproduct that will be of interest for us will be the operad of Lie admissible algebras, which we shall denote $\LA$. The notion of a Lie-admissible algebra goes back to Albert \cite{MR0027750}; classically, a Lie-admissible algebra is an algebra with one binary operation satisfying the identity 
 \[
\sum_{\sigma\in\Sigma_3} \mathrm{sign}(\sigma) ((a_{\sigma(1)}a_{\sigma(2)})a_{\sigma(3)}-a_{\sigma(1)}(a_{\sigma(2)}a_{\sigma(3)}))=0.
 \]
This identity is in fact equivalent to the Jacobi identity for the bracket $[a,b]=ab-ba$, which means that, considering the ``polarised'' operations $[a,b]$ and $a\bullet b=ab+ba$, we can give an equivalent definition of a Lie-admissible algebra: it is an algebra with an anticommutative bracket satisfying the Jacobi identity and a commutative product not satisfying any identities. This observation, going back to \cite{MR2225770}, literally means that the operad of Lie-admissible algebras is the coproduct of the Lie operad and the operad $\ComMag$ of commutative magmatic algebras, that is algebras with one commutative binary operation that does not satisfy any identities. 

\subsection{Dimensions and symmetric group characters}

For linear species, we also have exponential generating series
 \[
f_{\calF}(t)=\sum_{n\ge 0}\frac{\dim\calF(\{1,\ldots,n\})}{n!}t^n,
 \]
and they satisfy 
\begin{equation}
f_{\calF_1\cdot \calF_2}(t)=f_{\calF_1}(t)f_{\calF_2}(t),\quad 
f_{\calF_1\circ \calF_2}(t)=f_{\calF_1}(f_{\calF_2}(t)).
\end{equation}
Additionally, one can introduce a more refined invariant, a generating function of symmetric group characters. This is a formal power series in infinitely many variables $p_i$, $i\ge 1$, defined as
\begin{multline*}
F_{\calF}= \sum_{n\ge 0}\sum_{k_1+2k_2+\cdots+nk_n=n}\chi_{\calF,n}(1^{k_1}2^{k_2}\cdots n^{k_n})\frac{p_1^{k_1}p_2^{k_2}\cdots p_n^{k_n}}{1^{k_1}k_1!2^{k_2}k_2!\cdots n^{k_n}k_n!},
\end{multline*}
where $\chi_{\calF,n}$ is the character of the symmetric group $\Sigma_n$ acting on $\calF(\{1,2,\ldots,n\})$ via the maps $\calF(\sigma)$, and where $1^{k_1}2^{k_2}\cdots n^{k_n}$ denotes the conjugacy class in $\Sigma_n$ with $k_1$ cycles of length $1$, $k_2$ cycles of length $2$, \ldots, $k_n$ cycles of length $n$. 
We shall think of $F_{\calF}$ as an element of the \emph{ring of symmetric functions} $\Lambda$; roughly speaking, each variable $p_r$ should be thought of as a power sum $\sum_i t_i^r$ in a sufficiently large number of variables. We shall make use of that intuition in some character calculations, since it will be beneficial to also use elementary symmetric functions $e_r$ (sums of $r$-fold products of several different variables $t_i$) and the symmetric functions $m_{\lambda}$ (sums of elements in the $\Sigma_n$-orbit of the monomial $t_1^{\lambda_1}\cdots t_s^{\lambda_s}$). 

For our purposes, it is important to know that these generating functions are also compatible with operations on linear species:
\begin{equation}
F_{\calF_1\cdot \calF_2}=F_{\calF_1}F_{\calF_2},\quad
F_{\calF_1\circ \calF_2}=F_{\calF_1}|_{p_i:=F_{\calF_2}(p_{ij}\colon j\ge 1)}. 
\end{equation}

\subsection{Koszul duality}

An important tool from homotopical algebra one can use for operads is Koszul duality. We refer the reader to \cite{MR2954392} for details, and give just the executive summary here. When we discuss constructions of homotopical algebra, we often use the formal symbol $s$ of homological degree $1$ to handle suspensions: for a homologically graded vector space $V$, we denote by $sV$ the homologically graded vector space with shifted homological degrees, $(sV)_i=V_{i-1}$. 

Let $\calO$ be an augmented operad. We can then define its \emph{bar construction}: it is made of rooted trees whose vertices are decorated by $s\overline{\calO}$ (to be completely precise, one should say ``the cofree cooperad on $s\overline{\calO}$''), and it has a differential $d$ which forms the sum, with appropriate signs, of all ways to contract one edge in such a tree (and label the newly obtained vertex by the partial composition of the labels of two disappeared vertices). The axioms of an operad ensure that $d^2=0$, so one can compute the homology of this differential.   

Suppose that our operad $\calO$ is ``homogeneous'', that is, it is generated by corollas of weight one, and relations between them are combinations of decorated trees of the same weight. Then the bar construction is bi-graded: it has the weight grading, and the homological grading. We say that our operad is a \emph{Koszul} operad if the homology of the bar construction is concentrated on the diagonal (where the weight is equal to the homological degree). In this case, if we denote by $\calH$ the collection of homologies of components of the bar construction, we have the following important properties:  
\begin{gather}
f_{\calA}(-f_{\calH}(-t))=t, \\
F_{\calA}(-F_{\calH}(-p_{ij}\colon j\ge 1)\colon i\ge 1)=p_1. 
\end{gather}

\subsection{Operadic twisting}

The definition of \emph{operadic twisting} goes back to the work of Willwacher \cite[Appendix~I]{MR3348138} who introduced it to work with Kontsevich's graph complexes. The original version of operadic twisting is a certain endofunctor of the category of differential graded operads equipped with a morphism from the operad of shifted $L_\infty$-algebras. There exists a counterpart of that endofunctor for operads equipped with a morphism from the operad~$L_\infty$, see \cite[Sec.~3.5]{MR3299688}. In particular, that version can be applied to an operad $\calP$ concentrated in degree zero equipped with a map of operads 
$f\colon \Lie\to\calP$ that sends the generator of $\Lie$ to a certain binary operation of $\calP$ that we denote $[-,-]$. In this case, the operad $\Tw(\calP)$, the result of applying the twisting procedure to the operad $\calP$, is a differential graded operad that can be defined as follows~\cite{MR4621635}. Denote by $\alpha$ a new operation of arity~$0$ and homological degree~$-1$. The underlying non-differential operad of $\Tw(\calP)$ is the coproduct $\calP\vee\kk\alpha$; algebras over that operad are $\calP$-algebras with a specified element of homological degree~$-1$. To define the differential, one performs two steps. First, one considers the operad
 $
\MC(\calP)=\left(\calP\vee\kk\alpha, d_{\MC} \right)
 $
encoding dg $\calP$-algebras with a Maurer--Cartan element; by definition, this means that the differential $d_{\MC}$ vanishes on $\calP$ and satisfies $d_{\MC}(\alpha)=-\frac12[\alpha,\alpha]$. The element $\ell_1^{\alpha}\in \MC(\calP)(1)$ defined by the formula
$
\ell_1^{\alpha}(a_1)=[\alpha,a_1] 
$
can be use to twist the differential of that operad, and one puts 
 \[
\Tw(\calP)=\left(\calP\vee\kk\alpha, d_{\Tw}=d_{\MC}+\mathrm{ad}_{\ell_1^{\alpha}}\right), 
 \]
where $\mathrm{ad}_{\ell_1^{\alpha}}(\mu)=\ell_1^{\alpha}\circ_1\mu-(-1)^{|\mu|}\sum_i\mu\circ_i\ell_1^{\alpha}$ 
The operad $\PL$ is equipped with a map from the Lie operad that sends the Lie bracket to 
\begin{equs}
   \begin{tikzpicture}[scale=0.2,baseline=-5]
    \coordinate (root) at (0,-2);
    \coordinate (center) at (0,2);
    \draw[] (root) -- (center);
    \node[var1] (rootnode) at (center) {$ 2 $};
    \node[var1] (rootnode) at (root) { $  1 $};
  \end{tikzpicture} -
  \begin{tikzpicture}[scale=0.2,baseline=-5]
    \coordinate (root) at (0,-2);
    \coordinate (center) at (0,2);
    \draw[] (root) -- (center);
    \node[var1] (rootnode) at (center) {$ 1 $};
    \node[var1] (rootnode) at (root) { $  2 $};
  \end{tikzpicture}, 
\end{equs} 
which corresponds to the fact that the commutator in every pre-Lie algebra satisfies the Jacobi identity. If one applies to the operad $\PL$ the operadic twisting procedure, the result is a differential graded operad that admits the following explicit description \cite{dotsenko2021homotopical}. The arity $n$ component $\Tw(\PL)(n)$ is spanned by rooted trees with ``normal'' vertices labelled $1$, $\ldots$, $n$, and a certain number of ``special'' vertices labelled~$\alpha$. The differential in $\Tw(\PL)$, applied to such tree~$T$, is made of the following terms:

-- the sum over all possible ways to split a normal vertex labelled $i$ into a normal vertex with the same label $i$ and a special vertex, and to connect the incoming edges of that vertex to one of the two vertices thus obtained, so that the term where the vertex further from the root retains the label is taken with the plus sign, and the other term is taken with the minus sign (corresponding to the operadic insertions of $\ell_1^{\alpha}$ at labelled vertices):
\begin{equs}
 \begin{tikzpicture}[scale=0.2,baseline=-5]
  \coordinate (root) at (0,0);
  \coordinate (right) at (2,2);
  \coordinate (left) at (-2,2);
  \coordinate (center) at (0,2);
  \draw[] (root) -- (left);
  \draw[] (right) -- (root);
  \node[var1] (rootnode) at (root) {$ i $};
  \node[] (rootnode) at (center) {$ \cdots $};
\end{tikzpicture} \rightarrow  \sum \begin{tikzpicture}[scale=0.2,baseline=-5]
\coordinate (root) at (0,0);
\coordinate (right) at (2,2);
\coordinate (left) at (-3,3);
\coordinate (lright) at (-5,5);
\coordinate (lleft) at (-1,5);
\coordinate (center) at (0,2);
\coordinate (lcenter) at (-3,5);
\draw[] (lleft) -- (left);
\draw[] (lright) -- (left);
\draw[] (root) -- (left);
\draw[] (right) -- (root);
\node[var1] (rootnode) at (root) {$ \alpha $};
\node[var1] (rootnode) at (left) {$ i $};
\node[] (rootnode) at (center) {$ \cdots $};
\node[] (rootnode) at (lcenter) {$ \cdots $};
\end{tikzpicture}\,  -  \begin{tikzpicture}[scale=0.2,baseline=-5]
\coordinate (root) at (0,0);
\coordinate (right) at (2,2);
\coordinate (left) at (-3,3);
\coordinate (lright) at (-5,5);
\coordinate (lleft) at (-1,5);
\coordinate (center) at (0,2);
\coordinate (lcenter) at (-3,5);
\draw[] (lleft) -- (left);
\draw[] (lright) -- (left);
\draw[] (root) -- (left);
\draw[] (right) -- (root);
\node[var1] (rootnode) at (root) {$ i $};
\node[var1] (rootnode) at (left) {$ \alpha $};
\node[] (rootnode) at (center) {$ \cdots $};
\node[] (rootnode) at (lcenter) {$ \cdots $};
\end{tikzpicture}
\end{equs}
-- the sum over all special vertices of $T$ of all possible ways to split that vertex into two vertices of the same kind, and to connect the incoming edges of that vertex to one of the two vertices thus obtained, taken with the plus sign (corresponding to computing the differential of the Maurer--Cartan element, since we have $d_{\Tw}(\alpha)=d(\alpha)+\ell_1^\alpha(\alpha)=-\frac12[\alpha,\alpha]+[\alpha,\alpha]=\frac12[\alpha,\alpha]=\alpha\triangleleft\alpha$): 
\begin{equs}
  \begin{tikzpicture}[scale=0.2,baseline=-5]
    \coordinate (root) at (0,0);
    \coordinate (right) at (2,2);
    \coordinate (left) at (-2,2);
    \coordinate (center) at (0,2);
    \draw[] (root) -- (left);
    \draw[] (right) -- (root);
    \node[var1] (rootnode) at (root) {$ \alpha $};
    \node[] (rootnode) at (center) {$ \cdots $};
  \end{tikzpicture} \rightarrow  \sum \begin{tikzpicture}[scale=0.2,baseline=-5]
    \coordinate (root) at (0,0);
    \coordinate (right) at (2,2);
    \coordinate (left) at (-3,3);
    \coordinate (lright) at (-5,5);
    \coordinate (lleft) at (-1,5);
    \coordinate (center) at (0,2);
    \coordinate (lcenter) at (-3,5);
    \draw[] (lleft) -- (left);
    \draw[] (lright) -- (left);
    \draw[] (root) -- (left);
    \draw[] (right) -- (root);
    \node[var1] (rootnode) at (root) {$ \alpha $};
    \node[var1] (rootnode) at (left) {$ \alpha $};
    \node[] (rootnode) at (center) {$ \cdots $};
    \node[] (rootnode) at (lcenter) {$ \cdots $};
  \end{tikzpicture}
\end{equs}
-- grafting the tree $T$ at the new root which is a special vertex, taken with the minus sign, and the sum of all possible ways to create one extra leaf which is a special vertex, taken with the plus sign (corresponding to operadic insertions of the tree $T$ at the only vertex of $\ell_1^{\alpha}$).

\section{Chain rule symmetry}

\label{sec::3}

\subsection{Results for space-time white noises}

Theorem~\ref{thm:main renormalisation} tells us that we have an infinite number of solutions parametrised by a finite dimensional space of dimension $54$ which is $ \CS_4 $. The aim of the work \cite{BGHZ} was to consider only one natural solution by imposing natural symmetries. One of them is the chain rule which corresponds to invariance under the change of coordinates.
We first recall how a diffeomorphism acts on the various coefficients of \eqref{e:genClass}.
Given a diffeomorphism $\phi$ of $\mathbb{R}^d$, we then act on connections $\Gamma$, vector fields $\sigma$ 
and $(1,1)$-tensors $K$ in the usual way by imposing that
\begin{equs} \label{eq:actioncristo}
	(\phi \act \Gamma)^\alpha_{\eta\zeta}(\phi(u)) \, \d_\beta \phi^\eta(u) \, \d_\gamma \phi^\zeta(u) &= 
	\d_\mu \phi^\alpha(u) \, \Gamma^{\mu}_{\beta\gamma}(u) - \d^2_{\beta\gamma}\phi^\alpha(u)\;, \\
	\label{eq:actionvect}
	(\phi\act \sigma)^\alpha(\phi(u)) &= \d_\beta \phi^\alpha(u) \, \sigma^\beta(u)\;,\\
	(\phi\act K)^\alpha_\eta(\phi(u)) \, \d_\beta \phi^\eta(u)  &= \d_\mu \phi^\alpha(u) \, K^\mu_\beta(u)\;.\label{eq:actiontens}
\end{equs}

Recall that the covariant derivative $\nabla_X Y$ \label{covariant derivative page ref} of a vector field $Y$ in the direction of 
another vector field $X$ is the vector field given by 
\begin{equ}[e:covDiff]
	(\nabla_X Y)^\alpha (u) = X^\beta(u)\,\d_\beta Y^\alpha(u) + \Gamma^\alpha_{\beta\gamma}(u) \,X^\beta(u) \,Y^\gamma(u)\;.
\end{equ}
It is straightforward to verify that this definition satisfies
\begin{equ}
	\phi \act (\nabla_X Y) =  (\phi \act \nabla)_{\phi \act X} (\phi \act Y)\;,
\end{equ}
where $\phi \act \nabla$ denotes the covariant differentiation built as in \eqref{e:covDiff}, but
with $\Gamma$ replaced by $\phi \act \Gamma$.

\begin{definition} \label{geo_terms} Given a vector space of decorated trees $ V $, the space $ V_{\geo} \subset V $   consists of those elements $\tau$ such that, for all
	$d, m \geq 1$ and all choices of $\Gamma$ and $ \sigma $ as above and all diffeomorphisms $ \varphi $ of $  \mathbb{R}^d$ 
homotopic to the identity, one has the identity
\begin{equs}
	\varphi \cdot \Upsilon_{\Gamma, \sigma}[\tau] = \Upsilon_{\varphi \cdot \Gamma, \varphi \cdot \sigma}[\tau].
	\end{equs}
\end{definition}
One then defines in this way the various geometric spaces: $ \CS_{\geo}^g \subset \CS^g $, $ \CS_{\geo,2n}^g \subset \CS^g_{2n} $,  $ \CS_{\geo}^c \subset \CS^c $, $ \CS_{\geo,2n}^c \subset \CS^c_{n} $.

Note that one can define the covariant derivative operation on the level of decorating trees by putting 
\begin{equs} \label{covariant_derivative_semi_linear}
    \Nabla_{\tau_1}\tau_2 = 
    \begin{tikzpicture}[scale=0.2,baseline=2]
      \draw[symbols]  (-.5,2.5) -- (0,0) ;
      \draw[tinydots] (0,0)  -- (0,-1.3);
      \node[var] (root) at (0,-0.1) {\tiny{$ \tau_1 $ }};
      \node[var] (diff) at (-0.5,2.5) {\tiny{$ \tau_2 $ }};
      \node[blank] at (0.3,1.25) {\tiny{}};
    \end{tikzpicture} \, +  \frac{1}{2}\;
  \begin{tikzpicture}[scale=0.2,baseline=2]
  \coordinate (root) at (0,0);
  \coordinate (t1) at (-1,2);
  \coordinate (t2) at (1,2);
  \draw[tinydots] (root)  -- +(0,-0.8);
  \draw[kernels2] (t1) -- (root);
  \draw[kernels2] (t2) -- (root);
  \node[not] (rootnode) at (root) {};
  \node[var] (t1) at (t1) {\tiny{$ \tau_1 $}};
  \node[var] (t1) at (t2) {\tiny{$ \tau_2 $}};
  \node[blank] at (-1.2,0.6) {\tiny{$$}};
  \node[blank] at (1.2,0.6) {\tiny{$$}};
\end{tikzpicture}
\end{equs}
where the $ \begin{tikzpicture}[scale=0.2,baseline=2]
  \draw[symbols]  (-.5,2.5) -- (0,0) ;
  \draw[tinydots] (0,0)  -- (0,-1.3);
  \node[var] (root) at (0,-0.1) {\tiny{$ \tau_1 $ }};
  \node[var] (diff) at (-0.5,2.5) {\tiny{$ \tau_2 $ }};
  \node[blank] at (0.3,1.25) {\tiny{}};
\end{tikzpicture} $  are all the the decorated trees obtained from the grafting of $ \tau_1 $ onto $ \tau_2 $ (connecting the root of $ \tau_1 $ via an edge to a node of $ \tau_2 $ and summing over all the possibilities) via an edge decorated by  $\<thin>$. Let us consider the following $15$ elements defined using iterations of covariant derivatives: \label{VV page ref}
\begin{equs}
	\label{eq:covderiv}
	\begin{aligned}
	\VV^g_4 & =
	\big\{ \Nabla_{\<generic>}\<generic>,  \Nabla_{\<genericb>}\Nabla_{\<generic>}\Nabla_{\<genericb>}\<generic>, 
	\Nabla_{\<generic>}\Nabla_{\<genericb>}\Nabla_{\<genericb>}\<generic>,
	\Nabla_{\<genericb>}\Nabla_{\<genericb>}\Nabla_{\<generic>}\<generic>, 
	\Nabla_{\<genericb>}\Nabla_{\Nabla_{\<genericb>}\<generic>}\<generic>, 
	\Nabla_{\Nabla_{\<genericb>}\<generic>} \Nabla_{\<genericb>}\<generic>,
	\Nabla_{\Nabla_{\<generic>}\<genericb>} \Nabla_{\<genericb>}\<generic>, \\ & 
	\Nabla_{\Nabla_{\<genericb>}\<genericb>}\Nabla_{\<generic>}\<generic>, 
	\Nabla_{\Nabla_{\<genericb>}\Nabla_{\<genericb>}\<generic>}\<generic>,
	\Nabla_{\Nabla_{\<genericb>}\Nabla_{\<generic>}\<generic>}\<genericb>,
	\Nabla_{\Nabla_{\Nabla_{\<genericb>}\<generic>}\<generic>}\<genericb>,
	\Nabla_{\Nabla_{\Nabla_{\<genericb>}\<generic>}\<genericb>}\<generic>,
	\Nabla_{\<generic>}\Nabla_{\Nabla_{\<genericb>}\<generic>}\<genericb>,
	\Nabla_{\Nabla_{\<generic>}\Nabla_{\<genericb>}\<generic>}\<genericb>,
	\Nabla_{\Nabla_{\Nabla_{\<generic>}\<generic>}\<genericb>}\<genericb>,
	\Nabla_{\<genericb>}\Nabla_{\Nabla_{\<generic>}\<generic>}\<genericb>\ 
	\big\}.
	\end{aligned}
\end{equs}
It was established in  \cite{BGHZ} that these elements span the vector space $\CS_{\tiny{\geo},4}^g$, more precisely 
\begin{theorem} \label{thm_covariant_derivatives}
For all sufficiently high dimensions $d, m$, one has $
		\CS_{\tiny{\geo},4}^g = \left\langle \VV_4^g \right\rangle $, and the dimension of the previous space is $15$.
	\end{theorem}

Let us briefly recall the strategy of the proof. One starts by introducing a new space of decorated trees that contains another type of nodes denoted by $ \<diff> $. We consider a new rule given by 
\begin{equs}
    R_{\<diff>}(\<thin>) = \{(\<thin>^k,\<generic>_i), \, (\<thin>^k,\<diff>), (\<thick>^2,\<thin>^k)\,:\, k \ge 0 \}\;,
  \end{equs}
and we consider decorated trees generated from this rules with only one node of type $ \<diff> $. One can define the spaces   $ \CS_{\<diff>,2n}^g  $ and $ \CS^c_{\<diff>,n} $ with this extra generator and gets canonical injections of  $ \CS_{2n}^g  $ and $ \CS^c_{n} $ into the previous spaces.
In \cite[Sec. 6.1]{BGHZ}, a map $\phi_\geo: \CS^{j}  \to \CS^{j}_{\<diff>}  $ with $ j \in \lbrace g,c\rbrace $,
was defined as the unique infinitesimal morphism of $T_\d$-algebras such that
$\phi_\geo$ acts by setting
\begin{equ} \label{def_phi_geo_1}
  \def\offset{.8,1.3}
  \tikzset{external/export next=false}
  \begin{tikzpicture}[scale=0.2,baseline=-2]
    \draw[tinydots] (0,0)  -- (0,-0.8);
    \node[xi] (root) at (0,0) {};
  \end{tikzpicture}
  \;\mapsto\;
  \tikzset{external/export next=false}
  \begin{tikzpicture}[scale=0.2,baseline=2]
    \draw[symbols]  (-.5,2) -- (0,0) ;
    \draw[tinydots] (0,0)  -- (0,-0.8);
    \node[diff] (root) at (0,-0.1) {};
    \node[xi] (diff) at (-0.5,2) {};
  \end{tikzpicture}
  \;-\;
  \tikzset{external/export next=false}
  \begin{tikzpicture}[scale=0.2,baseline=2]
    \draw[symbols]  (-.5,2) -- (0,0) ;
    \draw[tinydots] (0,0)  -- (0,-0.8);
    \node[xi] (xi) at (0,0) {};
    \node[diff] (root) at (-0.5,2) {};
  \end{tikzpicture}\;,\qquad\qquad
  \def\offset{.8,1.3}
  \def\offsett{1.3,.8}
  \tikzset{external/export next=false}
  \begin{tikzpicture}[scale=0.2,baseline=-2]
    \coordinate (root) at (0,0);
    \coordinate (t1) at (-.8,1.5);
    \coordinate (t2) at (.8,1.5);
    \draw[tinydots] (root)  -- +(0,-0.8);
    \draw[kernels2] (t1) -- (root);
    \draw[kernels2] (t2) -- (root);
    \node[not] (rootnode) at (root) {};
  \end{tikzpicture}
  \;\mapsto\;
  \tikzset{external/export next=false}
  \begin{tikzpicture}[scale=0.2,baseline=-.25cm]
    \coordinate (root) at (0,0);
    \coordinate (tri) at (0,-2);
    \coordinate (t1) at (-.8,1.5);
    \coordinate (t2) at (.8,1.5);
    \draw[tinydots] (tri)  -- +(0,-0.8);
    \draw[kernels2] (t1) -- (root);
    \draw[kernels2] (t2) -- (root);
    \draw[symbols] (root) -- (tri);
    \node[not] (rootnode) at (root) {};
    \node[diff] (trinode) at (tri) {};
  \end{tikzpicture}
  \;-\;
  \tikzset{external/export next=false}
  \begin{tikzpicture}[scale=0.2,baseline=2]
    \coordinate (root) at (0,0);
    \coordinate (tri) at (1.2,1.2);
    \coordinate (t1) at (-1.2,1.2);
    \coordinate (t2) at (0,1.7);
    \draw[tinydots] (root)  -- +(0,-0.8);
    \draw[kernels2] (t1) -- (root);
    \draw[kernels2] (t2) -- (root);
    \draw[symbols] (root) -- (tri);
    \node[not] (rootnode) at (root) {};
    \node[diff] (trinode) at (tri) {};
  \end{tikzpicture}
  \;-\;
  \tikzset{external/export next=false}
  \begin{tikzpicture}[scale=0.2,baseline=-2]
    \coordinate (root) at (0,0);
    \coordinate (t1) at (-.5,2.5);
    \coordinate (tri) at (1,1);
    \coordinate (t2) at (1,2.5);
    \draw[tinydots] (root)  -- +(0,-0.8);
    \draw[kernels2] (t1) -- (root);
    \draw[symbols] (t2) -- (tri);
    \draw[kernels2] (tri) -- (root);
    \node[not] (rootnode) at (root) {};
    \node[diff] (trinode) at (tri) {};
  \end{tikzpicture}
  \;-\;
  \tikzset{external/export next=false}
  \begin{tikzpicture}[scale=0.2,baseline=-2]
    \coordinate (root) at (0,0);
    \coordinate (t1) at (.5,2.5);
    \coordinate (tri) at (-1,1);
    \coordinate (t2) at (-1,2.5);
    \draw[tinydots] (root)  -- +(0,-0.8);
    \draw[kernels2] (t1) -- (root);
    \draw[symbols] (t2) -- (tri);
    \draw[kernels2] (tri) -- (root);
    \node[not] (rootnode) at (root) {};
    \node[diff] (trinode) at (tri) {};
  \end{tikzpicture}
  \;-\;2
  \tikzset{external/export next=false}
  \begin{tikzpicture}[scale=0.2,baseline=-2]
    \coordinate (root) at (0,0);
    \coordinate (t1) at (-.8,1.5);
    \coordinate (t2) at (0.8,1.5);
    \draw[tinydots] (root)  -- +(0,-0.8);
    \draw[symbols] (t1) -- (root);
    \draw[symbols] (t2) -- (root);
    \node[diff] (rootnode) at (root) {};
  \end{tikzpicture}\;,
\end{equ}
which has a very natural interpretation in terms of how vector fields and Christoffel symbols transform under infinitesimal changes of coordinates. Furthermore, one defines the map $\hat \phi_\geo (\tau)$ by 
\begin{equ} \label{def_phi_geo_2}
  \hat \phi_\geo (\tau) = \phi_\geo(\tau) - [\tau,\<diff>]\;,
\end{equ}
for any $\tau \in \CS^j$. It is established in \cite[Prop.~6.2]{BGHZ} that $\CS^g_{\geo,4} = \CS^g_4 \cap \ker \hat \phi_\geo$. In order to conclude, one wants to show that
$ \CS^g_4 \cap \ker \hat \phi_\geo = \left\langle \VV^g_4 \right\rangle $. This is where the reasoning in \cite{BGHZ} starts to be specific by using a counting dimension argument. Indeed, one first has to compute the dimension of $\left\langle \VV^g_4 \right\rangle$. Then, one derives independent equations for $\ker \hat \phi_\geo $ in order to reach the dimension of  $\left\langle \VV^g_4 \right\rangle$. One concludes from the inclusion $ \left\langle \VV_4^g \right\rangle \subset \CS_4^g \cap \ker \hat \phi_\geo $.

\begin{remark}
Theorem \ref{thm_covariant_derivatives} has been proved only for $ \CS_4^g $; most of the arguments of its proof rely on computations by hand that prevent any generalisation for many years in the absence of new ideas.
\end{remark}

As  a consequence of Theorem \ref{thm_covariant_derivatives}, one is able to find suitable counter-terms in order to guarantee the chain rule property for the solution of \eqref{eq:renorm nonlocal1}.
\begin{theorem}\label{chain:rule}
In Theorem \ref{thm:main renormalisation}, one can choose the renormalisation constants $C_\eps(\tau)$ in such a way
that 
	\begin{equs}
		\sum_{\tau \in \SS_4} C_{\eps}(\tau) \frac{ \tau}{ S(\tau)}
	\end{equs}
is generated by the space $ \VV_4 $. Then the equations \eqref{eq:renorm nonlocal1} transform according to the chain rule under composition with diffeomorphisms. 
\end{theorem}

\subsection{The operad of Christoffel trees}

Let us develop an algebraic formalism corresponding to the combinatorics of trees that we have seen in the previous sections. Concretely, we shall define a new operad which we shall call the operad of Christoffel trees, and unravel its (rather simple) algebraic meaning. This will be a crucial step towards proving our main result. 

\begin{definition} \label{Christoffel trees}
A \emph{Christoffel tree} is a rooted tree with black and white vertices such that every black vertex has at least two input edges, and moreover for each black vertex two of its input edges are declared ``special''. In all figures the special edges will be drawn using thick lines. The \emph{species of Christoffel trees} assigns to every finite set $S$ the vector space $\NT(S)$ with a basis of all Christoffel trees whose white vertices are put in a bijection with $S$. 
\end{definition}

One can define an operad structure on the linear species $\NT$ by a rule similar to the one defining an operad on the species of rooted trees \cite{MR1827084}: the insertion $T_1\circ_v T_2$ of a Christoffel tree $T_2\in \NT(V_2)$  in the place of the vertex $v\in V_1$ of a Christoffel tree $T_1\in \NT(V_1)$ is given by  
 \[
T_1\circ_v T_2:=\sum_{f\colon \mathbb{E}^{+}_{\graftI }(v)\to V_2} T_1\circ_v^f T_2 .
 \] 
Here $\mathbb{E}^{+}_{\graftI }(v)$ is the set of incoming edges of the vertex $v$ in $T_1$; the tree $S\circ_i^f T$ is obtained by replacing the vertex $i$ of the tree $S$ by the tree $T$, and grafting each incoming edge $e$ of $v$ at the vertex $f(e)$ of $T_2$. Note that this grafting happens at all vertices, not just the white ones. For example, we have

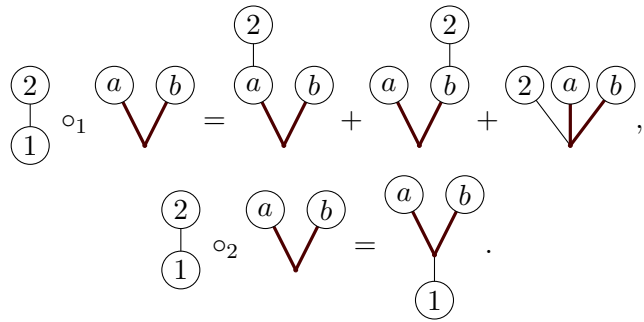
\begin{gather*}
  \begin{tikzpicture}[scale=0.2,baseline=-5]
    \coordinate (root) at (0,-2);
    \coordinate (center) at (0,2);
    \draw[] (root) -- (center);
    \node[var1] (rootnode) at (center) {$ 2 $};
    \node[var1] (rootnode) at (root) { $  1 $};
  \end{tikzpicture}\,\circ_1 \begin{tikzpicture}[scale=0.2,baseline=-5]
\coordinate (root) at (0,-2);
\coordinate (right) at (2,2);
\coordinate (left) at (-2,2);
\draw[kernels2] (root) -- (left);
\draw[kernels2] (root) -- (right);
\node[var1] (rootnode) at (left) {$ a $};
\node[var1] (rootnode) at (right) {$ b $};
\node[not] (rootnode) at (root) {};
\end{tikzpicture}=
\begin{tikzpicture}[scale=0.2,baseline=-5]
\coordinate (root) at (0,-2);
\coordinate (right) at (2,2);
\coordinate (left) at (-2,2);
\coordinate (leftv) at (-2,6);
\draw[] (left) -- (leftv);
\draw[kernels2] (root) -- (left);
\draw[kernels2] (root) -- (right);
\node[var1] (rootnode) at (left) {$ a $};
\node[var1] (rootnode) at (right) {$ b $};
\node[not] (rootnode) at (root) {};
\node[var1] (rootnode) at (leftv) { $  2 $};
\end{tikzpicture}+
\begin{tikzpicture}[scale=0.2,baseline=-5]
\coordinate (root) at (0,-2);
\coordinate (right) at (2,2);
\coordinate (rightv) at (2,6);
\coordinate (left) at (-2,2);
\draw[] (right) -- (rightv);
\draw[kernels2] (root) -- (left);
\draw[kernels2] (root) -- (right);
\node[var1] (rootnode) at (left) {$ a $};
\node[var1] (rootnode) at (right) {$ b $};
\node[not] (rootnode) at (root) {};
\node[var1] (rootnode) at (rightv) { $  2 $};
\end{tikzpicture}+
\begin{tikzpicture}[scale=0.2,baseline=-5]
\coordinate (root) at (0,-2);
\coordinate (right) at (3,2);
\coordinate (left) at (0,2);
\coordinate (newv) at (-3,2);
\draw[] (root) -- (newv);
\draw[kernels2] (root) -- (left);
\draw[kernels2] (root) -- (right);
\node[var1] (rootnode) at (left) {$ a $};
\node[var1] (rootnode) at (right) {$ b $};
\node[not] (rootnode) at (root) {};
\node[var1] (rootnode) at (newv) { $  2 $};
\end{tikzpicture}
,\\
  \begin{tikzpicture}[scale=0.2,baseline=-5]
    \coordinate (root) at (0,-2);
    \coordinate (center) at (0,2);
    \draw[] (root) -- (center);
    \node[var1] (rootnode) at (center) {$ 2 $};
    \node[var1] (rootnode) at (root) { $  1 $};
  \end{tikzpicture}\, \circ_2 \begin{tikzpicture}[scale=0.2,baseline=-5]
\coordinate (root) at (0,-2);
\coordinate (right) at (2,2);
\coordinate (left) at (-2,2);
\draw[kernels2] (root) -- (left);
\draw[kernels2] (root) -- (right);
\node[var1] (rootnode) at (left) {$ a $};
\node[var1] (rootnode) at (right) {$ b $};
\node[not] (rootnode) at (root) {};
\end{tikzpicture}=
\begin{tikzpicture}[scale=0.2,baseline=-5]
\coordinate (root) at (0,-4);
\coordinate (vert) at (0,-1);
\coordinate (right) at (2,3);
\coordinate (left) at (-2,3);
\draw[] (root) -- (vert);
\draw[kernels2] (vert) -- (left);
\draw[kernels2] (vert) -- (right);
\node[var1] (rootnode) at (left) {$ a $};
\node[var1] (rootnode) at (right) {$ b $};
\node[not] (rootnode) at (vert) {};
\node[var1] (rootnode) at (root) { $  1 $};
\end{tikzpicture}
.
\end{gather*}
Note that if we think of Christoffel trees as representing iterated derivatives of expressions involving the Christoffel symbols as in Equation \eqref{interpretation_renormalisation}, our formulae literally correspond to the rule of computing the derivative of the product. \\

To state the following result, recall that $\ComMag$ is the operad of commutative magmatic algebras, that is algebras with one commutative operation that does not satisfy any further identities.

\begin{proposition} \label{operad_christoffel}
The insertion operations $\circ_v$ make the linear species $\NT$ into an operad. That operad is generated by its binary operations. Moreover, that operad is isomorphic to the coproduct of the operads $\PL$ and $\ComMag$.
\end{proposition}

\begin{proof}
Proving that the operad axioms are satisfied is completely analogous to the corresponding properties for the rooted trees operad~\cite{MR1827084}. The binary operations of our operad are 
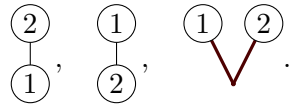
\begin{equs}
  \begin{tikzpicture}[scale=0.2,baseline=-5]
    \coordinate (root) at (0,-2);
    \coordinate (center) at (0,2);
    \draw[] (root) -- (center);
    \node[var1] (rootnode) at (center) {$ 2 $};
    \node[var1] (rootnode) at (root) { $  1 $};
  \end{tikzpicture} \, , \quad \begin{tikzpicture}[scale=0.2,baseline=-5]
  \coordinate (root) at (0,-2);
  \coordinate (center) at (0,2);
  \draw[] (root) -- (center);
  \node[var1] (rootnode) at (center) {$ 1 $};
  \node[var1] (rootnode) at (root) { $  2 $};
\end{tikzpicture} \, , \quad \begin{tikzpicture}[scale=0.2,baseline=-5]
\coordinate (root) at (0,-2);
\coordinate (right) at (2,2);
\coordinate (left) at (-2,2);
\draw[kernels2] (root) -- (left);
\draw[kernels2] (root) -- (right);
\node[var1] (rootnode) at (left) {$ 1 $};
\node[var1] (rootnode) at (right) {$ 2 $};
\node[not] (rootnode) at (root) {};
\end{tikzpicture}.
\end{equs}
To show that they generate everything, we note that the suboperad spanned by trees with white vertices only is clearly isomorphic to the rooted trees operad of Chapoton--Livernet, and hence is generated by its binary operations. To show that the operad of Christoffel trees is generated by its binary operations, it is enough to generate all the ``Christoffel corollas'', that is trees with one black vertex connected to $k+2$ white vertices. This is easily done by induction: for $k=0$ we have a binary operation, and for $k>0$, we note that
 \[
S\circ_1 T_{k-1}-\sum_{i=1}^{k+1}T_{k-1}\circ_i S=T_k,
 \]
where $S$ is the two-vertex tree with the root $1$ and leaf $k+2$, $T_{k-1}$ is a corolla with $k+1=(k-1)+2$ white vertices, and $T_k$ is a corolla with $k+2$ white vertices. Finally, since the rooted trees operad is isomorphic to the operad $\PL$, and the operad obtained by iterated compositions of the third binary operation is clearly isomorphic to the operad $\ComMag$, our operad received a surjective map from the coproduct $\PL\vee\ComMag$, so it is enough to show that the dimensions of components of the operad $\NT$ and the operad $\PL\vee\ComMag$ are the same. The coproduct of Koszul operads is known to be Koszul, and the quadratic dual is the connected sum of the quadratic dual operads. Since the quadratic dual of the operad $\PL$ is the operad usually denoted $\Perm$ with $\dim(\Perm(n))=n$ for all~$n$, and the quadratic dual of the operad $\ComMag$ is the operad supported at arities $1$ and $2$ (and one-dimensional in these arities), the exponential generating series for dimensions of components of the quadratic dual operad of $\PL\vee\ComMag$ is $t\exp(t)+\frac{t^2}2$. Using the functional equation relating the exponential generating series of a Koszul operad and its quadratic dual \cite{MR2954392}, we conclude that if we denote by $f$ the generating series of the operad $\PL\vee\ComMag$, we have
 \[
f\exp(-f)-\frac{f^2}2=t,
 \]
or, equivalently, 
 \[
f=t\exp(f)+\frac{f^2}2\exp(f).
 \]
At the same time, we have the following equation on the level of species: 
 \[
\NT=E_1\cdot E(\NT)\oplus E_2(\NT)\cdot E(\NT).
 \] 
Here $E=\bigoplus_{k\ge 0} E_k$ is the ``species of sets'', $E_n(I)=0$ unless $|I|=n$, in which case $E_n(I)=\kk\{I\}$. The way to prove that formula is as follows. First, we note that 
\begin{itemize}
\item $E(\NT)$ is the species of sets of Christoffel trees, 
\item $E_1\cdot E(\NT)$ is a species whose value on a set $I$ is a partition of $I$ into a singleton and its complement, and a structure of a set of Christoffel trees on that complement,
\item $(E_2\cdot E)(\NT)$ is a species whose value on a set $I$ is a partition of $I$ into several parts two of which are designated special, and a structure of a Christoffel tree on each part,
\end{itemize}
and we see that this precisely corresponds to the fact a Christoffel tree can either have a white root, in which case it is obtained by choosing the root label and a set of subtrees grafted at inputs of the root, or a black root, in which case it is obtained by choosing the set of two subtrees grafted at special inputs of the root and a set of subtrees grafted at other inputs. This species functional equation clearly gives the correct equation for exponential generating series, completing the proof.
\end{proof}

Let us remark that results of this section can be substantially generalised in the following way. One can consider the category of operads equipped with a map from the operad $\PL$; such operads are closely related to $T_\partial$-algebras from \cite{BGHZ}. Proposition \ref{operad_christoffel} asserts that $\NT$ is a free object in this category; it can be easily generalised to describe the free object generated by any given species. We have not yet found interesting applications of this more general construction, so we chose to avoid unnecessary generality and present the result and its proof for Christoffel trees.

\subsection{Generalisation for the full subcritical regime and cumulants}

The main result of this paper is a generalisation of Theorem~\ref{thm_covariant_derivatives}. We first denote by $ \VV_{2n}^{g} $ the set of iterated covariant derivatives involving $ 2 n $ generators that come in pairs, with the appropriate identifications, as in Equation~\ref{identification_trees}. Elements of $ \VV_{2n}^{g} $ could be seen as linear combinations of elements in $ \SS^g_{2n} $. The vector space  $ \VV_{n}^{c} $ in the case of noises described by their cumulants is defined similarly. For given noises $\xi_i$,  one has a set $ \VV_{\xi} $ of iterated covariant derivatives associated to $ \SS_{\xi}$ equal to one of the set described above.
\begin{theorem} \label{main_theorme_chain_rule}
	For sufficiently high dimensions $d, m$,
	one has 
	\begin{equs}
	\CS_{\tiny{\geo},2n}^g = \left\langle \VV_{2n}^g \right\rangle, \quad  \CS_{\tiny{\geo},n}^c = \left\langle \VV_{n}^c \right\rangle.
	\end{equs}
Moreover, their dimensions can be computed by an explicit direct procedure using generating functions (formal power series).
\end{theorem}

Before starting to prove this theorem, let us remark that a high dimension is needed for characterising $ \CS^g_{\geo,2n}  $
and $ \CS^c_{\geo,n} $
\begin{remark} \label{remark_over_paremetrisation}
  In the proof of \cite[Prop.~6.2]{BGHZ}, one derives the following characterisation of $  \CS^j_{\geo,n}$ ($ j \in \lbrace g,c \rbrace $): 
  \begin{equs}
    \tau  \in  \CS^j_{\geo,n} \quad  \text{ if and only if} \quad \Upsilon^{h}_{\Gamma,\sigma}[\hat \phi_\geo(\tau)] = 0, \ h : \mathbb{R}^d \rightarrow \mathbb{R}^d
  \end{equs}
where  $ \Upsilon^{h}_{\Gamma,\sigma} $ is an extension of $ \Upsilon_{\Gamma,\sigma} $ that sends $ \<diff> $ to $h$. Then, one uses the injectivity of the maps $  \Upsilon^{h}_{\Gamma,\sigma}  $ (see \cite[Thm. 5.25, Thm
5.31]{BGHZ}) in order to conclude. This injectivity is only true in high dimension. For small dimensions, it is an open problem to get a full characterisation of the space of geometric counter-terms. It is even not clear that $ \VV^j_{n} $ will be a suitable basis. But one can use the result in high dimension for still getting Theorem~\ref{chain:rule}. Indeed, the renormalisation constants do not change by changing the dimension: We can choose the same parametrisation. In small dimension, we just have an over parametrisation of the renormalised equation. Indeed from \cite[Prop. 3.9]{BGHZ}, there exists $v_{\geo} \in  \CS_{\geo}^{\bot} $
that $ \lim_{\eps \rightarrow 0}  C^{c}_{\geo} = v_{\geo}$. Furthermore,
$ v_{\geo} $ is independent of the choice of mollifier.
Here $\CS_{\geo}^{\bot} $ is the orthogonal of $ \CS_{\geo} $ and $  C^{c}_{\geo} $ is the term in
$ \sum_{\tau \in \SS_4} C_{\eps}(\tau) \frac{ \Upsilon_{\Gamma,\sigma}[\tau](u_{\eps})}{ S(\tau)} $
 that belongs to $\CS_{\geo}^{\bot} $. The renormalisation constants $C_{\eps}(\tau) $ correspond to the BPHZ renormalisation which guarantees the convergence of the solution $ u_{\eps} $.   This allows us to say that one can have finite renormalisation constants 
 on the orthogonal of $ \sum_{\tau \in \VV_{n}^j} C_{\eps}(\tau) \frac{ \Upsilon_{\Gamma,\sigma}[\tau](u_{\eps})}{ S(\tau)} $.
\end{remark}

The first part of the proof of Theorem~\ref{main_theorme_chain_rule} follows the steps of \cite{BGHZ}. In particular, one can define the vector spaces $ \CS_{\<diff>,2n}^g  $ and $ \CS^c_{\<diff>,n} $ with the extra  generator $\<diff>$, and canonical injections of  $ \CS_{2n}^g  $ and $ \CS^c_{n} $ into those spaces. Moreover, one can define a map $\phi_\geo: \CS^{j}  \to \CS^{j}_{\<diff>}  $ with $ j \in \lbrace g,c\rbrace $ by the exact same formula, and \cite[Prop.~6.2]{BGHZ} ensures that 
\begin{equs}
\CS^g_{\geo,2n} = \CS^g_{2n} \cap \ker \hat \phi_\geo, \quad \CS^c_{\geo,n} = \CS^c_{n} \cap \ker \hat \phi_\geo.
\end{equs}
However, the next steps that \cite{BGHZ} successfully accomplished in the case of $\CS^g_{\geo,4}$, that is computing the dimension of $\left\langle \VV^g_{2n} \right\rangle$ and deriving the independent equations characterising the space $ \CS_{2n}^g \cap \ker \hat \phi_\geo$, were out of reach for several years. Our strategy of the proof is drastically different, and brings in methods of category theory and homotopical algebra to describe the vector space $\ker \hat \phi_\geo$. 

The first nontrivial idea is to postpone the identifications of trees for later, and to work with the much bigger spaces of trees where we allow any number of vertices of various colours, and no identifications are yet made. Such an (infinite-dimensional) space of trees is, for the fixed number $k$ of colours (that is, of noises), the $k$-generated free algebras over the operad $\NT$ of Christoffel trees. A very useful seemingly trivial observation is that the map $\hat \phi_{k,\geo}$ is defined on the level of each such free algebra (mapping it to a free algebra with one extra generator), and commutes with the full symmetric group acting on the vertices, and thus commutes with any identifications that we may want to define using the symmetric group action. Thus, the vector space $\hat \phi_\geo$ on the space of trees we are led to consider in the SPDE context is exactly the same as the result of those identifications applied to the vector space $\hat \phi_{k,\geo}$ computed on the free algebra level. 

The next nontrivial idea is to not specify the number of noises yet, but rather work with the operad itself. However, the map $\hat \phi_{k,\geo}$ sends a free algebra with $k$ generators into a free algebra with $k+1$ generators, so one extra generator needs to be added. A clean way to do that is to consider the coproduct of operads $\NT\vee\kk\<diff>$, where $\kk\<diff>$ is a species supported at the empty set viewed as an operad without any compositions possible; it incorporates into our operad an operation without arguments, the generator that we always added in order to define the map $\hat \phi_{k,\geo}$. For each $S$ with $|S|=k$, the map $\hat \phi_{k,\geo}$ can be restricted to $\NT(S)$, and these maps assemble into a map of species 
 \[
\hat \Phi_{\geo}\colon \NT\to \NT\vee\kk\<diff>
 \]
Moreover, the property that states, in terminology of \cite{BGHZ}, that the map $\hat \phi_\geo$ is an infinitesimal morphism of $T_\partial$-algebras translates into the fact that $\hat \Phi_{\geo}$ is an derivation with the values in the bimodule $\NT\vee\kk\<diff>$. 

The last nontrivial reformulation is a homological interpretation of the kernel of the map $\hat \Phi_{\geo}$. For that, we shall need to furnish a different description of that map. Recall that the map $\hat \phi_\geo$ in \cite{BGHZ} is described by local rules allowing to act on individual vertices of a tree. From the operad theory point of view, it is more natural to describe the action of the derivation $\hat \Phi_{\geo}$ on the generators of the operad. The operad $\NT$ is generated by 
\begin{equs}
  \begin{tikzpicture}[scale=0.2,baseline=-5]
    \coordinate (root) at (0,-2);
    \coordinate (center) at (0,2);
    \draw[] (root) -- (center);
    \node[var1] (rootnode) at (center) {$ 2 $};
    \node[var1] (rootnode) at (root) { $  1 $};
  \end{tikzpicture} \, , \quad \begin{tikzpicture}[scale=0.2,baseline=-5]
  \coordinate (root) at (0,-2);
  \coordinate (center) at (0,2);
  \draw[] (root) -- (center);
  \node[var1] (rootnode) at (center) {$ 1 $};
  \node[var1] (rootnode) at (root) { $  2 $};
\end{tikzpicture} \, , \quad \begin{tikzpicture}[scale=0.2,baseline=-5]
\coordinate (root) at (0,-2);
\coordinate (right) at (2,2);
\coordinate (left) at (-2,2);
\draw[kernels2] (root) -- (left);
\draw[kernels2] (root) -- (right);
\node[var1] (rootnode) at (left) {$ 1 $};
\node[var1] (rootnode) at (right) {$ 2 $};
\node[not] (rootnode) at (root) {};
\end{tikzpicture},
\end{equs}
and a direct calculation shows that 
\begin{gather*}
  \hat \Phi_\geo \left( \begin{tikzpicture}[scale=0.2,baseline=-5]
    \coordinate (root) at (0,-2);
    \coordinate (center) at (0,2);
    \draw[] (root) -- (center);
    \node[var1] (rootnode) at (center) {$ 2 $};
    \node[var1] (rootnode) at (root) { $  1 $};
  \end{tikzpicture} \right) = \hat \Phi_\geo \left(  \begin{tikzpicture}[scale=0.2,baseline=-5]
    \coordinate (root) at (0,-2);
    \coordinate (center) at (0,2);
    \draw[] (root) -- (center);
    \node[var1] (rootnode) at (center) {$ 1 $};
    \node[var1] (rootnode) at (root) { $  2 $};
  \end{tikzpicture} \right) =
  \begin{tikzpicture}[scale=0.2,baseline=-5]
  \coordinate (root) at (0,-2);
  \coordinate (right) at (2,2);
  \coordinate (left) at (-2,2);
  \draw[] (root) -- (left);
  \draw[] (right) -- (root);
  \node[var1] (rootnode) at (left) {$ 1 $};
  \node[var1] (rootnode) at (right) {$ 2 $};
  \node[diff] (rootnode) at (root) {};
  \end{tikzpicture} ,\\
  \hat \Phi_\geo \left(\begin{tikzpicture}[scale=0.2,baseline=-5]
    \coordinate (root) at (0,-2);
    \coordinate (right) at (2,2);
    \coordinate (left) at (-2,2);
    \draw[kernels2] (root) -- (left);
    \draw[kernels2] (root) -- (right);
    \node[var1] (rootnode) at (left) {$ 1 $};
    \node[var1] (rootnode) at (right) {$ 2 $};
    \node[not] (rootnode) at (root) {};
  \end{tikzpicture} \right) = -2\begin{tikzpicture}[scale=0.2,baseline=-5]
  \coordinate (root) at (0,-2);
  \coordinate (right) at (2,2);
  \coordinate (left) at (-2,2);
  \draw[] (root) -- (left);
  \draw[] (right) -- (root);
  \node[var1] (rootnode) at (left) {$ 1 $};
  \node[var1] (rootnode) at (right) {$ 2 $};
  \node[diff] (rootnode) at (root) {};
\end{tikzpicture}.
\end{gather*} 

Since in the operadic twisting $\Tw(\PL)$, we have
\begin{equs}
d_{\Tw} \left( \,   \begin{tikzpicture}[scale=0.2,baseline=-5]
  \coordinate (root) at (0,-2);
  \coordinate (center) at (0,2);
  \draw[] (root) -- (center);
  \node[var1] (rootnode) at (center) {$ 2 $};
  \node[var1] (rootnode) at (root) { $  1 $};
\end{tikzpicture} \, \right)   = \, \begin{tikzpicture}[scale=0.2,baseline=-5]
    \coordinate (root) at (0,-2);
    \coordinate (right) at (2,2);
    \coordinate (left) at (-2,2);
    \draw[] (root) -- (left);
    \draw[] (right) -- (root);
    \node[var1] (rootnode) at (left) {$ 1 $};
    \node[var1] (rootnode) at (right) {$ 2 $};
    \node[var1] (rootnode) at (root) {$ \alpha $};
  \end{tikzpicture},
\end{equs}
this immediately implies that on the suboperad $\PL\subset \NT\vee\kk\<diff>$ made of trees that do not have thick edges or $\<diff>$ vertices, the action of $\hat{\Phi}_{\geo}$ agrees precisely with the way the differential in $\Tw(\PL)$ acts on $\PL$, up to renaming $\<diff>$ into~$\alpha$. Since we already saw that allowing thick edges amounts to considering the coproduct $\PL\vee\ComMag$, it is reasonable to consider the coproduct of operads $\Tw(\PL)\vee\ComMag$. Its differential on trees without special vertices does not quite match the map $\hat{\Phi}_{\geo}$, but the mismatch is easily fixed as follows.

\begin{proposition} \label{correspondence_ker_phi_0}
Consider the unique derivation $d_0$ of the coproduct of operads $\PL\vee\ComMag\vee\kk\alpha$ that annihilates the suboperad $\PL\vee\kk\alpha$ and sends the generator of $\ComMag$ to 
\begin{equs}
- 2 \, \begin{tikzpicture}[scale=0.2,baseline=-5]
    \coordinate (root) at (0,-2);
    \coordinate (right) at (2,2);
    \coordinate (left) at (-2,2);
    \draw[] (root) -- (left);
    \draw[] (right) -- (root);
    \node[var1] (rootnode) at (left) {$ 1 $};
    \node[var1] (rootnode) at (right) {$ 2 $};
    \node[var1] (rootnode) at (root) {$ \alpha $};
  \end{tikzpicture}.
\end{equs}
The map $\mathrm{d}=d_{\MC}+\mathrm{ad}_{\ell_1^{\alpha}}+d_0$ makes $\PL\vee\ComMag\vee\kk\alpha$ into a differential graded operad. The degree zero homology of that operad is naturally isomorphic to $\ker\hat{\Phi}_{\geo}$.
\end{proposition}

\begin{proof}
Let us first show that $\mathrm{d}^2=0$. Since $\mathrm{d}=d_{\Tw}+d_0$, the endomorphism $\mathrm{d}$ is an operad derivation, and hence $[\mathrm{d},\mathrm{d}]=2\mathrm{d}^2$ is also an operad derivation (note that the graded commutator of endomorphisms of odd homological degree is their anticommutator), so it is enough to show that $\mathrm{d}^2$ vanishes on generators of the operad $\PL\vee\ComMag\vee\kk\alpha$. We have
$\mathrm{d}^2=d_{\Tw}d_0+d_0d_{\Tw}+d_0^2$, since $d_{\Tw}^2=0$. If we apply $\mathrm{d}^2$ to the generators of the suboperad $\PL\vee\kk\alpha$, the three terms vanish individually since $d_0$ annihilates that suboperad. If we apply $\mathrm{d}^2$ to the generator of $\ComMag$, these terms also vanish individually but in a more subtle way: the first term vanishes, since $d_0$ sends that generator to an element in the image of $d_{\Tw}$, the second term vanishes since this generator is already annihilated by $d_{\Tw}$, and the last term vanishes since $d_0$ annihilates the suboperad $\PL\vee\kk\alpha$.  

To show that the degree zero homology of our operad is naturally isomorphic to $\ker\hat{\Phi}_{\geo}$, we note that $\MC(\PL)\vee\ComMag$ is concentrated in negative degrees, and therefore there is no quotient to form: the degree zero homology of $\mathrm{d}$ is simply $\ker\mathrm{d}$. Moreover, since the homological degree of $\alpha$ is defined to be $-1$, the degree zero part of $\PL\vee\ComMag\vee\kk\alpha$ is precisely $\PL\vee\ComMag$, and the differential $\mathrm{d}$ was specifically designed in such a way that we have a commutative diagram 
 \[
\begin{tikzcd}
  \NT \arrow[r, "\hat{\Phi}_{\geo}"] \arrow[d]
    & \NT\vee\kk\<diff> \arrow[d] \\
  \PL\vee\ComMag \arrow[r, "\mathrm{d}"]
& \PL\vee\ComMag\vee\kk\alpha 
\end{tikzcd} ,
 \]
with vertical arrows being isomorphisms given by Proposition~\ref{operad_christoffel}, which proves the last assertion. 
\end{proof}

Our main problem is now reformulated in terms of computing homology. We shall now perform that homology computation; the following result is the main  algebraic result of the article, and the key ingredient in the proof of Theorem~\ref{main_theorme_chain_rule}.

\begin{theorem}\label{th:twisting}
The homology of the operad 
 \[
(\PL\vee\ComMag\vee \kk\alpha, d_{\MC}+\mathrm{ad}_{\ell_1^{\alpha}}+d_0)
 \] 
is concentrated in homological degree zero and is isomorphic to the operad $\LA$ of Lie-admissible algebras. 
\end{theorem}

\begin{proof}
Let us consider the filtration of the dg species
 \[
(\PL\vee\ComMag\vee \kk\alpha, d_{\MC}+\mathrm{ad}_{\ell_1^{\alpha}}+d_0),
 \]  
defining $\mathrm{F}_p (\PL\vee\ComMag\vee \kk\alpha)$ to be the span of all trees with at most $2p$ thick edges. The maps $d_{\MC}$ and $[\ell_1^{\alpha},-]$ preserve the number of thick edges of the tree, and the part $d_0$ decreases it by two, so this is a filtration by subcomplexes. Note that this filtration is bounded in the sense of \cite[Def.~5.4.2]{MR1269324}, since for each given finite set $I$ and each given homological degree $d$, the degree $d$ part of the component $(\PL\vee\ComMag\vee \kk\alpha)(I)$ of our dg species is spanned by trees with $|I|+d$ vertices, so their number of edges is bounded. Thus, we may use the Classical Convergence Theorem \cite[Th.~5.5.1]{MR1269324}, so for each $I$ the spectral sequence associated to this filtration converges to the homology of the complex $(\PL\vee\ComMag\vee \kk\alpha)(I)$. Since in the associated graded chain complex with respect to this filtration, the deformed part $d_0$ disappears, the spectral sequence computation starts with computing the homology of 
 \[
(\PL\vee\ComMag\vee\kk\alpha, d_{\MC}+\mathrm{ad}_{\ell_1^{\alpha}})\cong \Tw(\PL)\vee\ComMag,
 \]
and we clearly have
\begin{multline}
H_\bullet(\Tw(\PL)\vee\ComMag)\cong H_\bullet(\Tw(\PL))\vee\ComMag\\ \cong \Lie\vee\ComMag\cong\LA.
\end{multline}
Here the first isomorphism follows from the fact that coproducts of augmented differential graded operads commute with homology, the second isomorphism follows from the isomorphism $H_\bullet(\Tw(\PL))\cong \Lie$ established in \cite{dotsenko2021homotopical}, and the coproduct factorisation of the operad $\LA$ was already mentioned above. Since the homology is concentrated in degree zero, there are no room for further differentials, and the spectral sequence abuts at the first page, so the result follows. 
\end{proof}

The following corollary establishes Theorem \ref{main_theorme_chain_rule} for the full subcritical regime. 

\begin{corollary}\label{cor:gaussian-dim}
For an integer $k>0$, let us denote by $U_k$ the subspace of the free Lie-admissible algebra $\LA(x_1,x_2,\ldots,x_k)$ in $k$ generators consisting of elements of degree two in each of the $k$ generators. We have the natural vector space isomorphism 
 \[
\CS_{\tiny{\geo},2n}^g\cong U_1^{\Sigma_1}\oplus U_2^{\Sigma_2}\oplus \cdots \oplus U_n^{\Sigma_n},
 \]
where $U_k^{\Sigma_k}$ is the vector space of elements of $U_k$ that are invariant under the action of the group $\Sigma_k$ permuting the generators $x_1,x_2,\ldots,x_k$. All elements of $\CS_{\tiny{\geo},2n}^g$ are obtained as linear combinations of iterations of covariant derivatives. 
\end{corollary}

\begin{proof}
First of all, let us note that $\CS_{\tiny{\geo},2n}^g$ naturally splits into a direct sum of subspaces spanned by trees having exactly $2k$ vertices (which we think of as $k$ pairs of equal vertices representing the same noise). For the given $k$, we may use our identification  made by renumbering, and look at Christoffel trees with two vertices numbered $1$, two vertices numbered $2$, \ldots, two vertices numbered $k$, up to renumbering. This means that we pass to coinvariants of the symmetric group $\Sigma_k$ acting on labels. Over a field of zero characteristic, it is equivalent to passing to invariants of that action, which is what we shall do. 

We now note that the map $\hat{\phi}_{\geo}$ is also defined on the level of Christoffel trees without any identifications, and that it commutes with the symmetric group $\Sigma_k$ acting on labels, so, due to the Maschke's theorem on complete reducibility of representations of $\Sigma_k$, we can choose the order in which we compute the kernel of $\hat{\phi}_{\geo}$ and pass to $\Sigma_k$-invariants. We shall choose to first compute the kernel and then pass to invariants. To compute the kernel, we shall use the previous result. We know that the homology of the differential graded operad
 \[
(\PL\vee\ComMag\vee \,\alpha, d_{\MC}+[\ell_1^{\alpha},-]+d_0)
 \]
is concentrated in homological degree zero and is isomorphic to the operad $\LA$ of Lie-admissible algebras. Considering, instead of the operad (all labels different) vertices of $k$ types amounts to evaluating our operad on the vector space $W:=\mathrm{Vect}(x_1,\ldots,x_k)$, forming the vector space 
 \[
\bigoplus_{n\ge 0} (\PL\vee\ComMag\vee \kk\alpha)(n)\otimes_{\kk \Sigma_n}W^{\otimes n}.
 \]
It follows from the K\"unneth formula that the homology is concentrated in homological degree zero and is given by 
 \[
\bigoplus_{n\ge 0} \LA(n)\otimes_{\kk \Sigma_n}W^{\otimes n}, 
 \] 
which is precisely $\LA(x_1,x_2,\ldots,x_k)$. Passing to invariants of the symmetric group action corresponds to imposing our equivalence relation on trees, and considering elements of degree two in each of the $k$ generators corresponds precisely to the space of Christoffel trees on which we seek to determine the kernel of $\hat{\phi}_{\geo}$.

Finally, we note that for the differential $\mathrm{d}=d_{\MC}+[\ell_1^{\alpha},-]+d_0$, we have
\begin{equs}
  \mathrm{d} \left( \begin{tikzpicture}[scale=0.2,baseline=-5]
    \coordinate (root) at (0,-2);
    \coordinate (center) at (0,2);
    \draw[] (root) -- (center);
    \node[var1] (rootnode) at (center) {$ 2 $};
    \node[var1] (rootnode) at (root) { $  1 $};
  \end{tikzpicture} \right) = \mathrm{d} \left(  \begin{tikzpicture}[scale=0.2,baseline=-5]
    \coordinate (root) at (0,-2);
    \coordinate (center) at (0,2);
    \draw[] (root) -- (center);
    \node[var1] (rootnode) at (center) {$ 1 $};
    \node[var1] (rootnode) at (root) { $  2 $};
  \end{tikzpicture} \right) =  \mathrm{d} \left( -\frac{1}{2} \begin{tikzpicture}[scale=0.2,baseline=-5]
    \coordinate (root) at (0,-2);
    \coordinate (right) at (2,2);
    \coordinate (left) at (-2,2);
    \draw[kernels2] (root) -- (left);
    \draw[kernels2] (root) -- (right);
    \node[var1] (rootnode) at (left) {$ 1 $};
    \node[var1] (rootnode) at (right) {$ 2 $};
    \node[not] (rootnode) at (root) {};
  \end{tikzpicture} \right) = \begin{tikzpicture}[scale=0.2,baseline=-5]
  \coordinate (root) at (0,-2);
  \coordinate (right) at (2,2);
  \coordinate (left) at (-2,2);
  \draw[] (root) -- (left);
  \draw[] (right) -- (root);
  \node[var1] (rootnode) at (left) {$ 1 $};
  \node[var1] (rootnode) at (right) {$ 2 $};
  \node[var1] (rootnode) at (root) {$ \alpha $};
\end{tikzpicture}.
\end{equs} 
so the element 
\begin{equs}
   \begin{tikzpicture}[scale=0.2,baseline=-5]
    \coordinate (root) at (0,-2);
    \coordinate (center) at (0,2);
    \draw[] (root) -- (center);
    \node[var1] (rootnode) at (center) {$ 2 $};
    \node[var1] (rootnode) at (root) { $  1 $};
  \end{tikzpicture} + \frac{1}{2} \begin{tikzpicture}[scale=0.2,baseline=-5]
    \coordinate (root) at (0,-2);
    \coordinate (right) at (2,2);
    \coordinate (left) at (-2,2);
    \draw[kernels2] (root) -- (left);
    \draw[kernels2] (root) -- (right);
    \node[var1] (rootnode) at (left) {$ 1 $};
    \node[var1] (rootnode) at (right) {$ 2 $};
    \node[not] (rootnode) at (root) {};
  \end{tikzpicture}  
\end{equs} 
is in the kernel of $\mathrm{d}$. Since we know that the homology operad is concentrated in degree zero and is generated by binary operations, this element is precisely the generator of $\LA$, so all elements of $\CS_{\tiny{\geo},2n}^g$  are obtained as linear combinations of iterations of covariant derivatives.
\end{proof}

The following corollary establishes Theorem \ref{main_theorme_chain_rule} for noises described by cumulants.

\begin{corollary}\label{cor:cumulant-dim}
For integers $n_2,\ldots,n_p\in\mathbb{N}$, let us denote by $U_{2,\ldots,2,\ldots,p,\ldots,p}$ the subspace of the free Lie-admissible algebra in $n_2+\cdots+n_p$ generators consisting of elements of degree two in each of the first $n_2$ generators, \ldots, of degree $n_p$ in each of the last $n_p$ generators. We have the natural vector space isomorphism 
 \[
\CS_{\tiny{\geo},n}^c\cong \bigoplus_{2n_2+\cdots+p n_p\leq n} U_{2,\ldots,2,\ldots,p,\ldots,p}^{\Sigma_{n_2}\times\cdots\times \Sigma_{n_p}}
 \]
where $U_{2,\ldots,2,\ldots,p,\ldots,p}^{\Sigma_{n_2}\times\cdots\times \Sigma_{n_p}}$ is the vector space of elements of $U_{2,\ldots,2,\ldots,p,\ldots,p}$ that are invariant under the action of the group $\Sigma_{n_2}\times\cdots\times \Sigma_{n_p}$ permuting the groups of generators separately. All elements of $\CS_{\tiny{\geo},n}^c$ are obtained as linear combinations of iterations of covariant derivatives. 
\end{corollary}

\begin{proof}
The argument is completely analogous to that of Corollary \ref{cor:gaussian-dim}: computing $\Sigma_{n_2}\times\cdots\times \Sigma_{n_p}$-invariants commutes with computing the homology, so we may use the result of Theorem \ref{th:twisting}. 
\end{proof}

In the following section, we shall explain how to compute dimensions of our vector spaces by appropriate manipulations with generating functions, which will complete the proof of Theorem \ref{main_theorme_chain_rule}.

\subsection{Dimension counting}

\label{dimension_counting_sec}

Let us explain how our results can be used to compute the dimensions of the vector spaces $\CS_{\tiny{\geo},2n}^g$ and $\CS_{\tiny{\geo},n}^c$. First, let us explain the passage from operads to free algebras. Let $\calO$ be an arbitrary operad, and consider the free algebra $\calO(x_1,\ldots,x_k)$ with $k$ generators. This free algebra carries a natural representation of the general linear group $GL_k$, and one may want to compute the character of that representation (which is a polynomial in variables $t_1,\ldots,t_k$ given by the trace of the diagonal matrix with entries $t_1,\ldots,t_k$). It turns out that this character can be easily computed from $F_{\calO}$ by using the power sum interpretation alluded to earlier, that is by substituting $p_r=\sum_{i=1}^k t_i^r$. For our purposes of computation of invariants such as $U_k^{\Sigma_k}$ or $U_{2,\ldots,2,\ldots,p,\ldots,p}^{\Sigma_{n_2}\times\cdots\times \Sigma_{n_p}}$, we shall specialise to appropriate subgroups of $GL_k$. 

Before any specialisation, let us note that since the operad $\LA$ is Koszul, one can compute $F_{\LA}$ by computing the compositional inverse of 
 \[
-F_{\LA^!}(-p_i\colon i\ge 1)=1-\exp\left(-\sum_i \frac{p_i}{i}\right)-1-\frac{p_1^2+p_2}{2}
 \]
(this form of the character of the quadratic dual is due to the explicit description of the quadratic dual mentioned earlier). That compositional inverse calculation is a built-in procedure of the symmetric functions implementation in \texttt{sage} \cite{sagemath}, so no human computational power is required here; however, a determined reader can verify that the first few terms of that compositional inverse in the elementary symmetric function basis are
 \[
e_1 + e_1^2 + 2e_1^3-e_3 + 5e_1^4-e_2e_1^2-2e_3e_1+e_4 +\cdots
 \]

Recall now that the vector space $U_k$ consists of elements of the free Lie-admissible algebra with $k$ generators of degree two in each of the $k$ generators. Keeping track of that kind of homogeneity is easy using diagonal matrices from $GL_k$ acting on generators: we are interested in elements of the free algebra that are multiplied by $t_1^2t_2^2\cdots t_k^2$ under this action. We also want to compute $\Sigma_k$-invariants, so we should keep track of the action of the symmetric group on generators. Overall, the subgroup of $GL_k$ that is of interest to us is the semidirect product $\Sigma_k\ltimes \mathrm{Diag}_k$, where $\mathrm{Diag}_k$ is the group of invertible diagonal matrices. For the matrix $\sigma a$, where $\sigma\in \Sigma_k$, $a=\mathrm{diag}(t_1,\ldots,t_k)\in \mathrm{Diag}_k$, the characteristic polynomial $\det(X \sigma a-\mathrm{id})$ of its action on generators is easily seen to be equal to 
 \[
\prod_{c} (X^k-t_{i_1}\cdots t_{i_k}),
 \]
where $c=(i_1,\ldots,i_k)$ is a cycle of $\sigma$. Since 
 \[
\det(X \sigma a-\mathrm{id})=\sum_{i=0}^k(-1)^{k-i}X^ie_i(\sigma a),
 \]
this means that to compute the character of $\sigma$ on the homogeneous component of degree $2$ in each generator, one has to consider the evaluation of $F_{\LA}$ at $e_i=e_i(\sigma a)$, and take the coefficient of $(t_1t_2\cdots t_k)^2$. After that, computing the scalar product with the character of the trivial representation of $\Sigma_k$ gives the dimension of the space of invariants.

Let us demonstrate how to use this recipe to compute $\dim\CS_{\tiny{\geo},8}^g$. This means that we should look at $k=1,2,3,4$. We obtain the following results.
\begin{itemize}
\item[$k=1$] In this case, we can only have $\sigma=\mathrm{id}$, and $e_1=t_1$, so we obtain $t_1^2$ with coefficient $1$, so $\dim(U_1^{\Sigma_1})=1$. 
\item[$k=2$] In this case, we can have, up to conjugation, $\sigma=\mathrm{id}$ or $\sigma=(1, 2)$. In the first case, the characteristic polynomial is $(X-t_1)(X-t_2)$, so we have $e_1(\sigma a)=t_1+t_2$, $e_2(\sigma a)=t_1t_2$, and $e_i=0$ for $i>2$, therefore $5e_1^4-e_2e_1^2-2e_3e_1+e_4$ is evaluated to $5(t_1+t_2)^4-t_1t_2(t_1+t_2)^2$, which contains $t_1^2t_2^2$ with coefficient $28$. In the second case, the characteristic polynomial is $X^2-t_1t_2$, so we have $e_1(\sigma a)=0$, $e_2(\sigma a)=-t_1t_2$, and $e_i=0$ for $i>2$, therefore $5e_1^4-e_2e_1^2-2e_3e_1+e_4$ is evaluated to $0$. Computing the scalar product with the character of the trivial representation, we get $\frac12(28+0)=14$, so $\dim(U_2^{\Sigma_2})=14$. (At this point, we already recover the dimension $15$ of \cite{BGHZ} as the sum $\dim(U_1^{\Sigma_1})+\dim(U_2^{\Sigma_2})$.)
\item[$k=3$] In this case, the degree six component of $F_{\LA}$ is
 \[
40e_1^6-19e_2e_1^4+6e_2^2e_1^2-2e_2^3-17e_3e_1^3+10e_3e_2e_1-3e_3^2+5e_4e_1^2-e_4e_2-2e_5e_1+e_6,
 \]
which we can already truncate to 
 \[
40e_1^6-19e_2e_1^4+6e_2^2e_1^2-2e_2^3-17e_3e_1^3+10e_3e_2e_1-3e_3^2,
 \] 
since we work in the free Lie-admissible algebra on three generators. Furthermore, in the group $\Sigma_3$ we can have, up to conjugation, $\sigma=\mathrm{id}$, $\sigma=(1, 2)$, or $\sigma=(1, 2, 3)$. In the first case, the characteristic polynomial is $(X-t_1)(X-t_2)(X-t_3)$, so we have $e_1(\sigma a)=t_1+t_2+t_3$, $e_2(\sigma a)=t_1t_2+t_1t_3+t_2t_3$, $e_3(\sigma a)=t_1t_2t_3$ and $e_i=0$ for $i>3$; the evaluation of the element above contains $t_1^2t_2^2t_3^2$ with coefficient $2919$. In the second case, the characteristic polynomial is $(X^2-t_1t_2)(X-t_3)$, so we have $e_1(\sigma a)=t_3$, $e_2(\sigma a)=-t_1t_2$, $e_3(\sigma a)=-t_1t_2t_3$, and $e_i=0$ for $i>3$; the evaluation of the element above contains $t_1^2t_2^2t_3^2$ with coefficient $13$. Finally, in the third case, the characteristic polynomial is $X^3-t_1t_2t_3$, so we have $e_1(\sigma a)=0$, $e_2(\sigma a)=0$, $e_3(\sigma a)=t_1t_2t_3$, and $e_i=0$ for $i>3$; the evaluation of the element above obviously contains $t_1^2t_2^2t_3^2$ with coefficient $-3$.
Computing the scalar product with the character of the trivial representation, we get $\frac16(2919+3\cdot 13-2\cdot 3)=492$, so $\dim(U_3^{\Sigma_3})=492$.
\item[$k=4$] In this case, the degree eight component of $F_{\LA}$ is 
\begin{multline*}
380e_1^8-260e_2e_1^6+141e_2^2e_1^4-61e_2^3e_1^2+11e_2^4-176e_3e_1^5+89e_3e_2e_1^3-14e_3e_2^2e_1\\
-4e_3^2e_1^2+51e_4e_1^4-20e_4e_2e_1^2+4e_4e_2^2+5e_4e_3e_1-4e_4^2-18e_5e_1^3+5e_5e_2e_1\\+8e_6e_1^2-2e_6e_2-2e_7e_1+e_8,
\end{multline*}
which we can already truncate to 
\begin{multline*}
380e_1^8-260e_2e_1^6+141e_2^2e_1^4-61e_2^3e_1^2+11e_2^4-176e_3e_1^5+89e_3e_2e_1^3-14e_3e_2^2e_1\\
-4e_3^2e_1^2+51e_4e_1^4-20e_4e_2e_1^2+4e_4e_2^2+5e_4e_3e_1-4e_4^2,
\end{multline*}
since we work in the free Lie-admissible algebra on three generators. Furthermore, in the group $\Sigma_4$ we can have, up to conjugation, $\sigma=\mathrm{id}$, $\sigma=(1, 2)$, $\sigma=(1, 2, 3)$, $\sigma=(1, 2, 3,4)$, or $\sigma=(1, 2) (3, 4)$. In the first case, the characteristic polynomial is $(X-t_1)(X-t_2)(X-t_3)(X-t_4)$, so we have $e_1(\sigma a)=t_1+t_2+t_3+t_4$, $e_2(\sigma a)=t_1t_2+t_1t_3+t_1t_4+t_2t_3+t_2t_4+t_3t_4$, $e_3(\sigma a)=t_1t_2t_3+t_1t_2t_4+t_1t_3t_4+t_2t_3t_4$, $e_4(\sigma a)=t_1t_2t_3t_4$ and $e_i=0$ for $i>4$; the evaluation of the element above contains $t_1^2t_2^2t_3^2t_4^2$ with coefficient $698946$. In the second case, the characteristic polynomial is $(X^2-t_1t_2)(X-t_3)(X-t_4)$, so we have $e_1(\sigma a)=t_3+t_4$, $e_2(\sigma a)=t_3t_4-t_1t_2$, $e_3(\sigma a)=-t_1t_2(t_3+t_4)$, $e_4(\sigma a)=-t_1t_2t_3t_4$ and $e_i=0$ for $i>4$; the evaluation of the element above contains $t_1^2t_2^2t_3^2t_4^2$ with coefficient $974$. In the third case, the characteristic polynomial is $(X^3-t_1t_2t_3)(X-t_4)$, so we have $e_1(\sigma a)=t_4$, $e_2(\sigma a)=0$, $e_3(\sigma a)=t_1t_2t_3$, $e_4(\sigma a)=t_1t_2t_3t_4$, and $e_i=0$ for $i>4$; the evaluation of the element above obviously contains $t_1^2t_2^2t_3^2$ with coefficient $-3$. In the fourth case, the characteristic polynomial is $X^4-t_1t_2t_3t_4$, so we have $e_1(\sigma a)=0$, $e_2(\sigma a)=0$, $e_3(\sigma a)=0$, $e_4(\sigma a)=-t_1t_2t_3t_4$, and $e_i=0$ for $i>4$; the evaluation of the element above obviously contains $t_1^2t_2^2t_3^2$ with coefficient $-4$. Finally, in the last case, the characteristic polynomial is $(X^2-t_1t_2)(X^2-t_3t_4)$, so we have $e_1(\sigma a)=0$, $e_2(\sigma a)=-t_1t_2-t_3t_4$, $e_3(\sigma a)=0$ $e_3(\sigma a)=t_1t_2t_3t_4$, and $e_i=0$ for $i>4$; the evaluation of the element above obviously contains $t_1^2t_2^2t_3^2$ with coefficient $70$.
Computing the scalar product with the character of the trivial representation, we get $\frac1{24}(698946+974\cdot 6-3\cdot8-6\cdot4+70\cdot3)=29373$, so $\dim(U_4^{\Sigma_4})=29373$.
\end{itemize}
Finally, we get $\dim\CS_{\tiny{\geo},8}^g=1+14+492+29373=29880$. Note that dimensions of these spaces grow so fast that they are completely impossible to determine by hand without any systematic method.

Let us also perform one computation in the case of cumulants, namely the dimension of the vector space $\CS_{\tiny{\geo},7}^c$. We have 
\begin{multline*}
\CS_{\tiny{\geo},7}^c\cong \bigoplus_{2n_2+\cdots+p n_p\leq 7} U_{2,\ldots,2,\ldots,p,\ldots,p}^{\Sigma_{n_2}\times\cdots\times \Sigma_{n_p}}=
U_{2}^{\Sigma_1}\oplus U_{3}^{\Sigma_1}\oplus U_{4}^{\Sigma_1}\oplus U_{2,2}^{\Sigma_2}\oplus U_{5}^{\Sigma_1}\oplus U_{2,3}^{\Sigma_1\times \Sigma_1}\\
\oplus U_{6}^{\Sigma_1}\oplus U_{2,4}^{\Sigma_1\times \Sigma_1}\oplus U_{3,3}^{\Sigma_2}\oplus U_{2,2,2}^{\Sigma_3}
\oplus U_{7}^{\Sigma_1}\oplus U_{2,5}^{\Sigma_1\times \Sigma_1}\oplus U_{3,4}^{\Sigma_1\times \Sigma_1}\oplus U_{2,2,3}^{\Sigma_2\times \Sigma_1}.
\end{multline*}
The strategy is more or less the same as above, except for needing the subgroup $\Sigma_{n_2}\times\cdots\times \Sigma_{n_p}$ of $\Sigma_{n_2+\cdots+n_p}$, and not the whole group. Since $U_{2,2,\ldots,2}$ is the vector space denoted $U_k$ above, we already know $\dim U_{2}^{\Sigma_1}=1$, $\dim U_{2,2}^{\Sigma_2}=14$, and $\dim U_{2,2,2}^{\Sigma_3}=492$. It remains to compute the dimensions of vector spaces $U_{3}^{\Sigma_1}$, $U_{4}^{\Sigma_1}$, $U_{5}^{\Sigma_1}$, $U_{2,3}^{\Sigma_1\times \Sigma_1}$, $U_{6}^{\Sigma_1}$, $U_{2,4}^{\Sigma_1\times \Sigma_1}$, $U_{3,3}^{\Sigma_2}$, $U_{7}^{\Sigma_1}$, $U_{2,5}^{\Sigma_1\times \Sigma_1}$, $U_{3,4}^{\Sigma_1\times \Sigma_1}$, $U_{2,2,3}^{\Sigma_2\times \Sigma_1}$. 

First, we note that some of these vector spaces correspond to partitions in distinct parts, and in such cases the group $\Sigma_{n_2}\times\cdots\times \Sigma_{n_p}$ consists of the unit element only, so we should just compute the dimensions of the corresponding homogeneous components of the free Lie-admissible algebra. This can be alternatively done by looking at the coefficient of the symmetric function $m_\lambda$ in $F_{\LA}$ for the partition $\lambda$ we consider. This way, we get $\dim U_{3}^{\Sigma_1}=2$, $\dim U_{4}^{\Sigma_1}=5$,  $\dim U_{5}^{\Sigma_1}=14$, $\dim U_{6}^{\Sigma_1}=40$, $\dim U_{7}^{\Sigma_1}=122$,  $\dim U_{2,3}^{\Sigma_1\times \Sigma_1}=124$, $\dim U_{2,4}^{\Sigma_1\times \Sigma_1}=530$, $\dim U_{2,5}^{\Sigma_1\times \Sigma_1}=2226$, $\dim U_{3,4}^{\Sigma_1\times \Sigma_1}=3623$.  

It remains to compute $\dim U_{3,3}^{\Sigma_2}$ and $\dim U_{2,2,3}^{\Sigma_2\times \Sigma_1}$. For $\dim U_{3,3}^{\Sigma_2}$, we use the same method as before, considering separately  $\sigma=\mathrm{id}$ or $\sigma=(1, 2)$, but looking at the homogeneous components of degree six component of $F_{\LA}$, which we truncate by setting $e_i=0$ for $i>2$, getting
 \[
40e_1^6-19e_2e_1^4+6e_2^2e_1^2-2e_2^3.
 \]  
For $\sigma=\mathrm{id}$, we have $e_1(\sigma a)=t_1+t_2$ and $e_2(\sigma a)=t_1t_2$, and the corresponding evaluation contains $t_1^3t_2^3$ with coefficient~$696$. For $\sigma=(1, 2)$, we have $e_1(\sigma a)=0$ and $e_2(\sigma a)=-t_1t_2$, and the corresponding evaluation contains $t_1^3t_2^3$ with coefficient~$2$. Computing the scalar product with the character of the trivial representation, we get $\frac12(696+2)=349$, so $\dim U_{3,3}^{\Sigma_2}=349$. To compute $\dim U_{2,2,3}^{\Sigma_2\times \Sigma_1}$, we have to consider $\Sigma_2\subset \Sigma_3$, so the cases $\sigma=\mathrm{id}$ and $\sigma=(1, 2)$ from the case $k=3$ for Gaussian noises above. We work with the homogeneous component of degree seven of $F_{\LA}$, that is,
\begin{multline*}
122e_1^7-75e_2e_1^5+39e_2^2e_1^3-14e_2^3e_1-51e_3e_1^4+21e_3e_2e_1^2\\
-e_3e_2^2-2e_3^2e_1+15e_4e_1^3-3e_4e_2e_1-6e_5e_1^2+e_5e_2+3e_6e_1-e_7,
\end{multline*}
which we truncate by setting $e_i=0$ for $i>3$, getting 
 \[
122e_1^7-75e_2e_1^5+39e_2^2e_1^3-14e_2^3e_1-51e_3e_1^4+21e_3e_2e_1^2-e_3e_2^2-2e_3^2e_1.
 \]
For $\sigma=\mathrm{id}$, we have $e_1(\sigma a)=t_1+t_2+t_3$, $e_2(\sigma a)=t_1t_2+t_1t_3+t_2t_3$, $e_3(\sigma a)=t_1t_2t_3$, and the corresponding evaluation contains $t_1^2t_2^2t_3^3$ with coefficient~$20150$. For $\sigma=(1, 2)$, we have $e_1(\sigma a)=t_3$, $e_2(\sigma a)=-t_1t_2$, $e_3(\sigma a)=-t_1t_2t_3$, and the corresponding evaluation contains $t_1^2t_2^2t_3^3$ with coefficient~$58$. Computing the scalar product with the character of the trivial representation, we get $\frac12(20150+58)=10104$, so $\dim U_{2,2,3}^{\Sigma_2\times \Sigma_1}=10104$. Adding up all the results, we obtain 
 \[
\dim\CS_{\tiny{\geo},7}^c=
17646.
 \]

\section{Quasi-generalised KPZ equation}
\label{sec::4}
We want to consider a quasi-linear version of \eqref{e:genClass}
\begin{equ}[e:genClass quasi]
	\d_t u^\alpha  = a(u)\d_x^2 u^\alpha + \Gamma^\alpha_{\beta\gamma}(u)\,\d_x u^\beta\d_x u^\gamma
	+ K^\alpha_\beta(u)\,\d_x u^\beta
	+h^\alpha(u) + \sigma_i^\alpha(u)\, \xi_i\;.
\end{equ}
We start by looking at the equation when $ a(u) $ is replaced by a parameter $c>0$:
\begin{equ} [e:genClass quasi bis]
	\begin{aligned}
		\partial_t u^{\alpha} - c \partial_x^{2} u^{\alpha}&  =  \left( \Gamma^\alpha_{\beta\gamma}(u) - \partial_{\beta}  a(u) \right) \,\d_x u^\beta\d_x u^\gamma
		\\ &	+ K^\alpha_\beta(u)\,\d_x u^\beta
		+h^\alpha(u) + \sigma_i^\alpha(u)\, \xi_i\;,
	\end{aligned}
\end{equ}
and for $\varrho\in\mathrm{Moll}$, $\eps>0$, denote its renormalisation counter-term by
\begin{equs}
	\sum_{\tau \in \SS_4 } C_{\eps}^{c}(\tau) \frac{\Upsilon^{\alpha}_{ \tilde{\Gamma},\sigma}[\tau]}{S(\tau)}.
\end{equs}
Here, we have supposed that the $ \xi_i $ are independent space-time white noises. The constants $ C_{\eps}^{c}(\tau) $ are smooth functions of the parameter $ c $. The map $ \tilde{\Gamma} $ is equal to $ \Gamma - \partial_{\beta} a $. In the sequel, we will use the short hand notation $ \Upsilon^{\alpha}_{F}[\tau] $ instead of $ \Upsilon^{\alpha}_{ \tilde{\Gamma},\sigma}[\tau] $ where $ F $ denoted the right hand side of \eqref{e:genClass quasi bis}. The main theorem of \cite{BGN} states:
\begin{theorem}\label{thm:main_quasi_KPZ_4}
	Let $a\in\mathcal{C}^6$, $\Gamma^{\alpha}_{\beta \gamma},\sigma_i^\alpha\in \mathcal{C}^{5}$ such that $a$ takes values in $[\lambda,\lambda^{-1}]$ for some $\lambda>0$.
	Let $u^{\alpha}_0\in\CC^r(\mathbb{T})$ for some $r>0$. For every $ \varrho \in \mathrm{Moll} $, $\eps>0$,	
	the renormalised equation of \eqref{e:genClass quasi} is given by:
	\begin{equs}[eq:renorm nonlocal6_bis]
		\d_t u^\alpha_{\eps}  & = a(u_{\eps})\d_x^2 u^\alpha_{\eps} + \Gamma^\alpha_{\beta\gamma}(u_{\eps})\,\d_x u^\beta_{\eps}\d_x u^\gamma_{\eps}
		+ K^\alpha_\beta(u_{\eps})\,\d_x u^\beta_{\eps}
		+h^\alpha(u_{\eps}) + \sigma_i^\alpha(u_{\eps})\, \xi_i^{\eps}\; \\ & + \sum_{\tau \in {\SS}_{4} } C_{\eps}^{a(u_\eps)}(\tau) \frac{\Upsilon_{ F}[\tau](u_{\eps})}{ S(\tau)}\,.
	\end{equs}
	That is, the solution $u_\eps$ of the random PDEs \eqref{eq:renorm nonlocal6_bis} converges as $\eps\to 0$ in probability, locally in time, to a nontrivial limit $u$. Here, one has to choose the functions $C^c_\eps(\tau)$ in such a way that the equations \eqref{eq:renorm nonlocal6_bis} transform according to the chain rule under composition with diffeomorphisms.
\end{theorem}

In \cite{BGN}, the equation considered was for $d=m=1$ but all the results remain true for the general case, except the long time existence (see also Section \ref{open_problems}).

In fact, in order to obtain the correct counter-terms in the theorem above one has first to renormalise an implicit system that contains non-local terms in $ u_{\eps} $. One obtains the following terms
\begin{equs} 
	\sum_{\tau \in \hat{\SS}_{4} } C_{\eps}^{a(u_{\eps})}(\tau) \frac{\Upsilon_{ \hat{F}}[\tau]}{q S(\tau)}(u_{\varepsilon})
\end{equs}
where $ q^{\alpha}(u_{\eps}) = 1- \partial_{\alpha} a'(u_{\eps}) \partial_c u^{\alpha}_{\eps}$ and  $ \hat{\SS}_{4} $ is a bigger combinatorial set that contains $ \SS_{4} $. The map $ \hat{F} $ is the non-linearity of the implicit system. Then, the chain rule symmetry, the fact that we focus on terms generated by the covariant derivatives, gives
\begin{equs} \label{S_4_parameter}
	\sum_{\tau \in \hat{\SS}_{4} } C_{\eps}^{a(u_{\varepsilon})}(\tau) \frac{\Upsilon_{ \hat{F}}[\tau]}{qS(\tau)}(u_{\varepsilon}) = 	\sum_{\tau \in \SS_{4} } C_{\eps}^{a(u_{\varepsilon})}(\tau) \frac{\Upsilon_{ F}[\tau]}{S(\tau)}(u_{\varepsilon}).
\end{equs}
Let us explain how the set $ \hat{\SS}_{4} $ is obtained. First, it contains $ \SS_4 $. Then, one has extra terms coming from a new decoration called parameter decoration on the edges. This decoration takes values in   $ \mathbb{N} $ and it is bounded by some $ m $. These new trees are partially planar in the sense that the planar order of edges with different parameter decorations matters. 
For example, one has
\begin{equs}
	\begin{tikzpicture}[baseline=0cm,scale=0.4];
		\node[xi] at (1,1) (a) {};
		\node[xi] at (-1,1) (b) {};
		\draw[blue] (a) to (0,0);
		\draw[blue] (b) to (0,0);
		\node at (0.7,0.3) {\tiny $\ell$};
		\node at (-0.8,0.3) {\tiny $m$};
		\node[not] at (0,0) (c) {};
	\end{tikzpicture} \neq \begin{tikzpicture}[baseline=0cm,scale=0.4];
		\node[xi] at (1,1) (a) {};
		\node[xi] at (-1,1) (b) {};
		\draw[blue] (a) to (0,0);
		\draw[blue] (b) to (0,0);
		\node at (0.7,0.3) {\tiny $ m$};
		\node at (-0.8,0.3) {\tiny $\ell$};
		\node[not] at (0,0) (c) {};
	\end{tikzpicture}
\end{equs}
where one has the following
interpretation for the new decorated edges
\begin{equs}
	\begin{tikzpicture}[baseline=0cm,scale=0.4];
		\draw[blue] (0,1) to (0,0);
		\node at (-0.5,0.5) {\tiny $ m$};
	\end{tikzpicture}  \equiv \partial_c^m P * \cdot, \quad 	\begin{tikzpicture}[baseline=0cm,scale=0.4];
		\draw[kernels2] (0,1) to (0,0);
		\node at (-0.5,0.5) {\tiny $ m$};
	\end{tikzpicture}   \equiv \partial_c^m \partial_x P * \cdot.
\end{equs} 
Here  $P(c,\cdot) $ is the Green's function of the operator $ \partial_t - c \partial_x^2  $.
This new decorated trees could been as augmented decorated trees and one can use the notation $\hat\tau=\d^{i_1}\otimes\cdots\otimes\d^{i_{[\tau]}}\otimes\tau$ where the $\d^{i_1}$ correspond to the $c$ derivative on the associated edge $i_1$. We denote by $ \hat{\SS}_{\xi} $ the extension of $\SS_{\xi}$.

 To each of the symbol $ \hat{\tau} $, one associates a function $C_\eps(\tau)(\cdot)$ in $[\tau]$ variables (set of edges in $ \tau $).
The renormalisation is then given by
\begin{equs} \label{assumption_derivatives}
	 C^c_\eps(\hat\tau)=\d^{i_1}_{c_1}\cdots\d^{i_{[\tau]}}_{c_{[\tau]}}
	C_\eps(\tau)(c_1,\ldots,c_{[\tau]})|_{c_1=\cdots c_{[\tau]}=c}.
\end{equs}
The previous identity is a crucial assumption to make on the renormalisation constants. Indeed, one does not want to introduce too many degrees of freedom. One can interpret \eqref{assumption_derivatives} as the fact that the renormalisation constants for decorated trees without any parameter derivatives determine completly the renormalisation constants for decorated trees with parameter derivatives.
The covariant derivative on decorated trees is defined by:
\begin{equs} \label{covariant_c}
	\Nabla_{\tau_2} \tau_1 &=  c
	\begin{tikzpicture}[scale=0.2,baseline=2]
		\draw[symbols]  (-.5,2.5) -- (0,0) ;
		\draw[tinydots] (0,0)  -- (0,-1.3);
		\node[var] (root) at (0,-0.1) {\tiny{$ \tau_1 $ }};
		\node[var] (diff) at (-0.5,2.5) {\tiny{$ \tau_2 $ }};
		\node[blank] at (0.3,1.25) {\tiny{$0$}};
	\end{tikzpicture} + \frac{1}{2}\;
	\begin{tikzpicture}[scale=0.2,baseline=2]
		\coordinate (root) at (0,0);
		\coordinate (t1) at (-1,2);
		\coordinate (t2) at (1,2);
		\draw[tinydots] (root)  -- +(0,-0.8);
		\draw[kernels2] (t1) -- (root);
		\draw[kernels2] (t2) -- (root);
		\node[not] (rootnode) at (root) {};
		\node[var] (t1) at (t1) {\tiny{$ \tau_1 $}};
		\node[var] (t1) at (t2) {\tiny{$ \tau_2 $}};
		\node[blank] at (-1.2,0.6) {\tiny{$0$}};
		\node[blank] at (1.2,0.6) {\tiny{$0$}};
	\end{tikzpicture} 
\end{equs}
Here, the main difference with $\eqref{covariant_derivative_semi_linear}$ is the parameter $  c $. This is due to the fact that one can perform the following change of variable $ v = c^{-3/2}u(c^{-1}t,c^{-1/2}x) $ which allows to remove the parameter $c$ from the operator and then $ c $ multiplies the vector fields $ \sigma_i  $.
All the results for the chain rule remain valid for these new covariant derivatives namely Theorem~\ref{main_theorme_chain_rule} and the fact that the solution $ u_{\eps} $ converges to a limit and it is invariant under change of coordinates.
Now, we want to consider covariant derivatives that include  parameter derivatives. We first define
\begin{equs} \label{covariant_parameter}
	\Nabla_{\tau_2}^1 \tau_1 &=  \partial(c\cdot)
	\begin{tikzpicture}[scale=0.2,baseline=2]
		\draw[symbols]  (-.5,2.5) -- (0,0) ;
		\draw[tinydots] (0,0)  -- (0,-1.3);
		\node[var] (root) at (0,-0.1) {\tiny{$ \tau_1 $ }};
		\node[var] (diff) at (-0.5,2.5) {\tiny{$ \tau_2 $ }};
		\node[blank] at (0.3,1.25) {\tiny{$1$}};
	\end{tikzpicture} + \frac{1}{2}\; (\partial \cdot )  \left( 
	\begin{tikzpicture}[scale=0.2,baseline=2]
		\coordinate (root) at (0,0);
		\coordinate (t1) at (-1,2);
		\coordinate (t2) at (1,2);
		\draw[tinydots] (root)  -- +(0,-0.8);
		\draw[kernels2] (t1) -- (root);
		\draw[kernels2] (t2) -- (root);
		\node[not] (rootnode) at (root) {};
		\node[var] (t1) at (t1) {\tiny{$ \tau_1 $}};
		\node[var] (t1) at (t2) {\tiny{$ \tau_2 $}};
		\node[blank] at (-1.2,0.6) {\tiny{$1$}};
		\node[blank] at (1.2,0.6) {\tiny{$0$}};
	\end{tikzpicture}  + 
	\begin{tikzpicture}[scale=0.2,baseline=2]
		\coordinate (root) at (0,0);
		\coordinate (t1) at (-1,2);
		\coordinate (t2) at (1,2);
		\draw[tinydots] (root)  -- +(0,-0.8);
		\draw[kernels2] (t1) -- (root);
		\draw[kernels2] (t2) -- (root);
		\node[not] (rootnode) at (root) {};
		\node[var] (t1) at (t1) {\tiny{$ \tau_1 $}};
		\node[var] (t1) at (t2) {\tiny{$ \tau_2 $}};
		\node[blank] at (-1.2,0.6) {\tiny{$0$}};
		\node[blank] at (1.2,0.6) {\tiny{$1$}};
	\end{tikzpicture} \right)
\end{equs}
Here we have used short hand notation. In fact for a renormalisation constant
\begin{equs}
	C^c_{\varepsilon} = 	 C_\eps(c_1 \cdots , c_{[\tau_1]}, \bar{c}_1, \bar{c}_2 , \tilde{c}_1,\cdots , \tilde{c}_{[\tau_2]})|_{c_i= \bar{c}_j = \tilde{c}_{\ell} = c},
\end{equs}
one has
\begin{equs}
	\partial(c C^c_{\varepsilon}) = C^c_{\varepsilon} + \left(   (\partial_{\bar{c}_1} + \partial_{ \bar{c}_2})C_\eps(c_1 \cdots , c_{[\tau_1]}, \bar{c}_1, \bar{c}_2 , \tilde{c}_1,\cdots , \tilde{c}_{[\tau_2]}) \right)|_{c_i= \bar{c}_j = \tilde{c}_{\ell} = c}.
\end{equs} 
The covariant derivative in \eqref{covariant_parameter} could be seen as the derivative of \eqref{covariant_c}. 
More generally, one can set
\begin{equs} \label{covariant_derivative_high_order}
	\Nabla_{\tau_2}^m \tau_1 &=  \partial^m(c\cdot)
	\begin{tikzpicture}[scale=0.2,baseline=2]
		\draw[symbols]  (-.5,2.5) -- (0,0) ;
		\draw[tinydots] (0,0)  -- (0,-1.3);
		\node[var] (root) at (0,-0.1) {\tiny{$ \tau_1 $ }};
		\node[var] (diff) at (-0.5,2.5) {\tiny{$ \tau_2 $ }};
		\node[blank] at (0.45,1.25) {\tiny{$m$}};
	\end{tikzpicture} + \frac{1}{2}\; (\partial^m \cdot )  \left(  \sum_{k+\ell = m} \frac{1}{\ell!k!} 
	\begin{tikzpicture}[scale=0.2,baseline=2]
		\coordinate (root) at (0,0);
		\coordinate (t1) at (-1,2);
		\coordinate (t2) at (1,2);
		\draw[tinydots] (root)  -- +(0,-0.8);
		\draw[kernels2] (t1) -- (root);
		\draw[kernels2] (t2) -- (root);
		\node[not] (rootnode) at (root) {};
		\node[var] (t1) at (t1) {\tiny{$ \tau_1 $}};
		\node[var] (t1) at (t2) {\tiny{$ \tau_2 $}};
		\node[blank] at (-1.3,0.6) {\tiny{$k$}};
		\node[blank] at (1.3,0.6) {\tiny{$\ell$}};
	\end{tikzpicture}
	\right).
\end{equs} 
If we assume that the constants $ C^{c}_{\eps}(\tau) $ are chosen such that the counter-terms without parameter derivatives are generated by the  covariant derivative \eqref{covariant_c}, then the rest of the counter-terms with parameter derivative is generated by all the covariant derivatives given by \eqref{covariant_derivative_high_order}.
 We denote by $ \hat{\VV}_{\xi} $ the extension of $\VV_{\xi}$ that contains the itreated covariant derivatives with paremeter derivative.
\begin{proof}[of Theorem \ref{thm:main renormalisation_intro_quasi}]
We use the BPHZ renormalisation as described in \cite[Eq (3.8)]{GH19}. Then, the renormalisation constants $ C_{\eps}^c(\tau)$ satisfy \eqref{assumption_derivatives}. We perform the following decomposition 
\begin{equs}
	\sum_{\tau \in \hat{\SS}_{\xi} } C_{\eps}^{a(u_{\varepsilon})}(\tau) \frac{\Upsilon_{ \hat{F}}[\tau]}{q S(\tau)}(u_{\varepsilon}) = 
		\sum_{ v \in \hat{\VV}_{\xi} } C_{\eps}^{a(u_{\varepsilon})}(v) \frac{\Upsilon_{ \hat{F}}[v]}{q} (u_{\varepsilon})+ 	\sum_{\tau \in \hat{\SS}_{\xi} } \hat{C}_{\eps}^{a(u_{\varepsilon})}(\tau) \frac{\Upsilon_{ \hat{F}}[\tau]}{q}(u_{\varepsilon})
	\end{equs}
where the constants $ C_{\eps}^{a(u_{\varepsilon})}(v) $ and $ \hat{C}_{\eps}^{a(u_{\varepsilon})}(\tau) $ has been chosen in such a way that  one has
\begin{equs} \label{orthogonal}
		\sum_{\tau \in \SS_{\xi} } \hat{C}_{\eps}^{a(u_{\varepsilon})}(\tau) \tau \in \CS_{\geo,\xi}^{\bot}. 
\end{equs}
In fact, due to the condition \eqref{assumption_derivatives}, it is easy to see that by fixing \eqref{assumption_derivatives}, this implies a choice on $\hat{C}_{\eps}^{a(u_{\varepsilon})}(\tau)$ for every $ \tau \in \hat{\SS}_{\xi} $ that respects the condition \eqref{assumption_derivatives}. This choice gives also the existence of the constants  $C_{\eps}^{a(u_{\varepsilon})}(v)$ satisfying also  \eqref{assumption_derivatives}.
From \cite[Prop. 3.9]{BGHZ}, we know that the $  \hat{C}_{\eps}^{a(u_{\varepsilon})}(\tau)$ for $ \tau \in  \SS_{\xi} $ converge to a finite limit. The result is also true for trees $ \tau \in  \hat{\SS}_{\xi} \setminus  \SS_{\xi} $ as the renormalisation constants are smooth functions in the parameter. This corresponds to the parameter derivation of the integration by parts formulae found in \cite[Lem. 2.4]{Mate19}.

The last step is to observe that one has:
  \begin{equs}
  	\sum_{ v \in \hat{\VV}_{\xi} } C_{\eps}^{a(u_{\varepsilon})}(v) \frac{\Upsilon_{ \hat{F}}[v]}{q}(u_{\varepsilon}) = 	\sum_{ v \in \VV_{\xi} } C_{\eps}^{a(u_{\varepsilon})}(v) \Upsilon_{ F}[v](u_{\varepsilon}).
  	\end{equs}
  This is checked inductively on the basis elements of $ \hat{\VV}_{\xi}  $ by using \cite[Prop. 3.12]{BGN} that allows us to remove high order parameter derivative for any trees $\tau_1$ and $\tau_2$:
  \begin{equs}
  	\Upsilon_{F}[\nabla_{\tau_2}^m \tau_1](u_{\varepsilon}) = 0, \quad m > 1.
  \end{equs}
Then, one has from \cite[Thm. 3.13]{BGN} that
\begin{equs}
		\Upsilon_{\hat F} \left[ \nabla_{\tau_2} \tau_1 \right] + \Upsilon_{\hat F} \left[ \nabla_{\tau_2}^1 \tau_1 \right]= q\Upsilon_{F}\brsq{ \nabla_{\tau_2}\tau_1}
	\end{equs}
where $ \tau_1, \tau_2  $ do not contain any parameter derivative.
\end{proof}

\printbibliography

@book {MR4174393,
    AUTHOR = {Friz, Peter K. and Hairer, Martin},
     TITLE = {A course on rough paths},
    SERIES = {Universitext},
   EDITION = {Second},
      NOTE = {With an introduction to regularity structures},
 PUBLISHER = {Springer, Cham},
      YEAR = {[2020] \copyright 2020},
     PAGES = {xvi+346},
      ISBN = {978-3-030-41556-3; 978-3-030-41555-6},
   MRCLASS = {60Lxx (34F05 35R60 60Hxx 93E03)},
  MRNUMBER = {4174393},
MRREVIEWER = {Fabrice\ Baudoin},
       DOI = {10.1007/978-3-030-41556-3},
       URL = {https://doi.org/10.1007/978-3-030-41556-3},
}

@misc{BH21,
      title={A tourist's guide to regularity structures and singular stochastic PDEs}, 
      author={I. Bailleul and M. Hoshino},
      year={2021},
      eprint={2006.03524},
      archivePrefix={arXiv},
      primaryClass={math.AP}
}

@book {MR0354798,
    AUTHOR = {MacLane, Saunders},
     TITLE = {Categories for the working mathematician},
    SERIES = {Graduate Texts in Mathematics},
    VOLUME = {Vol. 5},
 PUBLISHER = {Springer-Verlag, New York-Berlin},
      YEAR = {1971},
     PAGES = {ix+262},
   MRCLASS = {18-02},
  MRNUMBER = {354798},
MRREVIEWER = {H.-B.\ Brinkmann},
}

@article {MR0237408,
    AUTHOR = {Artamonov, V. A.},
     TITLE = {Clones of multilinear operations and multiple operator
              algebras},
   JOURNAL = {Uspehi Mat. Nauk},
  FJOURNAL = {Akademija Nauk SSSR i Moskovskoe Matemati\v{c}eskoe
              Ob\v{s}\v{c}estvo. Uspehi Matemati\v{c}eskih Nauk},
    VOLUME = {24},
      YEAR = {1969},
    NUMBER = {1(145)},
     PAGES = {47--59},
      ISSN = {0042-1316},
   MRCLASS = {08.30},
  MRNUMBER = {237408},
MRREVIEWER = {P.\ M.\ Cohn},
}

@book {MR0420610,
    AUTHOR = {May, J. P.},
     TITLE = {The geometry of iterated loop spaces},
    SERIES = {Lecture Notes in Mathematics},
    VOLUME = {Vol. 271},
 PUBLISHER = {Springer-Verlag, Berlin-New York},
      YEAR = {1972},
     PAGES = {viii+175},
   MRCLASS = {55D35},
  MRNUMBER = {420610},
MRREVIEWER = {J.\ Stasheff},
}

@article {BGHZ,
    AUTHOR = {Bruned, Y. and Gabriel, F. and Hairer, M. and Zambotti, L.},
     TITLE = {Geometric stochastic heat equations},
   JOURNAL = {J. Amer. Math. Soc.},
  FJOURNAL = {Journal of the American Mathematical Society},
    VOLUME = {35},
      YEAR = {2022},
    NUMBER = {1},
     PAGES = {1--80},
      ISSN = {0894-0347,1088-6834},
   MRCLASS = {60H15},
  MRNUMBER = {4322389},
MRREVIEWER = {Sergey\ V.\ Lototsky},
       DOI = {10.1090/jams/977},
       URL = {https://doi.org/10.1090/jams/977},
}

@article {BCCH,
    AUTHOR = {Bruned, Y. and Chandra, A. and Chevyrev, I. and Hairer, M.},
     TITLE = {Renormalising {SPDE}s in regularity structures},
   JOURNAL = {J. Eur. Math. Soc. (JEMS)},
  FJOURNAL = {Journal of the European Mathematical Society (JEMS)},
    VOLUME = {23},
      YEAR = {2021},
    NUMBER = {3},
     PAGES = {869--947},
      ISSN = {1435-9855,1435-9863},
   MRCLASS = {60L30 (16T05 54C35)},
  MRNUMBER = {4210726},
MRREVIEWER = {Torstein\ K.\ Nilssen},
       DOI = {10.4171/jems/1025},
       URL = {https://doi.org/10.4171/jems/1025},
}

@article {BHZ,
    AUTHOR = {Bruned, Y. and Hairer, M. and Zambotti, L.},
     TITLE = {Algebraic renormalisation of regularity structures},
   JOURNAL = {Invent. Math.},
  FJOURNAL = {Inventiones Mathematicae},
    VOLUME = {215},
      YEAR = {2019},
    NUMBER = {3},
     PAGES = {1039--1156},
      ISSN = {0020-9910,1432-1297},
   MRCLASS = {35R60 (16T05 35A30 60H15 82C28)},
  MRNUMBER = {3935036},
MRREVIEWER = {Peter\ E.\ Kloeden},
       DOI = {10.1007/s00222-018-0841-x},
       URL = {https://doi.org/10.1007/s00222-018-0841-x},
}

@misc{BGN,
      title={Quasi-generalised {KPZ} equation}, 
      author={Y. Bruned and M. Gerencs\'{e}r and U. Nadeem},
      year={2024},
      eprint={2401.13620 },
      archivePrefix={arXiv},
      primaryClass={math.PR}
}

@article {Rough11,
    AUTHOR = {Hairer, M.},
     TITLE = {Rough stochastic {PDE}s},
   JOURNAL = {Comm. Pure Appl. Math.},
  FJOURNAL = {Communications on Pure and Applied Mathematics},
    VOLUME = {64},
      YEAR = {2011},
    NUMBER = {11},
     PAGES = {1547--1585},
      ISSN = {0010-3640,1097-0312},
   MRCLASS = {60H15 (60G17 60H07)},
  MRNUMBER = {2832168},
MRREVIEWER = {Antoine\ J.\ Lejay},
       DOI = {10.1002/cpa.20383},
       URL = {https://doi.org/10.1002/cpa.20383},
}

@article {F92,
    AUTHOR = {Funaki, T.},
     TITLE = {A stochastic partial differential equation with values in a manifold},
   JOURNAL = {J. Funct. Anal.},
  FJOURNAL = {Annales de l'Institut Henri Poincar\'{e} C. Analyse Non
              Lin\'{e}aire},
    VOLUME = {109},
      YEAR = {1992},
    NUMBER = {},
     PAGES = {257--288}
    }

@article {G20,
    AUTHOR = {Gerencs\'{e}r, M.},
     TITLE = {Nondivergence form quasilinear heat equations driven by
              space-time white noise},
   JOURNAL = {Ann. Inst. H. Poincar\'{e} C Anal. Non Lin\'{e}aire},
  FJOURNAL = {Annales de l'Institut Henri Poincar\'{e} C. Analyse Non
              Lin\'{e}aire},
    VOLUME = {37},
      YEAR = {2020},
    NUMBER = {3},
     PAGES = {663--682},
      ISSN = {0294-1449,1873-1430},
   MRCLASS = {60H15 (35K59 35R60 60L30)},
  MRNUMBER = {4093623},
       DOI = {10.1016/j.anihpc.2020.01.003},
       URL = {https://doi.org/10.1016/j.anihpc.2020.01.003},
}

@misc{BEFH24,
      title={Multi-indice B-series}, 
      author={Y. Bruned and K. Ebrahimi-Fard and Y. Hou},
      year={2024},
      eprint={2402.13971},
      archivePrefix={arXiv},
      primaryClass={math.NA}
}

@misc{hairer2016motion,
      title={The motion of a random string}, 
      author={M. Hairer},
      year={2016},
      eprint={1605.02192},
      archivePrefix={arXiv},
      primaryClass={math.PR},
note={ Proceedings
of the XVIII ICMP}
}

@article {GH19,
    AUTHOR = {Gerencs\'{e}r, M. and Hairer, M.},
     TITLE = {A solution theory for quasilinear singular {SPDE}s},
   JOURNAL = {Comm. Pure Appl. Math.},
  FJOURNAL = {Communications on Pure and Applied Mathematics},
    VOLUME = {72},
      YEAR = {2019},
    NUMBER = {9},
     PAGES = {1983--2005},
      ISSN = {0010-3640,1097-0312},
   MRCLASS = {60H17 (35K59 35R60 60L30)},
  MRNUMBER = {3987723},
MRREVIEWER = {D.\ Erhard},
       DOI = {10.1002/cpa.21816},
       URL = {https://doi.org/10.1002/cpa.21816},
}

@article {KPZ,
    AUTHOR = {Hairer, M.},
     TITLE = {Solving the {KPZ} equation},
   JOURNAL = {Ann. of Math. (2)},
  FJOURNAL = {Annals of Mathematics. Second Series},
    VOLUME = {178},
      YEAR = {2013},
    NUMBER = {2},
     PAGES = {559--664},
      ISSN = {0003-486X,1939-8980},
   MRCLASS = {35K59 (35B10 35B65 35R60 60G22 60H15 60K35)},
  MRNUMBER = {3071506},
MRREVIEWER = {Alp\ O.\ Eden},
       DOI = {10.4007/annals.2013.178.2.4},
       URL = {https://doi.org/10.4007/annals.2013.178.2.4},
}

@misc{CH,
      title={An analytic {BPHZ} theorem for regularity structures}, 
      author={A. Chandra and M. Hairer},
      year={2018},
      eprint={1612.08138},
      archivePrefix={arXiv},
      primaryClass={math.PR}
}

@article {Mate19,
    AUTHOR = {Gerencs\'{e}r, M\'{a}t\'{e}},
     TITLE = {Nondivergence form quasilinear heat equations driven by
              space-time white noise},
   JOURNAL = {Ann. Inst. H. Poincar\'{e} C Anal. Non Lin\'{e}aire},
  FJOURNAL = {Annales de l'Institut Henri Poincar\'{e} C. Analyse Non
              Lin\'{e}aire},
    VOLUME = {37},
      YEAR = {2020},
    NUMBER = {3},
     PAGES = {663--682},
      ISSN = {0294-1449,1873-1430},
   MRCLASS = {60H15 (35K59 35R60 60L30)},
  MRNUMBER = {4093623},
       DOI = {10.1016/j.anihpc.2020.01.003},
       URL = {https://doi.org/10.1016/j.anihpc.2020.01.003},
}

@article {reg,
    AUTHOR = {Hairer, M.},
     TITLE = {A theory of regularity structures},
   JOURNAL = {Invent. Math.},
  FJOURNAL = {Inventiones Mathematicae},
    VOLUME = {198},
      YEAR = {2014},
    NUMBER = {2},
     PAGES = {269--504},
      ISSN = {0020-9910,1432-1297},
   MRCLASS = {60H15 (35R60 60H40 81S20 82C28)},
  MRNUMBER = {3274562},
MRREVIEWER = {Dora\ Sele\v{s}i},
       DOI = {10.1007/s00222-014-0505-4},
       URL = {https://doi.org/10.1007/s00222-014-0505-4},
}

@article {MR2503978,
    AUTHOR = {Markl, Martin},
     TITLE = {Natural differential operators and graph complexes},
   JOURNAL = {Differential Geom. Appl.},
  FJOURNAL = {Differential Geometry and its Applications},
    VOLUME = {27},
      YEAR = {2009},
    NUMBER = {2},
     PAGES = {257--278},
      ISSN = {0926-2245,1872-6984},
   MRCLASS = {18G35 (18D50 58A32)},
  MRNUMBER = {2503978},
MRREVIEWER = {Emily\ Burgunder},
       DOI = {10.1016/j.difgeo.2008.10.008},
       URL = {https://doi.org/10.1016/j.difgeo.2008.10.008},
}

@article {MR2817591,
    AUTHOR = {Jany\v{s}ka, J. and Markl, M.},
     TITLE = {Combinatorial differential geometry and ideal
              {B}ianchi-{R}icci identities},
   JOURNAL = {Adv. Geom.},
  FJOURNAL = {Advances in Geometry},
    VOLUME = {11},
      YEAR = {2011},
    NUMBER = {3},
     PAGES = {509--540},
      ISSN = {1615-715X,1615-7168},
   MRCLASS = {58A32 (53C05)},
  MRNUMBER = {2817591},
MRREVIEWER = {Marco\ Modugno},
       DOI = {10.1515/ADVGEOM.2011.017},
       URL = {https://doi.org/10.1515/ADVGEOM.2011.017},
}

@misc{munthekaas2023lie,
      title={Lie Admissible Triple Algebras: The Connection Algebra of Symmetric Spaces}, 
      author={Hans Munthe-Kaas and Jonatan Stava},
      year={2023},
      eprint={2306.15582},
      archivePrefix={arXiv},
      primaryClass={math.DG}
}

@article {MR2032454,
    AUTHOR = {Goze, Michel and Remm, Elisabeth},
     TITLE = {Lie-admissible algebras and operads},
   JOURNAL = {J. Algebra},
  FJOURNAL = {Journal of Algebra},
    VOLUME = {273},
      YEAR = {2004},
    NUMBER = {1},
     PAGES = {129--152},
      ISSN = {0021-8693,1090-266X},
   MRCLASS = {17D25 (18D50)},
  MRNUMBER = {2032454},
MRREVIEWER = {Fran\c{c}ois\ Goichot},
       DOI = {10.1016/j.jalgebra.2003.10.015},
       URL = {https://doi.org/10.1016/j.jalgebra.2003.10.015},
}

@article {MR2225770,
    AUTHOR = {Markl, M. and Remm, E.},
     TITLE = {Algebras with one operation including {P}oisson and other
              {L}ie-admissible algebras},
   JOURNAL = {J. Algebra},
  FJOURNAL = {Journal of Algebra},
    VOLUME = {299},
      YEAR = {2006},
    NUMBER = {1},
     PAGES = {171--189},
      ISSN = {0021-8693,1090-266X},
   MRCLASS = {17D25 (17A99 18D50)},
  MRNUMBER = {2225770},
MRREVIEWER = {David\ Chataur},
       DOI = {10.1016/j.jalgebra.2005.09.018},
       URL = {https://doi.org/10.1016/j.jalgebra.2005.09.018},
}

@article {MR4057606,
    AUTHOR = {Munthe-Kaas, Hans Z. and Stern, Ari and Verdier, Olivier},
     TITLE = {Invariant connections, {L}ie algebra actions, and foundations
              of numerical integration on manifolds},
   JOURNAL = {SIAM J. Appl. Algebra Geom.},
  FJOURNAL = {SIAM Journal on Applied Algebra and Geometry},
    VOLUME = {4},
      YEAR = {2020},
    NUMBER = {1},
     PAGES = {49--68},
      ISSN = {2470-6566},
   MRCLASS = {53C05 (17A30 17B66 65D30)},
  MRNUMBER = {4057606},
MRREVIEWER = {Mihai\ Anastasiei},
       DOI = {10.1137/19M1252879},
       URL = {https://doi.org/10.1137/19M1252879},
}

@article {MR0027750,
    AUTHOR = {Albert, A. A.},
     TITLE = {Power-associative rings},
   JOURNAL = {Trans. Amer. Math. Soc.},
  FJOURNAL = {Transactions of the American Mathematical Society},
    VOLUME = {64},
      YEAR = {1948},
     PAGES = {552--593},
      ISSN = {0002-9947,1088-6850},
   MRCLASS = {09.1X},
  MRNUMBER = {27750},
MRREVIEWER = {D.\ Rees},
       DOI = {10.2307/1990399},
       URL = {https://doi.org/10.2307/1990399},
}

@book {MR4621635,
    AUTHOR = {Dotsenko, Vladimir and Shadrin, Sergey and Vallette, Bruno},
     TITLE = {Maurer-{C}artan methods in deformation theory---the twisting
              procedure},
    SERIES = {London Mathematical Society Lecture Note Series},
    VOLUME = {488},
 PUBLISHER = {Cambridge University Press, Cambridge},
      YEAR = {2024},
     PAGES = {viii+177},
      ISBN = {978-1-108-96564-4},
   MRCLASS = {55P48 (18M60 53C05)},
  MRNUMBER = {4621635},
}

@misc{dotsenko2021homotopical,
      title={Homotopical rigidity of the pre-Lie operad}, 
      author={Vladimir Dotsenko and Anton Khoroshkin},
      year={2020},
      eprint={2002.12918},
      archivePrefix={arXiv},
      primaryClass={math.KT}
}

@misc{laubie2024hypertrees,
      title={Hypertrees and embedding of the $\mathrm{FMan}$ operad}, 
      author={Paul Laubie},
      year={2024},
      eprint={2401.17439},
      archivePrefix={arXiv},
      primaryClass={math.QA}
}

@misc{laubie2023combinatorics,
      title={Combinatorics of pre-Lie products sharing a Lie bracket}, 
      author={Paul Laubie},
      year={2023},
      eprint={2309.05552},
      archivePrefix={arXiv},
      primaryClass={math.QA}
}

@article {MR1827084,
    AUTHOR = {Chapoton, Fr\'{e}d\'{e}ric and Livernet, Muriel},
     TITLE = {Pre-{L}ie algebras and the rooted trees operad},
   JOURNAL = {Internat. Math. Res. Notices},
  FJOURNAL = {International Mathematics Research Notices},
      YEAR = {2001},
    NUMBER = {8},
     PAGES = {395--408},
      ISSN = {1073-7928,1687-0247},
   MRCLASS = {17A30 (17B60 18D50 55P48 81T15)},
  MRNUMBER = {1827084},
MRREVIEWER = {Olga\ Kravchenko},
       DOI = {10.1155/S1073792801000198},
       URL = {https://doi.org/10.1155/S1073792801000198},
}

@article {MR0633783,
    AUTHOR = {Joyal, Andr\'{e}},
     TITLE = {Une th\'{e}orie combinatoire des s\'{e}ries formelles},
   JOURNAL = {Adv. in Math.},
  FJOURNAL = {Advances in Mathematics},
    VOLUME = {42},
      YEAR = {1981},
    NUMBER = {1},
     PAGES = {1--82},
      ISSN = {0001-8708},
   MRCLASS = {05A99 (18B99 26B10)},
  MRNUMBER = {633783},
       DOI = {10.1016/0001-8708(81)90052-9},
       URL = {https://doi.org/10.1016/0001-8708(81)90052-9},
}

@book {MR1629341,
    AUTHOR = {Bergeron, F. and Labelle, G. and Leroux, P.},
     TITLE = {Combinatorial species and tree-like structures},
    SERIES = {Encyclopedia of Mathematics and its Applications},
    VOLUME = {67},
      NOTE = {Translated from the 1994 French original by Margaret Readdy,
              With a foreword by Gian-Carlo Rota},
 PUBLISHER = {Cambridge University Press, Cambridge},
      YEAR = {1998},
     PAGES = {xx+457},
      ISBN = {0-521-57323-8},
   MRCLASS = {05A15 (05C30 05E05)},
  MRNUMBER = {1629341},
MRREVIEWER = {Ira\ Gessel},
}

@book {MR2954392,
    AUTHOR = {Loday, Jean-Louis and Vallette, Bruno},
     TITLE = {Algebraic operads},
    SERIES = {Grundlehren der mathematischen Wissenschaften [Fundamental
              Principles of Mathematical Sciences]},
    VOLUME = {346},
 PUBLISHER = {Springer, Heidelberg},
      YEAR = {2012},
     PAGES = {xxiv+634},
      ISBN = {978-3-642-30361-6},
   MRCLASS = {18D50 (16E99)},
  MRNUMBER = {2954392},
MRREVIEWER = {Andrey\ Yu.\ Lazarev},
       DOI = {10.1007/978-3-642-30362-3},
       URL = {https://doi.org/10.1007/978-3-642-30362-3},
}

@article {MR3348138,
    AUTHOR = {Willwacher, Thomas},
     TITLE = {M. {K}ontsevich's graph complex and the
              {G}rothendieck-{T}eichm\"{u}ller {L}ie algebra},
   JOURNAL = {Invent. Math.},
  FJOURNAL = {Inventiones Mathematicae},
    VOLUME = {200},
      YEAR = {2015},
    NUMBER = {3},
     PAGES = {671--760},
      ISSN = {0020-9910},
   MRCLASS = {17B55 (18D50 53D55)},
  MRNUMBER = {3348138},
MRREVIEWER = {Rongmin Lu},
       DOI = {10.1007/s00222-014-0528-x},
       URL = {https://doi.org/10.1007/s00222-014-0528-x},
}

@article {MR3299688,
    AUTHOR = {Dolgushev, Vasily and Willwacher, Thomas},
     TITLE = {Operadic twisting---with an application to {D}eligne's
              conjecture},
   JOURNAL = {J. Pure Appl. Algebra},
  FJOURNAL = {Journal of Pure and Applied Algebra},
    VOLUME = {219},
      YEAR = {2015},
    NUMBER = {5},
     PAGES = {1349--1428},
      ISSN = {0022-4049},
   MRCLASS = {18D50 (17B55 18C15)},
  MRNUMBER = {3299688},
MRREVIEWER = {Bruno Vallette},
       DOI = {10.1016/j.jpaa.2014.06.010},
       URL = {https://doi.org/10.1016/j.jpaa.2014.06.010},
}

@book {MR3443860,
    AUTHOR = {Macdonald, I. G.},
     TITLE = {Symmetric functions and {H}all polynomials},
    SERIES = {Oxford Classic Texts in the Physical Sciences},
   EDITION = {Second},
   EDITION = {paperback},
      NOTE = {With contribution by A. V. Zelevinsky and a foreword by
              Richard Stanley},
 PUBLISHER = {The Clarendon Press, Oxford University Press, New York},
      YEAR = {2015},
     PAGES = {xii+475},
      ISBN = {978-0-19-873912-8},
   MRCLASS = {05E05 (01A75 05-02 20C30 20C33 20K01 33C80 33D80)},
  MRNUMBER = {3443860},
}

@article {MR2034546,
    AUTHOR = {Markl, Martin},
     TITLE = {Homotopy algebras are homotopy algebras},
   JOURNAL = {Forum Math.},
  FJOURNAL = {Forum Mathematicum},
    VOLUME = {16},
      YEAR = {2004},
    NUMBER = {1},
     PAGES = {129--160},
      ISSN = {0933-7741,1435-5337},
   MRCLASS = {18G55 (18D50 55U15 55U35)},
  MRNUMBER = {2034546},
MRREVIEWER = {Stanis\l aw\ Betley},
       DOI = {10.1515/form.2004.002},
       URL = {https://doi.org/10.1515/form.2004.002},
}

@book {MR0874337,
    AUTHOR = {Fuks, D. B.},
     TITLE = {Cohomology of infinite-dimensional {L}ie algebras},
    SERIES = {Contemporary Soviet Mathematics},
      NOTE = {Translated from the Russian by A. B. Sosinski\u{\i}},
 PUBLISHER = {Consultants Bureau, New York},
      YEAR = {1986},
     PAGES = {xii+339},
      ISBN = {0-306-10990-5},
   MRCLASS = {17-02 (17B56 17B65)},
  MRNUMBER = {874337},
}

@incollection {MR0611158,
    AUTHOR = {Kirillov, A. A.},
     TITLE = {Invariant operators over geometric quantities},
 BOOKTITLE = {Current problems in mathematics, {V}ol. 16 ({R}ussian)},
    SERIES = {Itogi Nauki i Tekhniki},
     PAGES = {3--29, 228},
 PUBLISHER = {Akad. Nauk SSSR, Vsesoyuz. Inst. Nauchn. i Tekhn. Inform.,
              Moscow},
      YEAR = {1980},
   MRCLASS = {58A99 (22E45 81C99)},
  MRNUMBER = {611158},
MRREVIEWER = {A.\ Verona},
}

@article {MR4537770,
    AUTHOR = {Linares, P. and Otto, F. and Tempelmayr, M.},
     TITLE = {The structure group for quasi-linear equations via universal
              enveloping algebras},
   JOURNAL = {Comm. Amer. Math. Soc.},
  FJOURNAL = {Communications of the American Mathematical Society},
    VOLUME = {3},
      YEAR = {2023},
     PAGES = {1--64},
      ISSN = {2692-3688},
   MRCLASS = {60L30 (16S30 16T05 17B35 35K59 60L70)},
  MRNUMBER = {4537770},
       DOI = {10.1090/cams/16},
       URL = {https://doi.org/10.1090/cams/16},
}

@article {MR3451427,
    AUTHOR = {Munthe-Kaas, H. and Verdier, O.},
     TITLE = {Aromatic {B}utcher series},
   JOURNAL = {Found. Comput. Math.},
  FJOURNAL = {Foundations of Computational Mathematics. The Journal of the
              Society for the Foundations of Computational Mathematics},
    VOLUME = {16},
      YEAR = {2016},
    NUMBER = {1},
     PAGES = {183--215},
      ISSN = {1615-3375,1615-3383},
   MRCLASS = {65L06 (37M99 41A58 53C21)},
  MRNUMBER = {3451427},
MRREVIEWER = {Jing\ Bo\ Chen},
       DOI = {10.1007/s10208-015-9245-0},
       URL = {https://doi.org/10.1007/s10208-015-9245-0},
}

@misc{BL23,
      title={A top-down approach to algebraic renormalization in regularity structures based on multi-indices}, 
      author={Y. Bruned and P. Linares},
      year={2023},
      eprint={2307.03036},
      archivePrefix={arXiv},
      primaryClass={math.PR}
}

@misc{OSSW,
      title={A priori bounds for quasi-linear SPDEs in the full sub-critical regime}, 
      author={Otto, F. and Sauer, J. and Smith, S. and Weber, H.},
      year={2023},
      eprint={2103.11039},
      archivePrefix={arXiv},
      primaryClass={math.AP}
}

@misc{LOTT,
      title={A diagram-free approach to the stochastic estimates in regularity structures}, 
      author={Linares, P., and Otto, Felix and Tempelmayr, M. and Tsatsoulis, P.},
      year={2022},
      eprint={2112.10739},
      archivePrefix={arXiv},
      primaryClass={math.PR}
}

@misc{BD23,
      title={Novikov algebras and multi-indices in regularity structures}, 
      author={Yvain Bruned and Vladimir Dotsenko},
      year={2023},
      eprint={2311.09091},
      archivePrefix={arXiv},
      primaryClass={math.RA}
}

@incollection{zbMATH02573963,
     Author = {Veblen, O.},
     Title = {Differential invariants and geometry.},
     Year = {1929},
     BookTitle = {Atti {Congresso} {Bologna}},
     Volume = {1},
     Pages = {181-189},
     zbMATH = {2573963},
     JFM = {55.1026.05}
}

@book {MR0066025,
    AUTHOR = {Schouten, J. A.},
     TITLE = {Ricci-calculus. {A}n introduction to tensor analysis and its
              geometrical applications},
    SERIES = {Die Grundlehren der mathematischen Wissenschaften in
              Einzeldarstellungen mit besonderer Ber\"{u}cksichtigung der
              Anwendungsgebiete},
    VOLUME = {Band X},
      NOTE = {2d. ed},
 PUBLISHER = {Springer-Verlag, Berlin-G\"{o}ttingen-Heidelberg},
      YEAR = {1954},
     PAGES = {xx+516},
   MRCLASS = {53.0X},
  MRNUMBER = {66025},
MRREVIEWER = {A.\ G.\ Walker},
}

@article {MR2081725,
    AUTHOR = {Dzhumadil'daev, A. S.},
     TITLE = {{$N$}-commutators},
   JOURNAL = {Comment. Math. Helv.},
  FJOURNAL = {Commentarii Mathematici Helvetici},
    VOLUME = {79},
      YEAR = {2004},
    NUMBER = {3},
     PAGES = {516--553},
      ISSN = {0010-2571,1420-8946},
   MRCLASS = {17A42 (17B66)},
  MRNUMBER = {2081725},
MRREVIEWER = {Janusz\ Grabowski},
       DOI = {10.1007/s00014-004-0807-2},
       URL = {https://doi.org/10.1007/s00014-004-0807-2},
}

@article {MR2383581,
    AUTHOR = {Dzhumadil'daev, Askar},
     TITLE = {10-commutators, 13-commutators and odd derivations},
   JOURNAL = {J. Nonlinear Math. Phys.},
  FJOURNAL = {Journal of Nonlinear Mathematical Physics},
    VOLUME = {15},
      YEAR = {2008},
    NUMBER = {1},
     PAGES = {87--103},
      ISSN = {1402-9251,1776-0852},
   MRCLASS = {17B66},
  MRNUMBER = {2383581},
       DOI = {10.2991/jnmp.2008.15.1.7},
       URL = {https://doi.org/10.2991/jnmp.2008.15.1.7},
}

@manual{sagemath,
  Key          = {SageMath},
  Author       = {{The Sage Developers}},
  Title        = {{S}ageMath, the {S}age {M}athematics {S}oftware {S}ystem ({V}ersion 9.5)},
  note         = {{\tt https://www.sagemath.org}},
  Year         = {2022},
}

@book {MR1269324,
    AUTHOR = {Weibel, Charles A.},
     TITLE = {An introduction to homological algebra},
    SERIES = {Cambridge Studies in Advanced Mathematics},
    VOLUME = {38},
 PUBLISHER = {Cambridge University Press, Cambridge},
      YEAR = {1994},
     PAGES = {xiv+450},
      ISBN = {0-521-43500-5; 0-521-55987-1},
   MRCLASS = {18-01 (16-01 17-01 20-01 55Uxx)},
  MRNUMBER = {1269324},
MRREVIEWER = {Kenneth\ A.\ Brown},
       DOI = {10.1017/CBO9781139644136},
       URL = {https://doi.org/10.1017/CBO9781139644136},
}

\end{document}